\documentclass{amsart}
\usepackage{accents}
\usepackage{amssymb}
\usepackage{amsmath}
\usepackage{amsfonts}
\usepackage{amscd}
\usepackage{appendix}
\usepackage{url}
\usepackage{graphicx}
\usepackage{hyperref}

\newtheorem{theorem}{Theorem}[section]
\theoremstyle{plain}

\newtheorem{corollary}[theorem]{Corollary}

\newtheorem{definition}[theorem]{Definition}
\newtheorem{example}{Example}

\newtheorem{lemma}[theorem]{Lemma}

\newtheorem{proposition}[theorem]{Proposition}
\newtheorem{remark}[theorem]{Remark}

\numberwithin{equation}{section}

\begin{document}
\title[$T$-Weyl calculus]{$T$-Weyl calculus}
\author{Gruia Arsu}
\address{Institute of Mathematics of The Romanian Academy\\
P.O. Box 1-174\\
RO-70700\\
Bucharest \\
Romania}
\email{Gruia.Arsu@imar.ro, agmilro@yahoo.com, agmilro@hotmail.com}

\begin{abstract}
Let $\left( W,\sigma \right) $ be a symplectic vector space and let $%
T:W\rightarrow W$ be a linear map that satisfies a certain condition of
non-degeneracy. We define the Schur multiplier $\omega _{\sigma ,T}$ on $W$.
To this multiplier we associate a $\omega _{\sigma ,T}$-representation and
and we build the $T$-Weyl calculus, $\mathrm{Op}_{\sigma ,T}$, whose
properties are are systematically studied further.
\end{abstract}

\maketitle
\tableofcontents

\section{Introduction}

In two classical papers \cite{Cordes} and \cite{Kato}, H.O. Cordes and T.
Kato develop an elegant method to deal with pseudo-differential operators.
In \cite{Cordes}, Cordes shows, among others, that if a symbol $a\left(
x,\xi \right) $\ defined on $\mathbb{R}^{n}\times \mathbb{R}^{n}$ has
bounded derivatives $D_{x}^{\alpha }D_{\xi }^{\beta }a$\ for $\left\vert
\alpha \right\vert ,\left\vert \beta \right\vert \leq $\ $\left[ n/2\right]
+1$, then the associated pseudo-differential operator\ $A=a\left( X,D\right) 
$ is $L^{2}$-bounded.

This method can be extended and can be used to prove trace-class properties
of pseudo-differential operators. For example, by using this method it can
be proved that $A\in \mathcal{B}_{p}\left( L^{2}\left( \mathbb{R}^{n}\right)
\right) $ if $D_{x}^{\alpha }D_{\xi }^{\beta }a$ is in $L^{p}\left( \mathbb{R%
}^{n}\times \mathbb{R}^{n}\right) $ for $\left\vert \alpha \right\vert
,\left\vert \beta \right\vert \leq $\ $\left[ n/2\right] +1$ and $1\leq
p<\infty $, where $\mathcal{B}_{p}\left( L^{2}\left( \mathbb{R}^{n}\right)
\right) $ denote the Schatten ideal of compact operators whose singular
values lie in $l^{p}$. It is remarkable that this method can be used for $%
\left( \theta ,\tau \right) $-quantization, in particular, for both the Weyl
quantization and Kohn-Nirenberg quantization.

\textsc{The }$\left( \theta ,\tau \right) $-\textsc{quantization}

In this section we shall briefly describe the $\left( \theta ,\tau \right) $%
-quantization. The definitions, the results and the notions introduced will
serve both as models and as examples for the $T$-Weyl calculus developed in
the following sections.

Let $V$ be an $n$ dimensional vector space over $\mathbb{R}$ and $V^{\ast }$
its dual. We denote by $\mathrm{d}x$ a Lebesgue measure in $V$ and by $%
\mathrm{d}p$ is the dual one in $V^{\ast }$ such that Fourier's inversion
formula holds with the usual constant. Replacing $\mathrm{d}x$ by $c\mathrm{d%
}x$ one must change $\mathrm{d}p$ to $c^{-1}\mathrm{d}p$ so $\mathrm{d}x%
\mathrm{d}p$ is invariantly defined.

Let $\theta ,\tau \in \mathrm{End}_{\mathbb{R}}\left( V\right) $. For $a\in 
\mathcal{S}\left( V\times V^{\ast }\right) $ and $v\in \mathcal{S}\left(
V\right) $ we define 
\begin{eqnarray*}
\mathrm{Op}_{\theta ,\tau }\left( a\right) v\left( x\right) &=&a_{\theta
,\tau }\left( X,D\right) v\left( x\right) \\
&=&\left( 2\pi \right) ^{-n}\iint \mathrm{e}^{\mathrm{i}\left\langle
x-y,p\right\rangle }a\left( \theta x+\tau y,p\right) v\left( y\right) 
\mathrm{d}y\mathrm{d}p.
\end{eqnarray*}%
If $u,v\in \mathcal{S}\left( V\right) $, then%
\begin{align*}
\left\langle \mathrm{Op}_{\theta ,\tau }\left( a\right) v,u\right\rangle _{%
\mathcal{S}^{\prime },\mathcal{S}}& =\left( 2\pi \right) ^{-n}\iiint \mathrm{%
e}^{\mathrm{i}\left\langle x-y,p\right\rangle }a\left( \theta x+\tau
y,p\right) u\left( x\right) v\left( y\right) \mathrm{d}x\mathrm{d}y\mathrm{d}%
p \\
& =\left\langle \left[ \left( \left( 1\otimes \mathcal{F}_{V}^{-1}\right)
a\right) \circ \mathrm{c}_{\theta ,\tau }\right] ,u\otimes v\right\rangle _{%
\mathcal{S}^{\prime },\mathcal{S}},
\end{align*}%
\begin{equation*}
\mathrm{c}_{\theta ,\tau }:V\times V\rightarrow V\times V,\quad \mathrm{c}%
_{\theta ,\tau }\left( x,y\right) =\left( \theta x+\tau y,x-y\right) .
\end{equation*}%
If $\theta +\tau :V\rightarrow V$ is a linear isomorphism, it is possible to
define $\mathrm{Op}_{\theta ,\tau }\left( a\right) $ as an operator in $%
\mathcal{B}\left( \mathcal{S}\left( V\right) ,\mathcal{S}^{\prime }\left(
V\right) \right) $ for any $a\in \mathcal{S}^{\prime }\left( V\times V^{\ast
}\right) $%
\begin{equation*}
\left\langle \mathrm{Op}_{\theta ,\tau }\left( a\right) v,u\right\rangle _{%
\mathcal{S}^{\prime }\left( V\right) ,\mathcal{S}\left( V\right)
}=\left\langle \mathcal{K}_{\mathrm{Op}_{\theta ,\tau }\left( a\right)
},u\otimes v\right\rangle _{\mathcal{S}^{\prime }\left( V\times V\right) ,%
\mathcal{S}\left( V\times V\right) },
\end{equation*}%
\begin{equation*}
\mathcal{K}_{\mathrm{Op}_{\theta ,\tau }\left( a\right) }=\left( \left(
1\otimes \mathcal{F}_{V}^{-1}\right) a\right) \circ \mathrm{c}_{\theta ,\tau
}.
\end{equation*}%
Note that $\det \mathrm{c}_{\theta ,\tau }=\left( -1\right) ^{n}\det \left(
\theta +\tau \right) $ and that $\theta +\tau :V\rightarrow V$ is a linear
isomorphism iff $\mathrm{c}_{\theta ,\tau }:V\times V\rightarrow V\times V$
is a linear isomorphism.

\begin{definition}
$(\mathrm{a})$ We denote by $\Omega \left( V\right) $ the set%
\begin{equation*}
\Omega \left( V\right) =\left\{ \left( \theta ,\tau \right) :\theta ,\tau
\in \mathrm{End}_{\mathbb{R}}\left( V\right) ,\text{ }\theta +\tau \text{
isomorphism}\right\} .
\end{equation*}%
which is symmetric i.e. $\left( \theta ,\tau \right) \in \Omega \left(
V\right) $ $\Leftrightarrow $ $\left( \tau ,\theta \right) \in \Omega \left(
V\right) $.

$(\mathrm{b})$ For $\left( \theta ,\tau \right) \in \Omega \left( V\right) $%
, the map 
\begin{equation*}
\mathrm{Op}_{\theta ,\tau }:\mathcal{S}^{\prime }\left( V\times V^{\ast
}\right) \rightarrow \mathcal{B}(\mathcal{S}\left( V\right) ,\mathcal{S}%
^{\prime }\left( V\right) ),a\rightarrow \mathrm{Op}_{\theta ,\tau
}(a)\equiv a_{\theta ,\tau }\left( X,D\right) ,
\end{equation*}%
is called $\left( \theta ,\tau \right) $-quantization. The distribution $%
a\in \mathcal{S}^{\prime }\left( V\times V^{\ast }\right) $ is called $%
\left( \theta ,\tau \right) $-symbol of $\mathrm{Op}_{\theta ,\tau }(a)$.

$(\mathrm{c})$ For $\tau \in \mathrm{End}_{\mathbb{R}}\left( V\right) $, the
map $\mathrm{Op}_{\tau }=\mathrm{Op}_{1-\tau ,\tau }$ is called the $\tau $%
-quantization. In particular, for $\tau =\frac{1}{2}$ we have the Weyl
quantization $\mathrm{Op}=\mathrm{Op}_{\frac{1}{2}}$ and for $\tau =0$ we
have the Kohn-Nirenberg quantization.
\end{definition}

Since the equation in $a\in \mathcal{S}^{\prime }\left( V\times V^{\ast
}\right) $, $\left( \left( \mathrm{id}\otimes \mathcal{F}_{X}^{-1}\right)
a\right) \circ \mathrm{c}_{\theta ,\tau }=\mathcal{K}$, has a unique
solution for each $\mathcal{K}\in \mathcal{S}^{\prime }\left( V\times
V\right) $, a consequence of the kernel theorem is the fact that the map 
\begin{equation*}
\mathrm{Op}_{\theta ,\tau }:\mathcal{S}^{\prime }\left( V\times V^{\ast
}\right) \rightarrow \mathcal{B}(\mathcal{S}\left( V\right) ,\mathcal{S}%
^{\prime }\left( V\right) ),a\rightarrow \mathrm{Op}_{\theta ,\tau
}(a)\equiv a_{\theta ,\tau }\left( X,D\right)
\end{equation*}%
is linear, continuous and bijective. Hence for each $A\in \mathcal{B}(%
\mathcal{S},\mathcal{S}^{\prime })$ there is a distribution $a\in \mathcal{S}%
^{\prime }\left( V\times V^{\ast }\right) $ such that $A=\mathrm{Op}_{\theta
,\tau }(a)$. This distribution is called $\left( \theta ,\tau \right) $%
-symbol of $A$.

\pagebreak

\textsc{The projective representation}

We shall use the isomorphism%
\begin{equation*}
V\times V^{\ast }\ni \left( y,p\right) \rightarrow L_{\left( y,p\right) }\in
\left( V\times V^{\ast }\right) ^{\ast },
\end{equation*}%
\begin{equation*}
L_{\left( y,p\right) }=\left\langle \cdot ,p\right\rangle _{V,V^{\ast
}}\otimes 1-1\otimes \left\langle y,\cdot \right\rangle _{V,V^{\ast
}}=\sigma \left( \left( y,p\right) ;\left( \cdot ,\cdot \right) \right) ,
\end{equation*}%
where $\sigma :\left( V\times V^{\ast }\right) \times \left( V\times V^{\ast
}\right) \rightarrow \mathbb{R}$ is the canonical symplectic form 
\begin{equation*}
\sigma \left( \left( x,k\right) ;\left( y,p\right) \right) =\left\langle
y,k\right\rangle _{V,V^{\ast }}-\left\langle x,p\right\rangle _{V,V^{\ast }}.
\end{equation*}%
For each $\left( \theta ,\tau \right) \in \Omega \left( V\right) $ we
consider the family $\left\{ \mathcal{W}_{\theta ,\tau }\left( z,\zeta
\right) \right\} _{\left( z,\zeta \right) \in V\times V^{\ast }}$, 
\begin{equation*}
\mathcal{W}_{\theta ,\tau }\left( y,p\right) =\left( \mathrm{e}^{\mathrm{i}%
L_{\left( y,p\right) }}\right) _{\theta ,\tau }\left( X,D\right) \in 
\mathcal{B}\left( \mathcal{S},\mathcal{S}^{\prime }\right) .
\end{equation*}%
We have 
\begin{eqnarray*}
\mathcal{W}_{\theta ,\tau }\left( y,p\right) &=&\mathrm{e}^{\mathrm{i}%
\left\langle \theta X,p\right\rangle _{V,V^{\ast }}}\mathrm{e}^{-\mathrm{i}%
\left\langle y,D\right\rangle _{V,V^{\ast }}}\mathrm{e}^{\mathrm{i}%
\left\langle \tau X,p\right\rangle _{V,V^{\ast }}} \\
&=&\mathrm{e}^{-\mathrm{i}\left\langle \tau y,p\right\rangle _{V,V^{\ast }}}%
\mathrm{e}^{\mathrm{i}\left\langle \left( \theta +\tau \right)
X,p\right\rangle _{V,V^{\ast }}}\mathrm{e}^{-\mathrm{i}\left\langle
y,D\right\rangle _{V,V^{\ast }}} \\
&=&\mathrm{e}^{\mathrm{i}\left\langle \theta y,p\right\rangle _{V,V^{\ast }}}%
\mathrm{e}^{-\mathrm{i}\left\langle y,D\right\rangle _{V,V^{\ast }}}\mathrm{e%
}^{\mathrm{i}\left\langle \left( \theta +\tau \right) X,p\right\rangle
_{V,V^{\ast }}}.
\end{eqnarray*}%
By using the canonical commutation relations we obtain that the map 
\begin{equation*}
\mathcal{W}_{\theta ,\tau }:V\times V^{\ast }\rightarrow \mathcal{U}\left(
L^{2}\left( V\right) \right)
\end{equation*}%
satisfies 
\begin{equation*}
\mathcal{W}_{\theta ,\tau }\left( x,k\right) \mathcal{W}_{\theta ,\tau
}\left( y,p\right) =\mathrm{e}^{\mathrm{i}\sigma \left( \left( x,k\right)
;\left( \tau y,\theta ^{\ast }p\right) \right) }\mathcal{W}_{\theta ,\tau
}\left( x+y,k+p\right) ,
\end{equation*}%
for every $\left( x,\xi \right) $, $\left( y,\eta \right) \in V\times
V^{\ast }$.

Let $\omega _{\theta ,\tau }$ be the function 
\begin{equation*}
\omega _{\theta ,\tau }:\left( V\times V^{\ast }\right) \times \left(
V\times V^{\ast }\right) \rightarrow \mathbb{T},
\end{equation*}%
\begin{equation*}
\omega _{\theta ,\tau }\left( \left( x,k\right) ;\left( y,p\right) \right) =%
\mathrm{e}^{\mathrm{i}\sigma \left( \left( x,k\right) ;\left( \tau
y,{}\theta ^{\ast }p\right) \right) }.
\end{equation*}%
It follows that $\omega _{\theta ,\tau }$ satisfies the cocycle equation 
\begin{multline*}
\omega _{\theta ,\tau }\left( \left( x,k\right) ;\left( y,p\right) \right)
\omega _{\theta ,\tau }\left( \left( x,k\right) +\left( y,p\right) ;\left(
z,q\right) \right) \\
=\omega _{\theta ,\tau }\left( \left( x,k\right) ;\left( y,p\right) +\left(
z,q\right) \right) \omega _{\theta ,\tau }\left( \left( y,p\right) ;\left(
z,q\right) \right) ,
\end{multline*}%
\begin{equation*}
\omega _{\theta ,\tau }\left( \left( x,k\right) ;\left( 0,0\right) \right)
=\omega _{\theta ,\tau }\left( \left( 0,0\right) ;\left( x,k\right) \right)
=1,
\end{equation*}%
hence $\omega _{\theta ,\tau }$ is a 2-cocycle or Schur multiplier.

Moreover, the Schur multiplier $\omega _{\theta ,\tau }$ is non-degenerate,
that is 
\begin{eqnarray*}
\omega _{\theta ,\tau }\left( \left( x,k\right) ;\left( y,p\right) \right)
&=&\omega _{\theta ,\tau }\left( \left( y,p\right) ;\left( x,k\right)
\right) ,\ \forall \left( y,p\right) \in V\times V^{\ast }\  \\
&\Rightarrow &\left( x,k\right) \ =0.
\end{eqnarray*}

\begin{proposition}
Let $\left( \theta ,\tau \right) \in \Omega \left( V\right) $. Then

$\left( \mathrm{a}\right) $ $\omega _{\theta ,\tau }$ is a non-degenerate $2$%
-cocycle or Schur multiplier.

$\left( \mathrm{b}\right) $ The couple $\left( L^{2}\left( V\right) ,%
\mathcal{W}_{\theta ,\tau }\right) $ is a projective representation of $%
V\times V^{\ast }$ with $\omega _{\theta ,\tau }$ the associated multiplier.

$\left( \mathrm{c}\right) $ The projective representation $\left(
L^{2}\left( V\right) ,\mathcal{W}_{\theta ,\tau }\right) $ is irreducible.

$\left( \mathrm{d}\right) $ For $u,v\in \mathcal{S}\left( V\right) $, the
function 
\begin{equation*}
w_{\theta ,\tau ,u,v}:V\times V^{\ast }\rightarrow \mathbb{R},\quad
w_{\theta ,\tau ,u,v}\left( y,p\right) =\left\langle \mathcal{W}_{\theta
,\tau }\left( y,p\right) v,u\right\rangle _{\mathcal{S}^{\prime }\left(
V\right) ,\mathcal{S}\left( V\right) },
\end{equation*}%
is in $\mathcal{S}\left( V\times V^{\ast }\right) $.
\end{proposition}

For $a\in \mathcal{S}\left( V\times V^{\ast }\right) $ we define the
symplectic Fourier transform%
\begin{equation*}
\mathcal{F}_{\sigma }\left( a\right) \left( x,k\right) =\left( 2\pi \right)
^{-n}\iint \mathrm{e}^{-\mathrm{i}\sigma \left( \left( x,k\right) ;\left(
y,p\right) \right) }a\left( y,p\right) \mathrm{d}y\mathrm{d}p,\quad \left(
x,k\right) \in V\times V^{\ast }
\end{equation*}%
and we use the same notation $\mathcal{F}_{\sigma }\mathcal{\ }$for the
extension to $\mathcal{S}^{\prime }\left( V\times V^{\ast }\right) $ of this
Fourier transform. It follows that $\mathcal{F}_{\sigma }$ is involutive
(i.e. $\mathcal{F}_{\mathfrak{\sigma }}^{2}=\boldsymbol{1}$) and unitary on $%
L^{2}(V\times V^{\ast })$.

The family $\left\{ \mathcal{W}_{\theta ,\tau }\left( x,k\right) \right\}
_{\left( x,k\right) \in V\times V^{\ast }}$ completely characterizes $\left(
\theta ,\tau \right) $-Weyl calculus or $\left( \theta ,\tau \right) $%
-quantization. The general definition of $\mathrm{Op}_{\theta ,\tau }\left(
a\right) $ is deduced from this family and the symplectic Fourier
decomposition of $a$.

\begin{proposition}
Let $\left( \theta ,\tau \right) \in \Omega \left( V\right) $ and $a\in 
\mathcal{S}\left( V\times V^{\ast }\right) $. Then%
\begin{equation*}
\mathrm{Op}_{\theta ,\tau }\left( a\right) =\left( 2\pi \right) ^{-n}\iint 
\mathcal{F}_{\sigma }\left( a\right) \left( y,p\right) \mathcal{W}_{\theta
,\tau }\left( y,p\right) \mathrm{d}y\mathrm{d}p.
\end{equation*}%
The integral is taken in the weak sense, i.e. for $u,v\in \mathcal{S}\left(
V\right) $%
\begin{align*}
& \left\langle \left( \left( 2\pi \right) ^{-n}\iint \mathcal{F}_{\sigma
}\left( a\right) \left( y,p\right) \mathcal{W}_{\theta ,\tau }\left(
y,p\right) \mathrm{d}y\mathrm{d}p\right) v,u\right\rangle _{\mathcal{S}%
^{\prime }\left( V\right) ,\mathcal{S}\left( V\right) } \\
& =\left\langle \mathcal{F}_{\sigma }\left( a\right) ,\left\langle \mathcal{W%
}_{\theta ,\tau }\left( \cdot ,\cdot \right) v,u\right\rangle \right\rangle
_{\mathcal{S}^{\prime }\left( V\times V^{\ast }\right) ,\mathcal{S}\left(
V\times V^{\ast }\right) } \\
& =\left\langle \mathcal{F}_{\sigma }\left( a\right) ,w_{\theta ,\tau
,u,v}\right\rangle _{\mathcal{S}^{\prime }\left( V\times V^{\ast }\right) ,%
\mathcal{S}\left( V\times V^{\ast }\right) }
\end{align*}%
where $w_{\theta ,\tau ,u,v}\in \mathcal{S}\left( V\times V^{\ast }\right) $
is given by 
\begin{equation*}
w_{\theta ,\tau ,u,v}\left( y,p\right) =\left\langle \mathcal{W}_{\theta
,\tau }\left( y,p\right) v,u\right\rangle _{\mathcal{S}^{\prime }\left(
V\right) ,\mathcal{S}\left( V\right) }.
\end{equation*}
\end{proposition}

\begin{proof}
It is enough to prove equality for $a\in \mathcal{S}\left( V\times V^{\ast
}\right) $. Let $u$, $v\in \mathcal{S}\left( V\right) $. Then 
\begin{multline*}
\left( 2\pi \right) ^{-n}\left\langle \mathcal{F}_{\sigma }\left( a\right)
,w_{\theta ,\tau ,u,v}\right\rangle =\left( 2\pi \right) ^{-n}\iint \mathcal{%
F}_{\sigma }\left( a\right) \left( y,p\right) w_{\theta ,\tau ,u,v}\left(
y,p\right) \mathrm{d}y\mathrm{d}p \\
=\left( 2\pi \right) ^{-n}\iint \mathcal{F}_{\sigma }\left( a\right) \left(
y,p\right) \int \left( \mathcal{W}_{\theta ,\tau }\left( y,p\right) v\right)
\left( x\right) u\left( x\right) \mathrm{d}x\mathrm{d}y\mathrm{d}p \\
=\int \left( \left( 2\pi \right) ^{-n}\iint \mathcal{F}_{\sigma }\left(
a\right) \left( y,p\right) \left( \mathcal{W}_{\theta ,\tau }\left(
y,p\right) v\right) \left( x\right) \mathrm{d}y\mathrm{d}p\right) u\left(
x\right) \mathrm{d}x.
\end{multline*}%
Since%
\begin{eqnarray*}
\left( \mathcal{W}_{\theta ,\tau }\left( y,p\right) v\right) \left( x\right)
&=&\mathrm{e}^{-\mathrm{i}\left\langle \tau y,p\right\rangle _{V,V^{\ast }}}%
\mathrm{e}^{\mathrm{i}\left\langle \left( \theta +\tau \right)
x,p\right\rangle _{V,V^{\ast }}}v\left( x-y\right)  \\
&=&\mathrm{e}^{\mathrm{i}\left\langle \theta x+\tau \left( x-y\right)
,p\right\rangle _{V,V^{\ast }}}v\left( x-y\right) ,
\end{eqnarray*}%
we get that%
\begin{multline*}
\left( 2\pi \right) ^{-n}\iint \mathcal{F}_{\sigma }\left( a\right) \left(
y,p\right) \left( \mathcal{W}_{\theta ,\tau }\left( y,p\right) v\right)
\left( x\right) \mathrm{d}y\mathrm{d}p \\
=\left( 2\pi \right) ^{-n}\iint \mathcal{F}_{\sigma }\left( a\right) \left(
y,p\right) \mathrm{e}^{\mathrm{i}\left\langle \theta x+\tau \left(
x-y\right) ,p\right\rangle _{V,V^{\ast }}}v\left( x-y\right) \mathrm{d}y%
\mathrm{d}p \\
=\left( 2\pi \right) ^{-n}\iint \mathrm{e}^{\mathrm{i}\left\langle \theta
x+\tau \left( x-y\right) ,p\right\rangle _{V,V^{\ast }}}v\left( x-y\right)
\left( \mathcal{F}_{V}\otimes \mathcal{F}_{V}^{-1}\right) a\left( p,y\right) 
\mathrm{d}y\mathrm{d}p \\
=\int v\left( x-y\right) \left( \mathcal{F}_{V}^{-1}\otimes 1\right) \left(
\left( \mathcal{F}_{V}\otimes \mathcal{F}_{V}^{-1}\right) a\right) \left(
\theta x+\tau \left( x-y\right) ,y\right) \mathrm{d}y \\
=\int v\left( x-y\right) \left( \left( 1\otimes \mathcal{F}_{V}^{-1}\right)
a\right) \left( \theta x+\tau \left( x-y\right) ,y\right) \mathrm{d}y \\
\underset{x-y\rightsquigarrow y}{=}\int \left( \left( 1\otimes \mathcal{F}%
_{V}^{-1}\right) a\right) \left( \theta x+\tau y,x-y\right) v\left( y\right) 
\mathrm{d}y \\
=\int \mathcal{K}_{\mathrm{Op}_{\theta ,\tau }\left( a\right) }\left(
x,y\right) v\left( y\right) \mathrm{d}y=\mathrm{Op}_{\theta ,\tau }\left(
a\right) v\left( x\right) =a_{\theta ,\tau }\left( X,D\right) v\left(
x\right) .
\end{multline*}%
Therefore 
\begin{equation*}
\left( 2\pi \right) ^{-n}\left\langle \mathcal{F}_{\sigma }\left( a\right)
,w_{\theta ,\tau ,u,v}\right\rangle =\int a_{\theta ,\tau }\left( X,D\right)
v\left( x\right) u\left( x\right) \mathrm{d}x=\left\langle a_{\theta ,\tau
}\left( X,D\right) v,u\right\rangle .
\end{equation*}
\end{proof}

The map $t\rightarrow \mathcal{W}_{\theta ,\tau }\left( ty,tp\right) $ is
not a group representation of $\mathbb{R}$. Instead, if we replace the
family $\left\{ \mathcal{W}_{\theta ,\tau }\left( y,p\right) \right\}
_{\left( y,p\right) \in V\times V^{\ast }}$ with the family $\left\{ 
\widetilde{\mathcal{W}}_{\theta ,\tau }\left( y,p\right) \right\} _{\left(
y,p\right) \in V\times V^{\ast }}$,%
\begin{eqnarray*}
\widetilde{\mathcal{W}}_{\theta ,\tau }\left( y,p\right) &=&\mathrm{e}^{%
\frac{\mathrm{i}}{2}\sigma \left( \left( y,p\right) ;\left( \tau y,\theta
^{\ast }p\right) \right) }\mathcal{W}_{\theta ,\tau }\left( y,p\right) \\
&=&\mathrm{e}^{\frac{\mathrm{i}}{2}\left\langle \left( \tau -\theta \right)
y,p\right\rangle }\mathcal{W}_{\theta ,\tau }\left( y,p\right) \\
&=&\mathrm{e}^{\frac{\mathrm{i}}{2}\left\langle \left( \tau +\theta \right)
y,p\right\rangle }\mathrm{e}^{-\mathrm{i}\left\langle y,D\right\rangle
_{V,V^{\ast }}}\mathrm{e}^{\mathrm{i}\left\langle \left( \theta +\tau
\right) X,p\right\rangle _{V,V^{\ast }}},\left( y,p\right) \in V\times
V^{\ast },
\end{eqnarray*}%
then the couple $(L^{2}\left( V\right) ,\widetilde{\mathcal{W}}_{\theta
,\tau })$ is a projective representation of $V\times V^{\ast }$ with $%
\widetilde{\omega }_{\theta ,\tau }$ the associated Schur multiplier. Here%
\begin{eqnarray*}
\widetilde{\omega }_{\theta ,\tau }\left( \left( x,k\right) ;\left(
y,p\right) \right) &=&\frac{\mathrm{e}^{\frac{\mathrm{i}}{2}\sigma \left(
\left( x,k\right) ;\left( \tau x,\theta ^{\ast }k\right) \right) }\mathrm{e}%
^{\frac{\mathrm{i}}{2}\sigma \left( \left( y,p\right) ;\left( \tau y,\theta
^{\ast }p\right) \right) }}{\mathrm{e}^{\frac{\mathrm{i}}{2}\sigma \left(
\left( x+y,k+p\right) ;\left( \tau \left( x+y\right) ,\theta ^{\ast }\left(
k+p\right) \right) \right) }}\mathrm{e}^{\mathrm{i}\sigma \left( \left(
x,k\right) ;\left( \tau y,{}\theta ^{\ast }p\right) \right) } \\
&=&\mathrm{e}^{\frac{\mathrm{i}}{2}\left( \sigma \left( \left( x,k\right)
;\left( \tau y,\theta ^{\ast }p\right) \right) -\sigma \left( \left(
y,p\right) ;\left( \tau x,\theta ^{\ast }k\right) \right) \right) } \\
&=&\mathrm{e}^{\frac{\mathrm{i}}{2}\sigma \left( \left( \left( \theta +\tau
\right) x,k\right) ;\left( \left( \theta +\tau \right) y,p\right) \right) }=%
\mathrm{e}^{\frac{\mathrm{i}}{2}\left( \left\langle \left( \theta +\tau
\right) y,k\right\rangle -\left\langle \left( \theta +\tau \right)
x,p\right\rangle \right) }.
\end{eqnarray*}

The functions $\omega _{\theta ,\tau }$ and $\widetilde{\omega }_{\theta
,\tau }$ are cohomologous and $\widetilde{\omega }_{\theta ,\tau }$ is
normalized, i.e. 
\begin{equation*}
\widetilde{\omega }_{\theta ,\tau }\left( \left( x,k\right) ;-\left(
x,k\right) \right) =1,\quad \left( x,k\right) \in V\times V^{\ast }.
\end{equation*}

The map $t\rightarrow \widetilde{\mathcal{W}}_{\theta ,\tau }\left(
ty,tp\right) $ is a group representation of $\mathbb{R}$ and for each $%
\left( y,p\right) \in V\times V^{\ast }$ there is a unique self-adjoint
operator $\phi _{\theta ,\tau }\left( y,p\right) $, $\left( \theta ,\tau
\right) $-field operator associated to $\left( y,p\right) $, such that 
\begin{equation*}
\widetilde{\mathcal{W}}_{\theta ,\tau }\left( ty,tp\right) =\mathrm{e}^{%
\mathrm{i}t\phi _{\theta ,\tau }\left( y,p\right) }
\end{equation*}%
for all real $t$. The map $\left( y,p\right) \rightarrow \phi _{\theta ,\tau
}\left( y,p\right) $ $\mathbb{R}$-linear. The $\left( \theta ,\tau \right) $%
-annihilation and $\left( \theta ,\tau \right) $-creation operators
associated to $\left( y,p\right) $ are defined by%
\begin{eqnarray*}
a_{\theta ,\tau }\left( y,p\right) &=&\frac{1}{2}\left( \phi _{\theta ,\tau
}\left( y,p\right) +\mathrm{i}\phi _{\theta ,\tau }\left( \mathrm{i}\left(
y,p\right) \right) \right) , \\
a_{\theta ,\tau }^{\ast }\left( y,p\right) &=&\frac{1}{2}\left( \phi
_{\theta ,\tau }\left( y,p\right) -\mathrm{i}\phi _{\theta ,\tau }\left( 
\mathrm{i}\left( y,p\right) \right) \right) ,
\end{eqnarray*}%
e.t.c..

We notice that the formula%
\begin{equation*}
\omega _{\theta ,\tau }\left( \left( x,k\right) ;\left( y,p\right) \right) =%
\mathrm{e}^{\mathrm{i}\sigma \left( \left( x,k\right) ;\left( \tau
y,{}\theta ^{\ast }p\right) \right) }.
\end{equation*}%
highlights the symplectic structure on $V\times V^{\ast }$ and a linear map $%
\tau \times \theta ^{\ast }:V\times V^{\ast }\rightarrow V\times V^{\ast }$
which satisfies a certain condition contained in the definition of the set $%
\Omega \left( V\right) $. This condition is equivalent to the condition that
the Schur factor $\omega _{\theta ,\tau }$ is non-degenerate. Note also that
the representation 
\begin{equation*}
\mathrm{Op}_{\theta ,\tau }\left( a\right) =\left( 2\pi \right) ^{-n}\iint 
\mathcal{F}_{\sigma }\left( a\right) \left( y,p\right) \mathcal{W}_{\theta
,\tau }\left( y,p\right) \mathrm{d}y\mathrm{d}p,
\end{equation*}%
indicates the contribution of the symplectic structure of $V\times V^{\ast }$
in the definition of the $\left( \theta ,\tau \right) $-quantization, the
operator $\mathrm{Op}_{\theta ,\tau }$.

The $\left( \theta ,\tau \right) $-calculus built up in this section can be
further generalized, and we shall do this in this paper. The paper is
organized as follows. In Section 2 we summarize the most important notations
and results from linear symplectic algebra. In Section 3 we define the $%
\omega _{\sigma ,T}$-representation and the associated Weyl system and
present some of their properties. In Section 4 we define the $T$-Weyl
calculus and we prove one of the important results of the paper, Theorem \ref%
{n4}, which is an important technical result that establishes the connection
between the $T$-Weyl calculus and the standard Weyl calculus. In Section 5
we study modulation spaces and Schatten-class properties of operators in the 
$T$-Weyl calculus. The results in this section on Schatten-class properties
of operators in the $T$-Weyl calculus, together with the results in Section
6 are used to prove an extension of Cordes' lemma in Section 7. Sections
7-10 are devoted to the Cordes-Kato method for $T$-Weyl calculus.

As can be seen, we started from a natural definition for a
pseudo-differential calculus, and we obtained a projective representation.
In this paper, we shall follow the path in the opposite direction, namely
using symplectic $2$-form $\sigma $ we shall associate to any linear map $T$
on $W$ a $2$-cocycle or Schur multiplier $\omega _{\sigma ,T}$. If the Schur
multiplier $\omega _{\sigma ,T}$ is non-degenerate, which may be expressed
by a non-degeneracy condition of $T$, then any two irreducible $\omega
_{\sigma ,T}$-representations are unitary equivalent. For an irreducible $%
\omega _{\sigma ,T}$-representation $\left( \mathcal{H},\mathcal{W}_{\sigma
,T},\omega _{\sigma ,T}\right) $ of $W$ there is a well defined linear,
continuous and bijective map, the $T$-Weyl calculus, 
\begin{equation*}
\mathrm{Op}_{\sigma ,T}:\mathcal{S}^{\ast }\left( W\right) \rightarrow 
\mathcal{B}\left( \mathcal{S},\mathcal{S}^{\ast }\right) ,\quad a\rightarrow 
\mathrm{Op}_{\sigma ,T}(a),
\end{equation*}%
where $\mathcal{S}$ is the dense linear subspace of $\mathcal{H}$ consisting
of the $\mathcal{C}^{\infty }$ vectors of the representation $\mathcal{W}%
_{\sigma ,T}$, $\mathcal{S}^{\ast }$ is the space of all continuous,
anti-linear mappings $\mathcal{S}\rightarrow \mathbb{C}$ and $\mathcal{S}%
^{\ast }\left( W\right) $ is the space of all continuous, anti-linear
mappings $\mathcal{S}\left( W\right) \rightarrow \mathbb{C}$.

\section{The framework}

Our notations are rather standard but we recall here some of them to avoid
any ambiguity. Let $\left( W,\sigma \right) $ be a symplectic vector space,
that is a real finite dimensional vector space $W$ equipped with a real
antisymmetric non-degenerate bilinear form $\sigma $. We denote by $\sigma
^{\flat }$ the isomorphism associated with the non-degenerate bilinear form $%
\sigma $,%
\begin{equation*}
\sigma ^{\flat }:W\rightarrow W^{\ast },\quad \sigma ^{\flat }\left( \xi
\right) =\sigma \left( \xi ,\cdot \right) ,\quad \xi \in W.
\end{equation*}

\textsc{Symplectic adjoint}

Suppose that $\left( W_{1},\sigma _{1}\right) $ and $\left( W_{2},\sigma
_{2}\right) $ are symplectic vector spaces and $T:W_{1}\rightarrow W_{2}$ is
a linear map. Define the symplectic adjoint $T^{\sigma }:W_{2}\rightarrow
W_{1}$ by%
\begin{equation*}
T^{\sigma }:W_{2}\overset{\sigma _{2}^{\flat }}{\longrightarrow }W_{2}^{\ast
}\overset{T^{\ast }}{\longrightarrow }W_{1}^{\ast }\overset{\left( \sigma
_{1}^{\flat }\right) ^{-1}}{\longrightarrow }W_{1},
\end{equation*}%
\begin{gather*}
\begin{array}{ccc}
W_{2}\quad & \overset{T^{\sigma }}{\longrightarrow } & \quad W_{1}%
\end{array}
\\
\begin{array}{ccc}
\sigma _{2}^{\flat }\downarrow \text{ }\quad & \text{ }\quad & \text{ }\quad
\downarrow \sigma _{1}^{\flat }%
\end{array}%
\text{ } \\
\begin{array}{ccc}
W_{2}^{\ast }\quad & \overset{T^{\ast }}{\longrightarrow } & \quad
W_{1}^{\ast }%
\end{array}%
\end{gather*}%
where $T^{\ast }$ is the the usual adjoint,%
\begin{equation*}
T^{\ast }:W_{2}^{\ast }\rightarrow W_{1}^{\ast },\quad T^{\ast }\lambda
_{2}=\lambda _{2}\circ T.
\end{equation*}

\begin{lemma}
The linear map $T^{\sigma }:W_{2}\longrightarrow W_{1}$ satisfies 
\begin{equation}
\sigma _{1}\left( T^{\sigma }\xi _{1},\xi _{2}\right) =\sigma _{2}\left( \xi
_{1},T\xi _{2}\right) ,\quad \xi _{1}\in W_{1},\xi _{2}\in W_{2}.  \label{n1}
\end{equation}
\end{lemma}

\begin{proof}
Let $\xi _{1}\in W_{1},\xi _{2}\in W_{2}$. Then 
\begin{eqnarray*}
\sigma _{1}\left( T^{\sigma }\xi _{1},\xi _{2}\right) &=&\sigma _{1}^{\flat
}\left( T^{\sigma }\xi _{1}\right) \left( \xi _{2}\right) =\sigma
_{1}^{\flat }T^{\sigma }\left( \xi _{1}\right) \left( \xi _{2}\right)
=T^{\ast }\sigma _{2}^{\flat }\left( \xi _{1}\right) \left( \xi _{2}\right)
\\
&=&\sigma _{2}^{\flat }\left( \xi _{1}\right) \left( T\xi _{2}\right)
=\sigma _{2}\left( \xi _{1},T\xi _{2}\right) .
\end{eqnarray*}
\end{proof}

\begin{remark}
The property $\left( \ref{n1}\right) $ characterizes the symplectic adjoint.
\end{remark}

\begin{definition}
$\left( \mathrm{a}\right) $ A symplectic isomorphism or a symplectomorphism $%
\phi $ between symplectic vector spaces $\left( W_{1},\sigma _{1}\right) $
and $\left( W_{2},\sigma _{2}\right) $ is a linear isomorphism $\phi
:W_{1}\rightarrow W_{2}$ such that $\phi ^{\ast }\sigma _{2}=\sigma _{1}$.
By definition, $\phi ^{\ast }\sigma _{2}\left( \xi ,\eta \right) =\sigma
_{2}\left( \phi \xi ,\phi \eta \right) $, $\xi $, $\eta \in W_{1}$. If a
symplectomorphism exists, $\left( W_{1},\sigma _{1}\right) $ and $\left(
W_{2},\sigma _{2}\right) $ are said to be symplectomorphic.

$\left( \mathrm{b}\right) $ For a symplectic vector space $\left( W,\sigma
\right) $ we denote by 
\begin{equation*}
\mathrm{Sp}\left( W,\sigma \right) :=\left\{ \phi \in \mathrm{GL}\left(
W\right) :\phi ^{\ast }\sigma =\sigma \right\}
\end{equation*}%
the group of linear symplectomorphisms of $\left( W,\sigma \right) $.
\end{definition}

\begin{remark}
We have $\phi :W_{1}\rightarrow W_{2}$ is a symplectic isomorphism if and
only if $\phi $ is a linear isomorphism and $\phi ^{\sigma }\circ \phi =%
\mathrm{Id}_{W_{1}}$.
\end{remark}

\begin{lemma}
Let $\left( W,\sigma \right) $ be a symplectic vector space and $%
S:W\rightarrow W$ a linear isomorphism. Then $S^{\sigma }=S$ if and only if
there is a linear isomorphism $\phi _{{S}}:W\rightarrow W$ such that $S=\phi
_{{S}}^{\sigma }\circ \phi _{{S}}$.
\end{lemma}

\begin{proof}
The map 
\begin{equation*}
\sigma _{S}:W\times W\rightarrow \mathbb{R},\quad \sigma _{S}\left( \xi
,\eta \right) =\sigma \left( \xi ,S\eta \right) ,\quad \xi ,\eta \in W,
\end{equation*}%
is bilinear. Let us note that $\sigma _{S}$ is antisymmetric if and only if $%
S^{\sigma }=S$ and $\sigma _{S}$ is non-degenerate if and only if $S$ is an
isomorphism.

If $S:W\rightarrow W$ is a linear isomorphism and $S^{\sigma }=S$, then the $%
2$-form $\sigma _{S}$ is symplectic. Let $\phi _{{S}}:W\rightarrow W$ a
linear isomorphism that takes a symplectic basis with respect to $\sigma
_{S} $ to a symplectic basis with respect to $\sigma $. Then 
\begin{equation*}
\sigma \left( \xi ,S\eta \right) =\sigma _{S}\left( \xi ,\eta \right)
=\sigma \left( \phi _{{S}}\xi ,\phi _{{S}}\eta \right) =\sigma \left( \xi
,\phi _{{S}}^{\sigma }\circ \phi _{{S}}\eta \right) ,\quad \xi ,\eta \in W,
\end{equation*}%
\begin{equation*}
\Rightarrow \sigma \left( \xi ,S\eta \right) =\sigma \left( \xi ,\phi _{{S}%
}^{\sigma }\circ \phi _{{S}}\eta \right) ,\quad \xi ,\eta \in W,
\end{equation*}%
hence $S=\phi _{{S}}^{\sigma }\circ \phi _{{S}}$. The converse is obvious.
\end{proof}

\begin{corollary}
Let $\left( W,\sigma \right) $ be a symplectic vector space and $%
S:W\rightarrow W$ a linear isomorphism. If $S^{\sigma }=S$, then $\det S>0$.
\end{corollary}

\begin{definition}
Let $\left( W,\sigma \right) $ be a and $X\subset W$ be a linear subspace.
The symplectic complement of $X$ is the subspace%
\begin{equation*}
X^{\sigma }=\left\{ v\in W:\sigma (v,w)=0\text{ for all }w\in X\right\}
\end{equation*}%
A subspace $X\subset W$ is called isotropic if $X\subset X^{\sigma }$ and
involutive if $X^{\sigma }\subset X$. If both are valid, i.e. $X=X^{\sigma }$%
, then $X$ is lagrangian. An isotropic subspace $X\subset W$ is lagrangian
if and only if $2\dim X=\dim W$.
\end{definition}

\begin{remark}
$\left( \mathrm{a}\right) $ Let $X$ be an $n$ dimensional vector space over $%
\mathbb{R}$ and $X^{\ast }$ its dual. Denote $x,y,...$ the elements af $X$
and $k,p,...$ those of $X^{\ast }$. Let $\left\langle \cdot ,\cdot
\right\rangle :X\times X^{\ast }\rightarrow \mathbb{R}$ be the duality form,
which is a non-degenerate bilinear form. The symplectic space is defined by $%
W=T^{\ast }(X)=X\times X^{\ast }$ the symplectic form being $\sigma \left(
\left( x,k\right) ,\left( y,p\right) \right) =\left\langle y,k\right\rangle
-\left\langle x,p\right\rangle $. Observe that $X$ and $X^{\ast }$ are
lagrangian subspaces of $W$.

$\left( \mathrm{b}\right) $ Let us mention that there is a kind of converse
to this construction. Let $\left( X,X^{\ast }\right) $ be a couple of
lagrangian subspaces of $W$ such that $X\cap X^{\ast }=0$ or, equivalently, $%
X+X^{\ast }=W$. If for $x\in X$ and $p\in X^{\ast }$ we define $\left\langle
x,p\right\rangle =\sigma \left( p,x\right) $, then we get a non-degenerate
bilinear form on $X\times X^{\ast }$ which allows us to identify $X^{\ast }$
with the dual of $X$. A couple $\left( X,X^{\ast }\right) $ of subspaces of $%
W$ with the preceding properties is called a holonomic decomposition of $W$.
Observe that if $\xi =x+k$ and $\eta =y+p$ are their decomposition s in $W$,
then 
\begin{equation*}
\sigma \left( \xi ,\eta \right) =\left\langle y,k\right\rangle -\left\langle
x,p\right\rangle .
\end{equation*}
\end{remark}

\textsc{The symplectic Fourier transform}

A symplectic vector space $\left( W,\sigma \right) $ is always orientable
since the $2$-form $\sigma $ is non-degenerate if and only if its $n$-fold
exterior power is non-zero, i.e. 
\begin{equation*}
\sigma ^{n}=\underset{n}{\underbrace{\sigma \wedge ...\wedge \sigma }}\neq 0,
\end{equation*}%
where $\dim W=2n$. We will call the exterior power $\sigma ^{n}$ the
symplectic volume form. When $\left( W,\sigma \right) $ is the standard
symplectic space $\left( \mathbb{R}_{z}^{2n},\sigma _{n}\right) $, 
\begin{equation*}
\sigma _{n}\left( z,z^{\prime }\right) =\sum \left( p_{j}x_{j}^{\prime
}-p_{j}^{\prime }x_{j}\right) ,
\end{equation*}%
$z=\left( x_{1},...,x_{n};p_{1},...,p_{n}\right) $ and $z^{\prime }=\left(
x_{1}^{\prime },...,x_{n}^{\prime };p_{1}^{\prime },...,p_{n}^{\prime
}\right) $, then the usual volume form on $\mathbb{R}_{z}^{2n}$,%
\begin{equation*}
\text{\textrm{Vol}}_{2n}=\mathrm{d}p_{1}\wedge ...\wedge \mathrm{d}%
p_{n}\wedge \mathrm{d}x_{1}\wedge ...\wedge \mathrm{d}x_{n},
\end{equation*}%
is related to the symplectic volume form by 
\begin{equation*}
\text{\textrm{Vol}}_{2n}=(-1)^{\frac{n\left( n-1\right) }{2}}\frac{1}{n!}%
\sigma ^{n}=(-1)^{\left[ \frac{n}{2}\right] }\frac{1}{n!}\sigma ^{n}.
\end{equation*}%
The form $\frac{1}{n!}\sigma ^{n}$ is called the Liouville volume of $\left(
W,\sigma \right) $.

We define the Fourier measure $\mathrm{d}^{\sigma }\xi $ as the unique Haar
measure on $\left( W,\sigma \right) $ such that the symplectic Fourier
transform or $\sigma $-Fourier transform, 
\begin{equation*}
\left( \mathcal{F}_{\sigma }a\right) \left( \xi \right) =\int_{W}\mathrm{e}%
^{-i\sigma \left( \xi ,\eta \right) }a\left( \eta \right) \mathrm{d}^{\sigma
}\eta ,\quad a\in \mathcal{S}\left( W\right) ,
\end{equation*}%
is involutive (i.e.$\mathcal{F}_{\sigma }^{2}=1$) and unitary on $%
L^{2}\left( W\right) $. We use the same notation $\mathcal{F}_{\sigma }%
\mathcal{\ }$for the extension to $\mathcal{S}^{\prime }\left( W\right) $ of
this $\sigma $-Fourier transform. Let us note that 
\begin{equation*}
\mathrm{d}^{\sigma }\xi =\left( 2\pi \right) ^{-\frac{\dim W}{2}}\left[
\left( \frac{\dim W}{2}\right) !\right] ^{-1}\underset{\frac{\dim W}{2}}{%
\left\vert \underbrace{\sigma \wedge ...\wedge \sigma }\right\vert },
\end{equation*}%
where $\underset{\frac{\dim W}{2}}{\left\vert \underbrace{\sigma \wedge
...\wedge \sigma }\right\vert }$ is the $1$- density given by the symplectic
volume form $\sigma ^{\frac{\dim W}{2}}$.

\begin{lemma}
Let $\left( W_{1},\sigma _{1}\right) $ and $\left( W_{2},\sigma _{2}\right) $
be two symplectic spaces of same dimension $2n$. If $\phi $ is a symplectic
isomorphism $\left( W_{1},\sigma _{1}\right) \longrightarrow \left(
W_{2},\sigma _{2}\right) $, then 
\begin{equation*}
\phi ^{\ast }\underset{n}{\left\vert \underbrace{\sigma _{2}\wedge ...\wedge
\sigma _{2}}\right\vert }=\underset{n}{\left\vert \underbrace{\sigma
_{1}\wedge ...\wedge \sigma _{1}}\right\vert },\quad \phi ^{\ast }\left( 
\mathrm{d}^{\sigma _{2}}\xi _{2}\right) =\mathrm{d}^{\sigma _{1}}\xi _{1}
\end{equation*}%
and 
\begin{equation*}
\mathcal{F}_{\sigma _{1}}\circ \phi ^{\ast }=\phi ^{\ast }\circ \mathcal{F}%
_{\sigma _{2}}
\end{equation*}%
\begin{equation*}
\Leftrightarrow \mathcal{F}_{\sigma _{2}}=\left( \phi ^{\sigma }\right)
^{\ast }\circ \mathcal{F}_{\sigma _{1}}\circ \phi ^{\ast }.
\end{equation*}
\end{lemma}

\begin{proof}
The first two equalities are direct consequences of the fact that $\phi $ is
a symplectic isomorphism ($\phi ^{\ast }\sigma _{2}=\sigma _{1}$). As for
the third equality, it is enough to to prove it equality for $b$ in $%
\mathcal{S}\left( W\right) $. Let $b\in \mathcal{S}\left( W\right) $. Then
we have 
\begin{eqnarray*}
\mathcal{F}_{\sigma _{1}}\left( b\circ \phi \right) \left( \xi _{1}\right)
&=&\int_{W_{1}}\mathrm{e}^{-\mathrm{i}\sigma _{1}\left( \xi _{1},\eta
_{1}\right) }b\left( \phi \left( \eta _{1}\right) \right) \mathrm{d}^{\sigma
_{1}}\eta _{1} \\
&=&\int_{W_{1}}\mathrm{e}^{-\mathrm{i}\sigma _{2}\left( \phi \left( \xi
_{1}\right) ,\phi \left( \eta _{1}\right) \right) }b\left( \phi \left( \eta
_{1}\right) \right) \mathrm{d}^{\sigma _{1}}\eta _{1} \\
&=&\int_{W_{1}}\phi ^{\ast }\left( \mathrm{e}^{-\mathrm{i}\sigma _{2}\left(
\phi \left( \xi _{1}\right) ,\cdot \right) }b\left( \cdot \right) \mathrm{d}%
^{\sigma _{2}}\cdot \right) \\
&=&\int_{W_{2}}\mathrm{e}^{-\mathrm{i}\sigma _{2}\left( \phi \left( \xi
_{1}\right) ,\eta _{2}\right) }b\left( \eta _{2}\right) \mathrm{d}^{\sigma
_{2}}\eta _{2} \\
&=&\mathcal{F}_{\sigma _{2}}\left( b\right) \left( \phi \left( \xi
_{1}\right) \right) =\phi ^{\ast }\circ \mathcal{F}_{\sigma _{2}}\left(
b\right) \left( \xi _{1}\right) ,\quad \xi _{1}\in W_{1}.
\end{eqnarray*}%
The equivalence is a consequence of identity $\phi ^{\sigma }\circ \phi =%
\mathrm{id}_{W_{1}}$.
\end{proof}

Let $S:W\rightarrow W$ is a linear isomorphism such that $S^{\sigma }=S$ and 
$\phi _{{S}}$ a symplectic isomorphism, $\phi _{{S}}:\left( W,\sigma
_{S}\right) \longrightarrow \left( W,\sigma \right) $, such that $S=\phi _{{S%
}}^{\sigma }\circ \phi _{{S}}$. Then,%
\begin{equation*}
\mathrm{d}^{\sigma _{S}}\eta =\phi _{{S}}^{\ast }\left( \mathrm{d}^{\sigma
}\eta \right) =\left\vert \det \phi _{{S}}\right\vert \mathrm{d}^{\sigma
}\eta =\left( \det S\right) ^{\frac{1}{2}}\mathrm{d}^{\sigma }\eta ,
\end{equation*}%
and for $b\in \mathcal{S}\left( W\right) $ we have%
\begin{eqnarray*}
\mathcal{F}_{\sigma }\left( b\right) \left( \xi \right) &=&\int_{W}\mathrm{e}%
^{-\mathrm{i}\sigma \left( \xi ,\eta \right) }b\left( \eta \right) \mathrm{d}%
^{\sigma }\eta =\int_{W}\mathrm{e}^{-\mathrm{i}\sigma _{S}\left( \xi
,S^{-1}\eta \right) }b\left( \eta \right) \mathrm{d}^{\sigma }\eta \\
&=&\det S\int_{W}\mathrm{e}^{-\mathrm{i}\sigma _{S}\left( \xi ,\zeta \right)
}b\left( S\zeta \right) \mathrm{d}^{\sigma }\zeta =\left( \det S\right) ^{%
\frac{1}{2}}\int_{W}\mathrm{e}^{-\mathrm{i}\sigma _{S}\left( \xi ,\zeta
\right) }b\left( S\zeta \right) \mathrm{d}^{\sigma _{S}}\zeta \\
&=&\left( \det S\right) ^{\frac{1}{2}}\mathcal{F}_{\sigma _{S}}\left( b\circ
S\right) \left( \xi \right) ,\quad \xi \in W.
\end{eqnarray*}

\begin{lemma}
\label{n3}Let $\left( W,\sigma \right) $ be a symplectic vector space and $%
S:W\rightarrow W$ a linear isomorphism such that $S^{\sigma }=S$. If $\sigma
_{S}$ is the symplectic form 
\begin{equation*}
\sigma _{S}:W\times W\rightarrow \mathbb{R},\quad \sigma _{S}\left( \xi
,\eta \right) =\sigma \left( \xi ,S\eta \right) ,\quad \xi ,\eta \in W,
\end{equation*}%
then 
\begin{equation*}
\mathcal{F}_{\sigma }=\left( \det S\right) ^{\frac{1}{2}}\mathcal{F}_{\sigma
_{S}}\circ S^{\ast }\quad \text{on }\mathcal{S}^{\prime }\left( W\right) .
\end{equation*}
\end{lemma}

Since $\mathrm{d}^{\sigma }\eta $ is a multiple of the Lebesgue measure, the
change of variables formula implies that if $A:W\longrightarrow W$ is a
linear isomorphism, then 
\begin{equation*}
A^{\ast }\circ \mathcal{F}_{\sigma }=\left\vert \det A\right\vert ^{-1}%
\mathcal{F}_{\sigma }\circ \left[ \left( A^{\sigma }\right) ^{-1}\right]
^{\ast }.
\end{equation*}%
Let $b\in \mathcal{S}\left( W\right) $. Then 
\begin{eqnarray*}
A^{\ast }\circ \mathcal{F}_{\sigma }\left( b\right) \left( \xi \right) &=&%
\mathcal{F}_{\sigma }\left( b\right) \left( A\xi \right) =\int_{W}\mathrm{e}%
^{-\mathrm{i}\sigma \left( A\xi ,\eta \right) }b\left( \eta \right) \mathrm{d%
}^{\sigma }\eta \\
&=&\int_{W}\mathrm{e}^{-\mathrm{i}\sigma \left( \xi ,A^{\sigma }\eta \right)
}b\circ \left( A^{\sigma }\right) ^{-1}\left( A^{\sigma }\eta \right) 
\mathrm{d}^{\sigma }\eta \\
&=&\left\vert \det A\right\vert ^{-1}\int_{W}\mathrm{e}^{-\mathrm{i}\sigma
\left( \xi ,\zeta \right) }b\circ \left( A^{\sigma }\right) ^{-1}\left(
\zeta \right) \mathrm{d}^{\sigma }\zeta \\
&=&\left\vert \det A\right\vert ^{-1}\mathcal{F}_{\sigma }\left( b\circ
\left( A^{\sigma }\right) ^{-1}\right) \left( \xi \right) .
\end{eqnarray*}

In particular, if $S=S^{\sigma }:W\rightarrow W$ is a linear isomorphism,
then%
\begin{equation*}
S^{\ast }\circ \mathcal{F}_{\sigma }=\left\vert \det S\right\vert ^{-1}%
\mathcal{F}_{\sigma }\circ \left( S^{-1}\right) ^{\ast }.
\end{equation*}%
For $\lambda \in \mathcal{C}_{\mathrm{pol}}^{\infty }\left( W\right) $, we
define the operator 
\begin{equation*}
\lambda \left( D_{\sigma }\right) =\mathcal{F}_{\sigma }\circ \mathrm{M}%
_{\lambda \left( \cdot \right) }\circ \mathcal{F}_{\sigma },
\end{equation*}%
where $\mathrm{M}_{\lambda \left( \cdot \right) }$ denotes the
multiplication operator by the function $\lambda \left( \cdot \right) $. If $%
S=S^{\sigma }:W\rightarrow W$ is a linear isomorphism, then%
\begin{equation*}
S^{\ast }\circ \lambda \left( D_{\sigma }\right) =\left( \lambda \circ
S^{-1}\right) \left( D_{\sigma }\right) \circ S^{\ast }.
\end{equation*}%
Indeed, for $b\in \mathcal{S}\left( W\right) $ we have%
\begin{eqnarray*}
S^{\ast }\circ \lambda \left( D_{\sigma }\right) &=&S^{\ast }\circ \mathcal{F%
}_{\sigma }\circ \mathrm{M}_{\lambda \left( \cdot \right) }\circ \mathcal{F}%
_{\sigma }=\left( \det S\right) ^{-1}\mathcal{F}_{\sigma }\circ \left(
S^{-1}\right) ^{\ast }\circ \mathrm{M}_{\lambda \left( \cdot \right) }\circ 
\mathcal{F}_{\sigma } \\
&=&\left( \det S\right) ^{-1}\mathcal{F}_{\sigma }\circ \mathrm{M}_{\lambda
\circ S^{-1}\left( \cdot \right) }\circ \left( S^{-1}\right) ^{\ast }\circ 
\mathcal{F}_{\sigma } \\
&=&\left( \det S\right) ^{-1}\left( \det S\right) \mathcal{F}_{\sigma }\circ 
\mathrm{M}_{\lambda \circ S^{-1}\left( \cdot \right) }\circ \mathcal{F}%
_{\sigma }\circ S^{\ast } \\
&=&\left( \lambda \circ S^{-1}\right) \left( D_{\sigma }\right) \circ
S^{\ast }.
\end{eqnarray*}

\begin{lemma}
\label{n11}Let $\left( W,\sigma \right) $ be a symplectic vector space

$\left( \mathrm{a}\right) $ If $A:W\rightarrow W$ is a linear isomorphism,
then%
\begin{equation*}
A^{\ast }\circ \mathcal{F}_{\sigma }=\left\vert \det A\right\vert ^{-1}%
\mathcal{F}_{\sigma }\circ \left[ \left( A^{\sigma }\right) ^{-1}\right]
^{\ast }.
\end{equation*}

$\left( \mathrm{b}\right) $ If $S=S^{\sigma }:W\rightarrow W$ is a linear
isomorphism and $\lambda \in \mathcal{C}_{\mathrm{pol}}^{\infty }\left(
W\right) $, then 
\begin{equation*}
S^{\ast }\circ \mathcal{F}_{\sigma }=\left( \det S\right) ^{-1}\mathcal{F}%
_{\sigma }\circ \left( S^{-1}\right) ^{\ast },
\end{equation*}%
and 
\begin{equation*}
S^{\ast }\circ \lambda \left( D_{\sigma }\right) =\left( \lambda \circ
S^{-1}\right) \left( D_{\sigma }\right) \circ S^{\ast },
\end{equation*}%
where 
\begin{equation*}
\lambda \left( D_{\sigma }\right) =\mathcal{F}_{\sigma }\circ \mathrm{M}%
_{\lambda \left( \cdot \right) }\circ \mathcal{F}_{\sigma }:\mathcal{S}%
^{\prime }\left( W\right) \longrightarrow \mathcal{S}^{\prime }\left(
W\right) .
\end{equation*}

$\left( \mathrm{c}\right) $ If $S=S^{\sigma }:W\rightarrow W$ is a linear
isomorphism and $\lambda \in \mathcal{C}_{\mathrm{pol}}^{\infty }\left(
W\right) $, then%
\begin{equation*}
\lambda \left( D_{\sigma }\right) =\lambda \circ S\left( D_{\sigma
_{S}}\right)
\end{equation*}
\end{lemma}

\begin{proof}
$\left( \mathrm{c}\right) $ We use $\left( \mathrm{b}\right) $ and Lemma \ref%
{n3} twice:%
\begin{eqnarray*}
\lambda \left( D_{\sigma }\right) &=&\mathcal{F}_{\sigma }\circ \mathrm{M}%
_{\lambda \left( \cdot \right) }\circ \mathcal{F}_{\sigma }=\left( \det
S\right) ^{\frac{1}{2}}\mathcal{F}_{\sigma _{S}}\circ S^{\ast }\circ \mathrm{%
M}_{\lambda \left( \cdot \right) }\circ \mathcal{F}_{\sigma } \\
&=&\left( \det S\right) ^{\frac{1}{2}}\mathcal{F}_{\sigma _{S}}\circ \mathrm{%
M}_{\lambda \circ S\left( \cdot \right) }\circ S^{\ast }\circ \mathcal{F}%
_{\sigma } \\
&=&\left( \det S\right) ^{-\frac{1}{2}}\mathcal{F}_{\sigma _{S}}\circ 
\mathrm{M}_{\lambda \circ S\left( \cdot \right) }\circ \mathcal{F}_{\sigma
}\circ \left( S^{-1}\right) ^{\ast } \\
&=&\left( \det S\right) ^{-\frac{1}{2}}\left( \det S\right) ^{\frac{1}{2}}%
\mathcal{F}_{\sigma _{S}}\circ \mathrm{M}_{\lambda \circ S\left( \cdot
\right) }\circ \mathcal{F}_{\sigma _{S}}\circ S^{\ast }\circ \left(
S^{-1}\right) ^{\ast } \\
&=&\mathcal{F}_{\sigma _{S}}\circ \mathrm{M}_{\lambda \circ S\left( \cdot
\right) }\circ \mathcal{F}_{\sigma _{S}}=\lambda \circ S\left( D_{\sigma
_{S}}\right) .
\end{eqnarray*}
\end{proof}

\begin{remark}
$\left( \mathrm{a}\right) $ If $a\in \mathcal{S}^{\ast }\left( W\right) $
and $\xi \in W$, then $\tau _{\xi }a$ denote the translate by $\xi $ of the
distribution $a$, i.e. $\left( \tau _{\xi }a\right) \left( \cdot \right)
=a\left( \cdot -\xi \right) $. The family $\left\{ \tau _{\xi }\right\}
_{\xi \in W}$ is the unitary representation in $L^{2}\left( W\right) $ of
the additive group $W$ by translations. The family $\left\{ \tau _{\xi
}\right\} _{\xi \in W}$ also defines a representation of the additive group $%
W$ by topological automorphisms of $\mathcal{S}^{\ast }\left( W\right) $
which leave $\mathcal{S}\left( W\right) $ invariant. Since from the
definition of the symplectic Fourier transform 
\begin{equation*}
\tau _{\xi }=\mathrm{e}^{-i\sigma \left( D_{\sigma },\xi \right) },\quad \xi
\in W,
\end{equation*}%
it follows that for $\lambda \in \mathcal{C}_{\mathrm{pol}}^{\infty }\left(
W\right) $ and $\xi \in W$%
\begin{equation*}
\tau _{\xi }\circ \lambda \left( D_{\sigma }\right) =\lambda \left(
D_{\sigma }\right) \circ \tau _{\xi }.
\end{equation*}

$\left( \mathrm{b}\right) $ If $S=S^{\sigma }:W\rightarrow W$ is a linear
isomorphism and $\xi \in W$, then 
\begin{equation*}
S^{\ast }\circ \tau _{\xi }=\tau _{S^{-1}\xi }\circ S^{\ast }.
\end{equation*}%
Indeed, using equality $\tau _{\xi }=\mathrm{e}^{-i\sigma \left( D_{\sigma
},\xi \right) }$ one sees that 
\begin{eqnarray*}
S^{\ast }\circ \tau _{\xi } &=&S^{\ast }\circ \mathrm{e}^{-i\sigma \left(
D_{\sigma },\xi \right) }=\left( \mathrm{e}^{-i\sigma \left( \cdot ,\xi
\right) }\circ S^{-1}\right) \left( D_{\sigma }\right) \circ S^{\ast } \\
&=&\left( \mathrm{e}^{-i\sigma \left( S^{-1}\cdot ,\xi \right) }\circ
\right) \left( D_{\sigma }\right) \circ S^{\ast }=\left( \mathrm{e}%
^{-i\sigma \left( \cdot ,S^{-1}\xi \right) }\circ \right) \left( D_{\sigma
}\right) \circ S^{\ast } \\
&=&\tau _{S^{-1}\xi }\circ S^{\ast }.
\end{eqnarray*}
\end{remark}

\begin{corollary}
\label{n12}If $S=S^{\sigma }:W\rightarrow W$ is a linear isomorphism, $%
\lambda \in \mathcal{C}_{\mathrm{pol}}^{\infty }\left( W\right) $ and $\xi
\in W$, then 
\begin{equation*}
\tau _{S^{-1}\xi }\circ S^{\ast }\circ \lambda \left( D_{\sigma }\right)
=S^{\ast }\circ \lambda \left( D_{\sigma }\right) \circ \tau _{\xi }.
\end{equation*}
\end{corollary}

In many situations we need to consider additional structures such as the
inner product or complex structures. We shall ask that these structures to
be compatible with symplectic structure.

Recall that a complex structure on a vector space $V$ is an automorphism $%
J:V\rightarrow V$ such that $J^{2}=-\mathrm{Id}$. We denote the space of
linear complex structures on $V$ by $\mathcal{J}\left( V\right) $.

A complex structure $J$ on a symplectic vector space $\left( W,\sigma
\right) $ is called $\sigma $-compatible if%
\begin{equation*}
(J^{\ast }\sigma )(v,w)=\sigma (Jv,Jw)=\sigma (v,w),
\end{equation*}%
for all $v,w\in W$ and 
\begin{equation*}
\sigma \left( v,Jv\right) >0
\end{equation*}%
for all nonzero $v\in W$. This is equivalent to 
\begin{equation*}
g:W\times W\rightarrow \mathbb{R},\quad g\left( v,w\right) =\sigma \left(
v,Jw\right) ,\text{ }v,w\in W,
\end{equation*}%
is a positive definite inner product. We denote by $\mathcal{J}\left(
W,\sigma \right) $ the space of $\sigma $-compatible complex structures on $%
\left( W,\sigma \right) $.

An inner product $g$ on a symplectic vector space $\left( W,\sigma \right) $
is called $\sigma $-compatible if there is a complex structure $J$ on $W$
such that 
\begin{equation*}
g(u,v)=\sigma \left( v,Jw\right)
\end{equation*}
for all $v,w\in W$. We denote by $\mathcal{G}\left( W\right) $ the space of
inner products on $W$, and by $\mathcal{G}\left( W,\sigma \right) $ the
space of $\sigma $-compatible inner products on $\left( W,\sigma \right) $.

\begin{remark}
$\left( \mathrm{a}\right) $ The compatibility condition $g(u,v)=\sigma
\left( v,Jw\right) $ defines a smooth diffeomorphism%
\begin{equation*}
\mathcal{J}\left( W,\sigma \right) \rightarrow \mathcal{G}\left( W,\sigma
\right) .
\end{equation*}%
In fact, the same formula $g(u,v)=\sigma \left( v,Jw\right) $ defines a
linear isomorphism $J\rightarrow g$ from the ambient vector space of linear
maps $V\rightarrow V$ to the ambient vector space of bilinear forms $V\times
V\rightarrow \mathbb{R}$. The bijection $\mathcal{J}\left( W,\sigma \right)
\rightarrow \mathcal{G}\left( W,\sigma \right) $ is the restriction of this
linear isomorphism, so it is a diffeomorphism.

$\left( \mathrm{b}\right) $ There is a canonical retraction $\mathcal{G}%
\left( W\right) \rightarrow \mathcal{G}\left( W,\sigma \right) $ (see for
instance Proposition 2.5.6 in \cite{McDuff}), so we can associate to any
inner product in a canonical manner a $\sigma $-compatible one.
\end{remark}

\section{$\protect\omega _{\protect\sigma ,T}$-representation and the
associated Weyl system}

\begin{lemma}
Let $T:W\rightarrow W$ be a linear map, and let $\omega _{\sigma ,T}$ be the
function 
\begin{equation*}
\omega _{\sigma ,T}:W\times W\longrightarrow \mathbb{T},
\end{equation*}%
\begin{equation*}
\omega _{\sigma ,T}\left( \xi ,\eta \right) =\mathrm{e}^{\mathrm{i}\sigma
\left( \xi ,T\eta \right) },\quad \xi ,\eta \in W.
\end{equation*}%
\newline
Then $\omega _{\sigma ,T}$ is a 2-cocycle or Schur multiplier. Moreover, the
Schur multiplier $\omega _{\sigma ,T}$ is non-degenerate, that is 
\begin{eqnarray*}
\omega _{\sigma ,{T}}\left( \xi ,\eta \right) &=&\omega _{\sigma ,T}\left(
\eta ,\xi \right) ,\quad \forall \eta \in W \\
&\Rightarrow &\xi \ =0,
\end{eqnarray*}%
if and only if $T+T^{\sigma }$ is an isomorphism.
\end{lemma}

\begin{proof}
We have to show that $\omega _{\sigma ,T}$ satisfies the cocycle equation 
\begin{equation*}
\omega _{\sigma ,T}\left( \xi ,\eta \right) \omega _{{\sigma ,T}}\left( \xi
+\eta ,\zeta \right) =\omega _{\sigma ,T}\left( \xi ,\eta +\zeta \right)
\omega _{\sigma ,T}\left( \eta ,\zeta \right) ,\quad \xi ,\eta ,\zeta \in W.
\end{equation*}%
By definition this equality is equivalent to 
\begin{eqnarray*}
\sigma \left( \xi ,T\eta \right) +\sigma \left( \xi +\eta ,T\zeta \right)
&=&\sigma \left( \xi ,T\left( \eta +\zeta \right) \right) +\sigma \left(
\eta ,T\zeta \right) \\
&=&\sigma \left( \xi ,T\eta \right) +\sigma \left( \xi ,T\zeta \right)
+\sigma \left( \eta ,T\zeta \right) ,\quad \xi ,\eta ,\zeta \in W.
\end{eqnarray*}%
Obviously we have $\omega _{\sigma ,T}\left( \xi ,0\right) =\omega _{\sigma
,T}\left( 0,\xi \right) =1$, $\xi \in W$. Hence $\omega _{\sigma ,T}$ is a
2-cocycle or Schur multiplier.

Let $\xi \in W$. Then 
\begin{equation*}
\omega _{\sigma ,T}\left( \xi ,\eta \right) =\omega _{\sigma ,T}\left( \eta
,\xi \right) ,\quad \forall \eta \in W
\end{equation*}%
\begin{equation*}
\Leftrightarrow \sigma \left( \xi ,T\eta \right) =\sigma \left( \eta ,T\xi
\right) \quad \forall \eta \in W
\end{equation*}%
\begin{equation*}
\Leftrightarrow \sigma \left( \left( T+T^{\sigma }\right) \xi ,\eta \right)
=0\quad \forall \eta \in W
\end{equation*}%
\begin{equation*}
\Leftrightarrow \left( T+T^{\sigma }\right) \xi =0\Leftrightarrow \xi \in
\ker \left( T+T^{\sigma }\right) .
\end{equation*}%
So we deduce that 
\begin{equation*}
\left( \forall \eta \in W\right) \left( \omega _{\sigma ,T}\left( \xi ,\eta
\right) =\omega _{\sigma ,T}\left( \eta ,\zeta \right) \right)
\Leftrightarrow \left( \xi \in \ker \left( T+T^{\sigma }\right) \right) ,
\end{equation*}%
and this clearly implies that $\omega _{\sigma ,T}$ is non-degenerate if and
only if $T+T^{\sigma }$ is an isomorphism.
\end{proof}

\begin{remark}
$\left( \mathrm{a}\right) $ We know that for any continuous multiplier $%
\omega $ on $W$, there is a projective representation $\left\{ \mathcal{W}%
\left( \xi \right) \right\} _{\xi \in W}\equiv \left( \mathcal{H},\mathcal{W}%
,\omega \right) $ whose multiplier is $\omega $, that is a strongly
continuous map 
\begin{equation*}
\mathcal{W}:W\rightarrow \mathcal{U}\left( \mathcal{H}\right) )
\end{equation*}%
which satisfies 
\begin{equation*}
\mathcal{W}\left( \xi \right) \mathcal{W}\left( \eta \right) =\omega \left(
\xi ,\eta \right) \mathcal{W}\left( \xi +\eta \right) ,\quad \xi ,\eta \in W.
\end{equation*}%
$\mathcal{W}$ is called a $\omega $-representation (or, less precisely, a
multiplier or ray, or cocycle representation).

$\left( \mathrm{b}\right) $ For instance, $\left( L^{2}\left( W\right) ,%
\mathcal{R}_{\omega },\omega \right) $ is a projective representation of $W$
with $\omega $ the associated multiplier, where 
\begin{equation*}
\mathcal{R}_{\omega }:W\rightarrow \mathcal{U}\left( L^{2}\left( W\right)
\right) ,\quad \mathcal{R}_{\omega }\left( \xi \right) f=\omega \left( \cdot
,\xi \right) f\left( \cdot +\xi \right) ,\quad \xi \in W.
\end{equation*}%
This representation is called the regular $\omega $-representation of $W$.
\end{remark}

Let $\left\{ \mathcal{W}_{\sigma ,T}\left( \xi \right) \right\} _{\xi \in
W}\equiv \left( \mathcal{H},\mathcal{W}_{\sigma ,T},\omega _{\sigma
,T}\right) $ be a $\omega _{\sigma ,T}$-representation of $W$. For fixed $%
\xi $ in $W$, the map 
\begin{equation*}
\mathbb{R}\ni t\rightarrow \mathcal{W}_{\sigma ,T}\left( t\xi \right) \in 
\mathcal{B}(\mathcal{H})
\end{equation*}%
sayisfies 
\begin{equation*}
\mathcal{W}_{\sigma ,T}\left( t\xi \right) \mathcal{W}_{\sigma ,T}\left(
s\xi \right) =\omega _{\sigma ,T}\left( t\xi ,s\xi \right) \mathcal{W}%
_{\sigma ,T}\left( \left( t+s\right) \xi \right) ,\quad s,t\in \mathbb{R},
\end{equation*}%
\begin{equation*}
\mathcal{W}_{\sigma ,T}\left( t\xi \right) \mathcal{W}_{\sigma ,T}\left(
s\xi \right) =\mathrm{e}^{\mathrm{i}ts\sigma \left( \xi ,T\xi \right) }%
\mathcal{W}_{\sigma ,T}\left( \left( t+s\right) \xi \right) ,\quad s,t\in 
\mathbb{R}.
\end{equation*}%
Here $t\rightarrow \mathcal{W}_{\sigma ,T}\left( t\xi \right) $ it is not in
general a group representation of $\mathbb{R}$. Instead, by using equality 
\begin{equation*}
ts=\frac{1}{2}\left[ \left( t+s\right) ^{2}-t^{2}-s^{2}\right] ,
\end{equation*}%
we find that the map $t\rightarrow \widetilde{\mathcal{W}}_{\sigma ,T}\left(
t\xi \right) $ is a group representation of $\mathbb{R}$, where 
\begin{equation*}
\left\{ \widetilde{\mathcal{W}}_{\sigma ,T}\left( \xi \right) \right\} _{\xi
\in W}\equiv \left( \mathcal{H},\widetilde{\mathcal{W}}_{\sigma ,T},%
\widetilde{\omega }_{\sigma ,T}\right)
\end{equation*}%
is the $\widetilde{\omega }_{\sigma ,T}$-representation of $W$ given by 
\begin{equation*}
\widetilde{\mathcal{W}}_{\sigma ,T}\left( \xi \right) =\mathrm{e}^{\frac{%
\mathrm{i}}{2}\sigma \left( \xi ,T\xi \right) }\mathcal{W}_{\sigma ,T}\left(
\xi \right) ,\quad \xi \in W,
\end{equation*}%
and%
\begin{eqnarray*}
\widetilde{\omega }_{\sigma ,T}\left( \xi ,\eta \right) &=&\frac{\mathrm{e}^{%
\frac{\mathrm{i}}{2}\sigma \left( \xi ,T\xi \right) }\mathrm{e}^{\frac{%
\mathrm{i}}{2}\sigma \left( \eta ,T\eta \right) }}{\mathrm{e}^{\frac{\mathrm{%
i}}{2}\sigma \left( \xi +\eta ,T\left( \xi +\eta \right) \right) }}\omega
_{\sigma ,T}\left( \xi ,\eta \right) \\
&=&\mathrm{e}^{-\frac{\mathrm{i}}{2}\sigma \left( \xi ,T\eta \right) -\frac{%
\mathrm{i}}{2}\sigma \left( \eta ,T\xi \right) }\mathrm{e}^{\mathrm{i}\sigma
\left( \xi ,T\eta \right) } \\
&=&\mathrm{e}^{\frac{\mathrm{i}}{2}\left( \sigma \left( \xi ,T\eta \right)
-\sigma \left( \eta ,T\xi \right) \right) } \\
&=&\mathrm{e}^{\frac{\mathrm{i}}{2}\sigma \left( \xi ,\left( T+T^{\sigma
}\right) \eta \right) }=\mathrm{e}^{\mathrm{i}\sigma \left( \xi ,\frac{1}{2}%
\left( T+T^{\sigma }\right) \eta \right) } \\
&=&\omega _{{\sigma ,}\frac{{1}}{{2}}\left( {T+T}^{{\sigma }}\right) }\left(
\xi ,\eta \right) ,\quad \xi ,\eta \in W.
\end{eqnarray*}

We recall that the set of all possible multipliers on $W$ can be given an
abelian group structure by defning the product of two multipliers as their
pointwise product. The resulting group we denote by $\mathrm{Z}^{2}\left( W;%
\mathbb{T}\right) $. The set of all multipliers satisfying%
\begin{equation*}
\alpha \left( \xi ,\eta \right) =\frac{\mu \left( \xi +\eta \right) }{\mu
\left( \xi \right) \mu \left( \eta \right) },\quad \xi ,\eta \in G,
\end{equation*}%
for an arbitrary function $\mu :G\rightarrow \mathbb{T}$ such that $\mu
\left( 0\right) =1$, forms an invariant subgroup $\mathrm{B}^{2}\left( W;%
\mathbb{T}\right) $ of $\mathrm{Z}^{2}\left( W;\mathbb{T}\right) $. Thus we
may form the quotient group 
\begin{equation*}
\mathrm{H}^{2}\left( W;\mathbb{T}\right) =\frac{\mathrm{Z}^{2}\left( W;%
\mathbb{T}\right) }{\mathrm{B}^{2}\left( W;\mathbb{T}\right) }.
\end{equation*}%
Two multipliers $\omega _{1}$ and are $\omega _{2}$ equivalent (or
cohomologous) if $\frac{\omega _{1}}{\omega _{2}}\in \mathrm{B}^{2}\left( W;%
\mathbb{T}\right) $.

The functions $\omega _{\sigma ,T}$ and $\widetilde{\omega }_{\sigma ,T}$
are cohomologous and $\widetilde{\omega }_{\sigma ,T}$ is normalized, i.e. 
\begin{equation*}
\widetilde{\omega }_{\sigma ,T}\left( \xi ,-\xi \right) =1,\quad \xi \in W.
\end{equation*}

The map $t\rightarrow \widetilde{\mathcal{W}}_{\sigma ,T}\left( t\xi \right) 
$ is a group representation of $\mathbb{R}$ and for each $\xi \in W$ there
is a unique self-adjoint operator $\phi _{{T}}\left( \xi \right) $, $T $%
-field operator associated to $\xi $, such that 
\begin{equation*}
\widetilde{\mathcal{W}}_{\sigma ,T}\left( t\xi \right) =\mathrm{e}^{\mathrm{i%
}t\phi _{{T}}\left( \xi \right) }
\end{equation*}%
for all real $t$.

Since 
\begin{equation*}
\widetilde{\mathcal{W}}_{\sigma ,T}\left( \xi \right) =\mathrm{e}^{\frac{%
\mathrm{i}}{2}\sigma \left( \xi ,T\xi \right) }\mathcal{W}_{\sigma ,T}\left(
\xi \right) ,\quad \xi \in W,
\end{equation*}%
and%
\begin{equation*}
\mathcal{W}_{\sigma ,T}\left( \xi \right) =\mathrm{e}^{-\frac{\mathrm{i}}{2}%
\sigma \left( \xi ,T\xi \right) }\widetilde{\mathcal{W}}_{\sigma ,T}\left(
\xi \right) =\mathrm{e}^{\frac{\mathrm{i}}{2}\sigma \left( T\xi ,\xi \right)
}\widetilde{\mathcal{W}}_{\sigma ,T}\left( \xi \right) =\mathrm{e}^{\frac{%
\mathrm{i}}{2}\sigma \left( \xi ,T^{\sigma }\xi \right) }\widetilde{\mathcal{%
W}}_{\sigma ,T}\left( \xi \right) \quad \xi \in W,
\end{equation*}%
we get 
\begin{equation*}
\left\{ \mathcal{W}_{\sigma ,T}\left( \xi \right) :\xi \in W\right\}
^{\prime }=\left\{ \widetilde{\mathcal{W}}_{\sigma ,T}\left( \xi \right)
:\xi \in W\right\} ^{\prime }.
\end{equation*}%
where $S^{\prime }$ is the commutant of the subset $S\subset \mathcal{B}%
\left( \mathcal{H}\right) $. Thus, we have partially proved the following
result.

\begin{lemma}
$\left( \mathrm{a}\right) $ The functions $\omega _{\sigma ,T}$ and $%
\widetilde{\omega }_{\sigma ,T}$ are cohomologous and $\widetilde{\omega }%
_{\sigma ,T}$ is normalized.

$\left( \mathrm{b}\right) $ The map $t\rightarrow \widetilde{\mathcal{W}}%
_{\sigma ,T}\left( t\xi \right) $ is a group representation of $\mathbb{R}$
and for each $\xi \in W$ there is a unique self-adjoint operator $\phi _{{T}%
}\left( \xi \right) $, $T$-field operator associated to $\xi $, such that 
\begin{equation*}
\widetilde{\mathcal{W}}_{\sigma ,T}\left( t\xi \right) =\mathrm{e}^{\mathrm{i%
}t\phi _{{T}}\left( \xi \right) }
\end{equation*}%
for all real $t$.

$\left( \mathrm{c}\right) $ The map 
\begin{eqnarray*}
\sigma _{T+T^{{\sigma }}} &:&W\times W\rightarrow \mathbb{R}, \\
\sigma _{T+T^{{\sigma }}}\left( \xi ,\eta \right) &=&\sigma \left( \xi
,T\eta \right) -\sigma \left( \eta ,T\xi \right) \\
&=&\sigma \left( \xi ,\left( T+T^{\sigma }\right) \eta \right) ,\quad \xi
,\eta \in W,
\end{eqnarray*}%
is a bilinear antisymmetric $2$-form. $\sigma _{T+T^{{\sigma }}}$ is
symplectic $($i.e. $\sigma _{T+T^{{\sigma }}}$ is non-degenerate$)$ if and
only if $T+T^{\sigma }$ is an isomorphism $(\Leftrightarrow $ $\omega
_{\sigma ,T}$ is a non-degenerate Schur multiplier$)$.

$\left( \mathrm{d}\right) $ If $T+T^{\sigma }$ is an isomorphism $%
(\Leftrightarrow $ $\omega _{\sigma ,T}$ is a non-degenerate Schur multiplier%
$)$, then $\sigma _{T+T^{{\sigma }}}$ is a symplectic form and in this case 
\begin{equation*}
\left\{ \widetilde{\mathcal{W}}_{\sigma ,T}\left( \xi \right) \right\} _{\xi
\in W}\equiv \left( \mathcal{H},\widetilde{\mathcal{W}}_{\sigma ,T},%
\widetilde{\omega }_{\sigma ,T}\right)
\end{equation*}%
is a Weyl system for the symplectic space $\left( W,\sigma _{T+T^{{\sigma }%
}}\right) $, i.e.,%
\begin{equation*}
\widetilde{\mathcal{W}}_{\sigma ,T}\left( \xi \right) \widetilde{\mathcal{W}}%
_{\sigma ,T}\left( \eta \right) =\mathrm{e}^{\frac{\mathrm{i}}{2}\sigma
_{T+T^{{\sigma }}}\left( \xi ,\eta \right) }\widetilde{\mathcal{W}}_{\sigma
,T}\left( \xi +\eta \right) ,\quad \xi ,\eta \in W.
\end{equation*}%
Also, in this case, the map $\xi \rightarrow \phi _{T}\left( \xi \right) $
is $\mathbb{R}$-linear.

$\left( \mathrm{e}\right) $ The projective representations $\left( \mathcal{H%
},\mathcal{W}_{\sigma ,T},\omega _{\sigma ,T}\right) $ and $\left( \mathcal{H%
},\widetilde{\mathcal{W}}_{\sigma ,T},\widetilde{\omega }_{\sigma ,T}\right) 
$ satisfy 
\begin{equation*}
\widetilde{\mathcal{W}}_{\sigma ,T}\left( \xi \right) =\mathrm{e}^{\frac{%
\mathrm{i}}{2}\sigma \left( \xi ,T\xi \right) }\mathcal{W}_{\sigma ,T}\left(
\xi \right) ,\quad \xi \in W,
\end{equation*}%
\begin{equation*}
\mathcal{W}_{\sigma ,T}\left( \xi \right) =\mathrm{e}^{-\frac{\mathrm{i}}{2}%
\sigma \left( \xi ,T\xi \right) }\widetilde{\mathcal{W}}_{\sigma ,T}\left(
\xi \right) =\mathrm{e}^{\frac{\mathrm{i}}{2}\sigma \left( T\xi ,\xi \right)
}\widetilde{\mathcal{W}}_{\sigma ,T}\left( \xi \right) =\mathrm{e}^{\frac{%
\mathrm{i}}{2}\sigma \left( \xi ,T^{\sigma }\xi \right) }\widetilde{\mathcal{%
W}}_{\sigma ,T}\left( \xi \right) \quad \xi \in W.
\end{equation*}%
In particular, 
\begin{equation*}
\left\{ \mathcal{W}_{\sigma ,T}\left( \xi \right) :\xi \in W\right\}
^{\prime }=\left\{ \widetilde{\mathcal{W}}_{\sigma ,T}\left( \xi \right)
:\xi \in W\right\} ^{\prime }.
\end{equation*}
\end{lemma}

\begin{proof}
As noted before, almost all the statements have been proven. The rest is a
simple interpretation of the definitions except the $\mathbb{R}$-limearity
of the map $\xi \rightarrow \phi _{T}\left( \xi \right) $ which is a
consequence of the fact that $\left( \mathcal{H},\widetilde{\mathcal{W}}%
_{\sigma ,T},\widetilde{\omega }_{\sigma ,T}\right) $ is a Weyl system.
\end{proof}

From now on, we shall always assume that $T+T^{\sigma }$ is an isomorphism.

\begin{corollary}
$\left( \mathcal{H},\mathcal{W}_{\sigma ,T},\omega _{\sigma ,T}\right) $ is
an irreducible $\omega _{\sigma ,T}$-representation if and only if $\left( 
\mathcal{H},\widetilde{\mathcal{W}}_{\sigma ,T},\widetilde{\omega }_{\sigma
,T},\left( W,\sigma _{T+T^{{\sigma }}}\right) \right) $ is an irreducible
projective representation.
\end{corollary}

\begin{corollary}
Any two irreducible $\omega _{\sigma ,T}$-representations are unitary
equivalent.
\end{corollary}

\begin{corollary}
Suppose that $\left( \mathcal{H},\mathcal{W},\omega _{\sigma ,T}\right) $ is
an irreducible $\omega _{\sigma ,T}$-representation and $S:W\rightarrow W$
is a linear map.

$\left( \mathrm{a}\right) $ If $S$ satisfies $S^{\sigma }\left( T+T^{\sigma
}\right) S=\left( T+T^{\sigma }\right) $, then there is a unitary
transformation $U$ in $\mathcal{H}$,uniquely determined apart from a
constant factor of modulus $1$, such that 
\begin{equation*}
\mathcal{W\circ }S\left( \xi \right) =\mu \left( \xi \right) U^{-1}\mathcal{W%
}\left( \xi \right) U,\quad \xi \in W,
\end{equation*}%
where 
\begin{equation*}
\mu \left( \xi \right) =\mathrm{e}^{\frac{\mathrm{i}}{2}\left[ \sigma \left(
\xi ,T\xi \right) -\sigma \left( S\xi ,TS\xi \right) \right] },\quad \xi \in
W.
\end{equation*}

$\left( \mathrm{b}\right) $ If $S$ satisfies $S^{\sigma }TS=T$, then there
is a unitary transformation $U$ in $\mathcal{H}$,uniquely determined apart
from a constant factor of modulus $1$, such that 
\begin{equation*}
\mathcal{W\circ }S\left( \xi \right) =U^{-1}\mathcal{W}\left( \xi \right)
U,\quad \xi \in W.
\end{equation*}
\end{corollary}

\begin{proof}
The hypothesis $S^{\sigma }\left( T+T^{\sigma }\right) S=\left( T+T^{\sigma
}\right) $ is equivalent to the fact that $S$ is a symplectic transformation
in $\left( W,\sigma _{T+T^{{\sigma }}}\right) $. Now $\left( \mathrm{a}%
\right) $ follows from Segal's theorem, Theorem 18.5.9 in \cite{Hormander3},
and the previous lemma, and $\left( \mathrm{b}\right) $ is a consequence of $%
\left( \mathrm{a}\right) $.
\end{proof}

\begin{example}
In $\left( \theta ,\tau \right) $-quantization we have $W=V\times V^{\ast }$
with $\sigma :\left( V\times V^{\ast }\right) \times \left( V\times V^{\ast
}\right) \rightarrow \mathbb{R}$ is the canonical symplectic form 
\begin{equation*}
\sigma :W\times W\rightarrow \mathbb{R},\quad \sigma \left( \left(
x,p\right) ;\left( x^{\prime },p^{\prime }\right) \right) =\left\langle
x^{\prime },p\right\rangle _{V,V^{\ast }}-\left\langle x,p^{\prime
}\right\rangle _{V,V^{\ast }}.
\end{equation*}%
In this case 
\begin{equation*}
T=\left( 
\begin{array}{cc}
\tau & 0 \\ 
0 & \theta ^{\ast }%
\end{array}%
\right) :W=V\times V^{\ast }\rightarrow W=V\times V^{\ast },
\end{equation*}%
and the condition $T+T^{\sigma }$ is an isomorphism $(\Leftrightarrow $ $%
T+T^{\sigma }$ is invertible$)$ is equivalent to $\tau +\theta $ is
invertible.

More generally, if%
\begin{equation*}
T=\left( 
\begin{array}{cc}
\tau & \alpha \\ 
\beta & \theta ^{\ast }%
\end{array}%
\right) :W=V\times V^{\ast }\rightarrow W=V\times V^{\ast },
\end{equation*}%
then 
\begin{equation*}
T^{\sigma }=\left( 
\begin{array}{cc}
\theta & -{}\alpha ^{\ast } \\ 
-{}\beta ^{\ast } & \tau ^{\ast }%
\end{array}%
\right) :W=V\times V^{\ast }\rightarrow W=V\times V^{\ast },
\end{equation*}%
and the condition $T+T^{\sigma }$ is iar invertible is equivalent to 
\begin{equation*}
\left( 
\begin{array}{cc}
\tau +\theta & \alpha -{}\alpha ^{\ast } \\ 
\beta -{}\beta ^{\ast } & \left( \tau +\theta \right) ^{\ast }%
\end{array}%
\right) \quad \text{is invertible.}
\end{equation*}
\end{example}

\section{$T$-Weyl calculus}

Let $\mathcal{S}$ be the dense linear subspace of $\mathcal{H}$ consisting
of the $\mathcal{C}^{\infty }$ vectors of the representation $\mathcal{W}%
_{\sigma ,T}$%
\begin{equation*}
\mathcal{S}=\mathcal{S}(\mathcal{H},\mathcal{W}_{\sigma ,T})=\left\{ \varphi
\in \mathcal{H}:W\ni \xi \rightarrow \mathcal{W}_{\sigma ,T}(\xi )\varphi
\in \mathcal{H}\text{\textit{\ is a }}\mathcal{C}^{\infty }\text{\textit{\
map}}\right\} .
\end{equation*}%
Since 
\begin{equation*}
\widetilde{\mathcal{W}}_{\sigma ,T}\left( \xi \right) =\mathrm{e}^{\frac{%
\mathrm{i}}{2}\sigma \left( \xi ,T\xi \right) }\mathcal{W}_{\sigma ,T}\left(
\xi \right) ,\quad \xi \in W,
\end{equation*}%
it follows that $\mathcal{S}$ is also the space of $\mathcal{C}^{\infty }$
vectors of the representation $\widetilde{\mathcal{W}}_{\sigma ,T}$%
\begin{equation*}
\mathcal{S}=\mathcal{S}(\mathcal{H},\widetilde{\mathcal{W}}_{\sigma
,T})=\left\{ \varphi \in \mathcal{H}:W\ni \xi \rightarrow \widetilde{%
\mathcal{W}}_{\sigma ,T}(\xi )\varphi \in \mathcal{H}\text{\textit{\ is a }}%
\mathcal{C}^{\infty }\text{\textit{\ map}}\right\} .
\end{equation*}%
The space $\mathcal{S}$ can be described in terms of the subspaces $D\left(
\phi _{{T}}\left( \xi \right) \right) $, where $\phi _{{T}}\left( \xi
\right) $ is $T$-field operator associated to $\xi $, 
\begin{eqnarray*}
\mathcal{S} &=&\bigcap_{k\in \mathbb{N}}\bigcap_{\xi _{1},...,\xi _{k}\in
W}D\left( \phi _{{T}}\left( \xi _{1}\right) ...\phi _{{T}}\left( \xi
_{k}\right) \right) \\
&=&\bigcap_{k\in \mathbb{N}}\bigcap_{\xi _{1},...,\xi _{k}\in \mathcal{B}%
}D\left( \phi _{{T}}\left( \xi _{1}\right) ...\phi _{{T}}\left( \xi
_{k}\right) \right) ,
\end{eqnarray*}%
where $\mathcal{B}$ is a (symplectic) basis. The topology in $\mathcal{S}$
defined by the family of seminorms $\left\{ \left\Vert \cdot \right\Vert
_{k,\xi _{1},...,\xi _{k}}\right\} _{k\in \mathbb{N},\xi _{1},...,\xi
_{k}\in W}$ 
\begin{equation*}
\left\Vert \varphi \right\Vert _{k,\xi _{1},...,\xi _{k}}=\left\Vert \phi _{{%
T}}\left( \xi _{1}\right) ...\phi _{{T}}\left( \xi _{1}\right) \varphi
\right\Vert _{\mathcal{H}},\quad \varphi \in \mathcal{S}
\end{equation*}%
makes $\mathcal{S}$ a Fr\'{e}chet space. We denote by $\mathcal{S}^{\ast }=%
\mathcal{S}\left( \mathcal{H},\mathcal{W}_{\sigma ,T}\right) ^{\ast }$ the
space of all continuous, antilinear (semilinear) mappings $\mathcal{S}%
\rightarrow \mathbb{\mathbb{C}}$ equipped with the weak topology $\sigma (%
\mathcal{S}^{\ast },\mathcal{S})$. Since $\mathcal{S}\hookrightarrow 
\mathcal{H}$ continuously and densely, and since $\mathcal{H}$ is always
identified with its adjoint $\mathcal{H}^{\ast }$, we obtain a scale of
dense inclusions%
\begin{equation*}
\mathcal{S}\hookrightarrow \mathcal{H}\hookrightarrow \mathcal{S}^{\ast }
\end{equation*}%
such that, if $\left\langle \cdot ,\cdot \right\rangle :\mathcal{S}\times 
\mathcal{S}^{\ast }\rightarrow \mathbb{\mathbb{C}}$ is the antiduality
between $\mathcal{S}$ and $\mathcal{S}^{\ast }$ (antilinear in the first and
linear in the second argument), then for $\varphi \in \mathcal{S}$ and $u\in 
\mathcal{H}$, if $u$ is considered as an element of $\mathcal{S}^{\ast }$,
the number $\left\langle \varphi ,u\right\rangle $ is just the scalar
product in $\mathcal{H}$. For this reason we do not distinguish between the
the scalar product in $\mathcal{H}$ and the antiduality between $\mathcal{S}$
and $\mathcal{S}^{\ast }$. See p. 83-85 in \cite{Arsu}.

\begin{lemma}
\label{n2}If $\varphi ,\psi \in \mathcal{S}$, then the map $W\ni \xi
\rightarrow \left\langle \psi ,\mathcal{W}_{\sigma ,T}\left( \xi \right)
\varphi \right\rangle _{\mathcal{S},\mathcal{S}^{\ast }}\in \mathbb{C}$
belongs to $\mathcal{S}\left( W\right) $. Moreover, for each continuous
seminorm $p$ on $\mathcal{S}\left( W\right) $ there are continuous seminorms 
$q$ and $q^{\prime }$ on $\mathcal{S}$ such that 
\begin{equation*}
p\left( \left\langle \psi ,\mathcal{W}_{\sigma ,T}\left( \cdot \right)
\varphi \right\rangle _{\mathcal{S},\mathcal{S}^{\ast }}\right) \leq q(\psi
)q^{\prime }(\varphi ).
\end{equation*}
\end{lemma}

\begin{proof}
For a proof see Lemma 1.1 in \cite{Arsu}.
\end{proof}

We use the symplectic Fourier decomposition of $a$, 
\begin{equation*}
a=\int \mathrm{e}^{-\mathrm{i}\sigma \left( \cdot ,\xi \right) }\mathcal{F}%
_{\sigma }\left( a\right) \left( \xi \right) \mathrm{d}^{\sigma }\xi
\end{equation*}%
to introduce the following operator in $\mathcal{B}\left( \mathcal{\mathcal{H%
}}\right) $ 
\begin{equation*}
\mathrm{Op}_{\sigma ,T}\left( a\right) =\int \mathcal{F}_{\sigma }\left(
a\right) \left( \xi \right) \mathcal{W}_{\sigma ,T}\left( \xi \right) 
\mathrm{d}^{\sigma }\xi ,\quad a\in \mathcal{S}\left( W\right) .
\end{equation*}%
Lemma \ref{n2} allows us to use the same formula to define the $T$-Weyl
calculus.

\begin{definition}[$T$-Weyl calculus]
Let $\left( \mathcal{H},\mathcal{W}_{\sigma ,T},\omega _{\sigma ,T}\right) $
be a $\omega _{\sigma ,T}$-representation of $W$, $\mathcal{S}=\mathcal{S}(%
\mathcal{H},\mathcal{W}_{\sigma ,T})$ and $\mathcal{S}^{\ast }=\mathcal{S}%
\left( \mathcal{H},\mathcal{W}_{\sigma ,T}\right) ^{\ast }$. Then for each $%
\mu $, $a\in \mathcal{S}^{\ast }\left( W\right) $ we can define the
operators 
\begin{equation*}
\mathcal{W}_{\sigma ,T}(\mu ):\mathcal{S}\rightarrow \mathcal{S}^{\ast
},\quad \mathrm{Op}_{\sigma ,T}\left( a\right) :\mathcal{S}\rightarrow 
\mathcal{S}^{\ast },
\end{equation*}%
by 
\begin{equation*}
\mathcal{W}_{\sigma ,T}\left( \mu \right) =\int_{W}\mathcal{W}_{\sigma
,T}\left( \xi \right) \mu \left( \xi \right) \mathrm{d}^{\sigma }\xi ,\quad 
\mathrm{Op}_{\sigma ,T}\left( a\right) =\int_{W}\mathcal{W}_{\sigma
,T}\left( \xi \right) \mathcal{F}_{\sigma }\left( a\right) \left( \xi
\right) \mathrm{d}^{\sigma }\xi .
\end{equation*}%
The above integrals make sense if they are taken in the weak sense, i.e. for 
$\varphi $, $\psi \in \mathcal{S}$%
\begin{eqnarray*}
\left\langle \varphi ,\mathcal{W}_{\sigma ,T}\left( \mu \right) \psi
\right\rangle _{\mathcal{S},\mathcal{S}^{\ast }} &=&\left\langle \overline{%
\left\langle \varphi ,\mathcal{W}_{\sigma ,T}\left( \cdot \right) \psi
\right\rangle }_{\mathcal{S},\mathcal{S}^{\ast }},\mu \right\rangle _{%
\mathcal{S}\left( W\right) ,\mathcal{S}^{\ast }\left( W\right) }, \\
\left\langle \varphi ,\mathrm{Op}_{\sigma ,T}\left( a\right) \psi
\right\rangle _{\mathcal{S},\mathcal{S}^{\ast }} &=&\left\langle \overline{%
\left\langle \varphi ,\mathcal{W}_{\sigma ,T}\left( \cdot \right) \psi
\right\rangle }_{\mathcal{S},\mathcal{S}^{\ast }},\mathcal{F}_{\sigma
}\left( a\right) \right\rangle _{\mathcal{S}\left( W\right) ,\mathcal{S}%
^{\ast }\left( W\right) }.
\end{eqnarray*}%
Moreover, from Lemma \ref{n2} one obtains that 
\begin{eqnarray*}
\left\vert \left\langle \varphi ,\mathcal{W}_{\sigma ,T}\left( \mu \right)
\psi \right\rangle _{\mathcal{S},\mathcal{S}^{\ast }}\right\vert +\left\vert
\left\langle \varphi ,\mathrm{Op}_{\sigma ,T}\left( a\right) \psi
\right\rangle _{\mathcal{S},\mathcal{S}^{\ast }}\right\vert &\leq &p\left(
\left\langle \varphi ,\mathcal{W}_{\sigma ,T}\left( \cdot \right) \psi
\right\rangle _{\mathcal{S},\mathcal{S}^{\ast }}\right) \\
&\leq &q\left( \varphi \right) q^{\prime }\left( \psi \right) ,
\end{eqnarray*}%
where $p$ is a continuous seminorm on $\mathcal{S}(W)$ and $q$ and $%
q^{\prime }$ are continuous seminorms on $\mathcal{S}$.

If we consider on $\mathcal{S}^{\ast }\left( W\right) $ the weak$^{\ast }$
topology $\sigma (\mathcal{S}^{\ast }\left( W\right) ,\mathcal{S}\left(
W\right) )$ and on $\mathcal{B}(\mathcal{S},\mathcal{S}^{\ast })$, the
topology defined by the seminorms $\left\{ p_{\varphi ,\psi }\right\}
_{\varphi ,\psi \in \mathcal{S}}$,%
\begin{equation*}
p_{\varphi ,\psi }\left( A\right) =\left\vert \left\langle \varphi ,A\psi
\right\rangle \right\vert ,\quad A\in \mathcal{B}(\mathcal{S},\mathcal{S}%
^{\ast }),
\end{equation*}%
the mappings%
\begin{eqnarray*}
\mathcal{W}_{\sigma ,T} &:&\mathcal{S}^{\ast }\left( W\right) \rightarrow 
\mathcal{B}\left( \mathcal{S},\mathcal{S}^{\ast }\right) ,\quad \mu
\rightarrow \mathcal{W}_{\sigma ,T}(\mu ), \\
\mathrm{Op}_{\sigma ,T} &:&\mathcal{S}^{\ast }\left( W\right) \rightarrow 
\mathcal{B}\left( \mathcal{S},\mathcal{S}^{\ast }\right) ,\quad a\rightarrow 
\mathrm{Op}_{\sigma ,T}(a),
\end{eqnarray*}%
are well defined linear and continuous.
\end{definition}

On $W$ we have two symplectic structures, the first, the initial one, is
given by the $2$-form $\sigma $, and the second one obtained in the
normalization process of the factor $\omega _{\sigma ,T}$ ($\omega _{\sigma
,T}$ $\rightsquigarrow $ $\widetilde{\omega }_{\sigma ,T}$), namely the
structure given by the $2$-form form $\sigma _{T+T^{{\sigma }}}$.
Accordingly, we have two symplectic Fourier transformations $\mathcal{F}%
_{\sigma }$ and $\mathcal{F}_{\sigma _{T+T^{{\sigma }}}}$. Their connection
will be established below.

\begin{lemma}
If $b\in \mathcal{S}^{\prime }\left( W\right) $, then 
\begin{equation*}
\mathcal{F}_{\sigma }\left( b\right) =\left( \det \left( T+T^{\sigma
}\right) \right) ^{\frac{1}{2}}\mathcal{F}_{\sigma _{T+T^{{\sigma }}}}\left(
b\circ \left( T+T^{\sigma }\right) \right) ,
\end{equation*}%
or in operator form,%
\begin{equation*}
\mathcal{F}_{\sigma }=\left( \det \left( T+T^{\sigma }\right) \right) ^{%
\frac{1}{2}}\mathcal{F}_{\sigma _{T+T^{{\sigma }}}}\circ \left( T+T^{\sigma
}\right) ^{\ast }\quad \text{on }\mathcal{S}^{\prime }\left( W\right) .
\end{equation*}
\end{lemma}

\begin{proof}
We take $S=T+T^{\sigma }$ in Lemma \ref{n3}.
\end{proof}

We have two projective representations $\left( \mathcal{H},\mathcal{W}%
_{\sigma ,T},\omega _{\sigma ,T}\right) $ and $\left( \mathcal{H},\widetilde{%
\mathcal{W}}_{\sigma ,T},\widetilde{\omega }_{\sigma ,T}\right) $ with $%
\omega _{\sigma ,T}$ and $\widetilde{\omega }_{\sigma ,T}$ cohomologous
associated multipliers, two symplectic structures $\left( W,\sigma \right) $
and $\left( W,\sigma _{T+T^{{\sigma }}}\right) $ with two associated
symplectic Fourier transformations $\mathcal{F}_{\sigma }$ and $\mathcal{F}%
_{\sigma _{T+T^{{\sigma }}}}$. Accordingly, we have the $T$-Weyl calculus, $%
\mathrm{Op}_{\sigma ,T}$, corresponding to the $\omega _{\sigma ,T}$%
-representation $\left( \mathcal{H},\mathcal{W}_{\sigma ,T},\omega _{\sigma
,T}\right) $ and the standard Weyl calculus, $\mathrm{Op}^{\widetilde{w}}$,
corresponding to the Weyl system $\left( \mathcal{H},\widetilde{\mathcal{W}}%
_{\sigma ,T},\widetilde{\omega }_{\sigma ,T}\right) $ (for the symplectic
space $\left( W,\sigma _{T+T^{{\sigma }}}\right) $). The connection between
these two calculus will be established hereafter.

We write $\theta _{{\sigma ,T}}$ for the quadratic form on $W$ given by%
\begin{equation*}
\theta _{{\sigma ,T}}\left( \xi \right) =\sigma \left( \xi ,T\xi \right)
,\quad \xi \in W,
\end{equation*}%
with the associated symmetric bilinear form $\beta _{{\sigma ,T}}:W\times
W\longrightarrow \mathbb{R}$ defined by 
\begin{equation*}
\beta _{{\sigma ,T}}\left( \xi ,\eta \right) =\frac{1}{2}\left[ \sigma
\left( \xi ,T\eta \right) +\sigma \left( \eta ,T\xi \right) \right] =\frac{1%
}{2}\sigma \left( \xi ,\left( T-T^{\sigma }\right) \eta \right) ,\quad \xi
,\eta \in W.
\end{equation*}%
We write $\lambda _{{\sigma ,T}}$ for the function in $\mathcal{C}_{\mathrm{%
pol}}^{\infty }\left( W\right) $ given by 
\begin{equation*}
\lambda _{{\sigma ,T}}\left( \xi \right) =\mathrm{e}^{-\frac{\mathrm{i}}{2}%
\theta _{{\sigma ,T}}\left( \xi \right) }=\mathrm{e}^{-\frac{\mathrm{i}}{2}%
\sigma \left( \xi ,T\xi \right) },\quad \xi \in W.
\end{equation*}%
To this function, we associate the convolution operator $\lambda _{{\sigma ,T%
}}\left( D_{\sigma }\right) $ defined by using the symplectic Fourier
transformation, i.e. 
\begin{equation*}
\lambda _{{\sigma ,T}}\left( D_{\sigma }\right) =\mathcal{F}_{\sigma }\circ 
\mathrm{M}_{\lambda _{{\sigma ,T}}\left( \cdot \right) }\circ \mathcal{F}%
_{\sigma }:\mathcal{S}^{\prime }\left( W\right) \longrightarrow \mathcal{S}%
^{\prime }\left( W\right) ,
\end{equation*}%
where $\mathrm{M}_{\lambda _{{\sigma ,T}}\left( \cdot \right) }$ denotes the
multiplication operator by the function $\lambda _{{\sigma ,T}}\left( \cdot
\right) $. For this operator we shall also use the notation $\mathrm{e}^{-%
\frac{\mathrm{i}}{2}\theta _{{\sigma ,T}}\left( D_{\sigma }\right) }$, i.e.%
\begin{equation*}
\lambda _{{\sigma ,T}}\left( D_{\sigma }\right) =\mathrm{e}^{-\frac{\mathrm{i%
}}{2}\theta _{{\sigma ,T}}\left( D_{\sigma }\right) }.
\end{equation*}

For $a$ in $\mathcal{S}^{\prime }\left( W\right) $, we shall denote by $%
a_{\sigma ,T}^{w}$ the temperate distribution in $\mathcal{S}^{\prime
}\left( W\right) $ given by%
\begin{equation*}
a_{{\sigma ,T}}^{w}=\lambda _{{\sigma ,T}}\left( D_{\sigma }\right) \left(
a\right) \circ \left( T+T^{\sigma }\right) =\left( \lambda _{{\sigma ,T}%
}\circ \left( T+T^{\sigma }\right) ^{-1}\right) \left( D_{\sigma }\right)
\left( a\circ \left( T+T^{\sigma }\right) \right) .
\end{equation*}%
If $a\in \mathcal{S}\left( W\right) $, then we have%
\begin{multline*}
\mathrm{Op}_{\sigma ,T}\left( a\right) =\int \mathcal{F}_{\sigma }\left(
a\right) \left( \xi \right) \mathcal{W}_{\sigma ,T}\left( \xi \right) 
\mathrm{d}^{\sigma }\xi =\int \mathrm{e}^{-\frac{\mathrm{i}}{2}\sigma \left(
\xi ,T\xi \right) }\mathcal{F}_{\sigma }\left( a\right) \left( \xi \right) 
\widetilde{\mathcal{W}}_{\sigma ,T}\left( \xi \right) \mathrm{d}^{\sigma }\xi
\\
=\int \lambda _{{\sigma ,T}}\left( \xi \right) \mathcal{F}_{\sigma }\left(
a\right) \left( \xi \right) \widetilde{\mathcal{W}}_{\sigma ,T}\left( \xi
\right) \mathrm{d}^{\sigma }\xi =\int \mathcal{F}_{\sigma }\left( \lambda _{{%
\sigma ,T}}\left( D_{\sigma }\right) \left( a\right) \right) \left( \xi
\right) \widetilde{\mathcal{W}}_{\sigma ,T}\left( \xi \right) \mathrm{d}%
^{\sigma }\xi \\
=\int \left( \det \left( T+T^{\sigma }\right) \right) ^{\frac{1}{2}}\mathcal{%
F}_{\sigma _{T+T^{{\sigma }}}}\left( \lambda _{{\sigma ,T}}\left( D_{\sigma
}\right) \left( a\right) \circ \left( T+T^{\sigma }\right) \right) \left(
\xi \right) \widetilde{\mathcal{W}}_{\sigma ,T}\left( \xi \right) \mathrm{d}%
^{\sigma }\xi \\
=\int \left( \det \left( T+T^{\sigma }\right) \right) ^{\frac{1}{2}}\mathcal{%
F}_{\sigma _{T+T^{{\sigma }}}}\left( a_{{\sigma ,T}}^{w}\right) \left( \xi
\right) \widetilde{\mathcal{W}}_{\sigma ,T}\left( \xi \right) \mathrm{d}%
^{\sigma }\xi \\
=\int \mathcal{F}_{\sigma _{T+T^{{\sigma }}}}\left( a_{{\sigma ,T}%
}^{w}\right) \left( \xi \right) \widetilde{\mathcal{W}}_{\sigma ,T}\left(
\xi \right) \mathrm{d}^{\sigma _{T+T^{{\sigma }}}}\xi =\mathrm{Op}^{%
\widetilde{w}}\left( a_{{\sigma ,T}}^{w}\right) .
\end{multline*}%
Hence%
\begin{equation*}
\mathrm{Op}_{\sigma ,T}\left( a\right) =\mathrm{Op}^{\widetilde{w}}\left(
a_{\sigma ,T}^{w}\right) ,\quad a\in \mathcal{S}\left( W\right) .
\end{equation*}%
Here we used the equality 
\begin{equation*}
\mathrm{d}^{\sigma _{T+T^{{\sigma }}}}\xi =\left( \det \left( T+T^{\sigma
}\right) \right) ^{\frac{1}{2}}\mathrm{d}^{\sigma }\eta .
\end{equation*}%
By continuity and density arguments we find that 
\begin{equation*}
\mathrm{Op}_{\sigma ,T}\left( a\right) =\mathrm{Op}^{\widetilde{w}}\left(
a_{\sigma ,T}^{w}\right) ,\quad a\in \mathcal{S}^{\prime }\left( W\right) ,
\end{equation*}%
with%
\begin{align*}
a_{{\sigma ,T}}^{w}& =\lambda _{{\sigma ,T}}\left( D_{\sigma }\right) \left(
a\right) \circ \left( T+T^{\sigma }\right) =\mathrm{e}^{-\frac{\mathrm{i}}{2}%
\theta _{{\sigma ,T}}\left( D_{\sigma }\right) }\left( a\right) \circ \left(
T+T^{\sigma }\right) \\
& =\left( \lambda _{{\sigma ,T}}\circ \left( T+T^{\sigma }\right)
^{-1}\right) \left( D_{\sigma }\right) \left( a\circ \left( T+T^{\sigma
}\right) \right) \\
& =\lambda _{{\sigma ,T}}\left( D_{\sigma _{T+T^{{\sigma }}}}\right) \left(
a\circ \left( T+T^{\sigma }\right) \right) \quad \text{by Lemma \ref{n11}} \\
& =\mathrm{e}^{-\frac{\mathrm{i}}{2}\theta _{{\sigma ,T}}\left( D_{\sigma
_{T+T^{{\sigma }}}}\right) }\left( a\circ \left( T+T^{\sigma }\right)
\right) \text{.}
\end{align*}%
Thus we have proved one of the important results of the paper.

\begin{theorem}
\label{n4}Let $\left( \mathcal{H},\mathcal{W}_{\sigma ,T},\omega _{\sigma
,T}\right) $ be a $\omega _{\sigma ,T}$-representation of $W$, the $T$-Weyl
calculus, $\mathrm{Op}_{\sigma ,T}$, corresponding to this representation
and the standard Weyl calculus, $\mathrm{Op}^{\widetilde{w}}$, corresponding
to the associated Weyl system $\left( \mathcal{H},\widetilde{\mathcal{W}}%
_{\sigma ,T},\widetilde{\omega }_{\sigma ,T}\right) $ $($for the symplectic
space $\left( W,\sigma _{T+T^{{\sigma }}}\right) )$. Then for every $a\in 
\mathcal{S}^{\prime }\left( W\right) $ 
\begin{equation*}
\mathrm{Op}_{\sigma ,T}\left( a\right) =\mathrm{Op}^{\widetilde{w}}\left(
a_{\sigma ,T}^{w}\right) ,
\end{equation*}%
where%
\begin{align*}
a_{{\sigma ,T}}^{w}& =\left( T+T^{\sigma }\right) ^{\ast }\left( \lambda _{{%
\sigma ,T}}\left( D_{\sigma }\right) \left( a\right) \right) =\left(
T+T^{\sigma }\right) ^{\ast }\left( \mathrm{e}^{-\frac{\mathrm{i}}{2}\theta
_{{\sigma ,T}}\left( D_{\sigma }\right) }\left( a\right) \right) \\
& =\left( \lambda _{{\sigma ,T}}\circ \left( T+T^{\sigma }\right)
^{-1}\right) \left( D_{\sigma }\right) \left( \left( T+T^{\sigma }\right)
^{\ast }\left( a\right) \right) \\
& =\lambda _{{\sigma ,T}}\left( D_{\sigma _{T+T^{{\sigma }}}}\right) \left(
\left( T+T^{\sigma }\right) ^{\ast }\left( a\right) \right) \\
& =\mathrm{e}^{-\frac{\mathrm{i}}{2}\theta _{{\sigma ,T}}\left( D_{\sigma
_{T+T^{{\sigma }}}}\right) }\left( a\circ \left( T+T^{\sigma }\right)
\right) \in \mathcal{S}^{\prime }\left( W\right) .
\end{align*}
\end{theorem}

The equation in $a\in \mathcal{S}^{\prime }\left( W\right) $, 
\begin{equation*}
a_{{\sigma ,T}}^{w}=\lambda _{{\sigma ,T}}\left( D_{\sigma }\right) \left(
a\right) \circ \left( T+T^{\sigma }\right) =\mathrm{e}^{-\frac{\mathrm{i}}{2}%
\theta _{{\sigma ,T}}\left( D_{\sigma }\right) }\left( a\right) \circ \left(
T+T^{\sigma }\right) ,
\end{equation*}%
has a unique solution in $\mathcal{S}^{\prime }\left( W\right) $,%
\begin{equation*}
a=\overline{\lambda }_{{\sigma ,T}}\left( D_{\sigma }\right) \left( a_{{%
\sigma ,T}}^{w}\circ \left( T+T^{\sigma }\right) ^{-1}\right) =\mathrm{e}^{%
\frac{\mathrm{i}}{2}\theta _{{\sigma ,T}}\left( D_{\sigma }\right) }\left(
a_{{\sigma ,T}}^{w}\circ \left( T+T^{\sigma }\right) ^{-1}\right) ,
\end{equation*}%
where $\overline{\lambda }_{{\sigma ,T}}\left( \cdot \right) =1/\lambda _{{%
\sigma ,T}}\left( \cdot \right) =\mathrm{e}^{\frac{\mathrm{i}}{2}\theta _{{%
\sigma ,T}}\left( \cdot \right) }$. It follows that the correspondence at
the symbol level $\mathcal{S}^{\prime }\left( W\right) \ni a\rightarrow a_{{%
\sigma ,T}}^{w}\in \mathcal{S}^{\prime }\left( W\right) $ is bijective and
obviously continuous.

Next, we assume that the $\omega _{\sigma ,T}$-representation $\left( 
\mathcal{H},\mathcal{W}_{\sigma ,T},\omega _{\sigma ,T}\right) $ is
irreducible, which is equivalent to $\left( \mathcal{H},\widetilde{\mathcal{W%
}}_{\sigma ,T},\widetilde{\omega }_{\sigma ,T}\right) $ is an irreducible
Weyl system. In this case, $\left( \mathcal{H},\widetilde{\mathcal{W}}%
_{\sigma ,T},\widetilde{\omega }_{\sigma ,T}\right) $ is unitary equivalent
to the Schr\"{o}dinger representation. It follows that the map%
\begin{equation*}
\mathrm{Op}^{\widetilde{w}}:\mathcal{S}^{\ast }\left( W\right) \rightarrow 
\mathcal{B}\left( \mathcal{S},\mathcal{S}^{\ast }\right) ,\quad a\rightarrow 
\mathrm{Op}^{\widetilde{w}}(a),
\end{equation*}%
it is linear, continuous and bijective, so the same is true for the map $%
\mathrm{Op}_{\sigma ,T}$.

\begin{proposition}
Let $\left( \mathcal{H},\mathcal{W}_{\sigma ,T},\omega _{\sigma ,T}\right) $
be an irreducible $\omega _{\sigma ,T}$-representation of $W$. Then the map 
\begin{equation*}
\mathrm{Op}_{\sigma ,T}:\mathcal{S}^{\ast }\left( W\right) \rightarrow 
\mathcal{B}\left( \mathcal{S},\mathcal{S}^{\ast }\right) ,\quad a\rightarrow 
\mathrm{Op}_{\sigma ,T}(a),
\end{equation*}%
it is linear, continuous and bijective.
\end{proposition}

From now on, we shall always assume that $\left( \mathcal{H},\mathcal{W}%
_{\sigma ,T},\omega _{\sigma ,T}\right) $ is an irreducible $\omega _{\sigma
,T}$-representation.

\begin{definition}
For each $A\in \mathcal{B}\left( \mathcal{S},\mathcal{S}^{\ast }\right) $
there is a unique distribution $a\in \mathcal{S}^{\ast }\left( W\right) $
such that $A=\mathrm{Op}_{\sigma ,T}(a)$. This distribution is called $T$%
-symbol of $A$. Likewise, for each $A\in \mathcal{B}\left( \mathcal{S},%
\mathcal{S}^{\ast }\right) $ there is a unique distribution $b\in \mathcal{S}%
^{\ast }\left( W\right) $ such that $A=\mathrm{Op}^{\widetilde{w}}\left(
b\right) $. This distribution is called the Weyl symbol of $A$.
\end{definition}

If $\mathrm{Op}_{\sigma ,T}(a)=A=\mathrm{Op}^{\widetilde{w}}\left( b\right) $%
, then we have 
\begin{equation*}
b=a_{{\sigma ,T}}^{w}=\left( \mathrm{Op}^{\widetilde{w}}\right) ^{-1}\mathrm{%
Op}_{\sigma ,T}(a)=\lambda _{{\sigma ,T}}\left( D_{\sigma }\right) \left(
a\right) \circ \left( T+T^{\sigma }\right) ,
\end{equation*}%
and%
\begin{equation*}
a=\left( \mathrm{Op}_{\sigma ,T}\right) ^{-1}\mathrm{Op}^{\widetilde{w}}(b)=%
\overline{\lambda }_{{\sigma ,T}}\left( D_{\sigma }\right) \left( b\circ
\left( T+T^{\sigma }\right) ^{-1}\right) .
\end{equation*}

\section{Modulation spaces and Schatten-class properties of operators in the 
$T$-Weyl calculus}

The importance of Theorem \ref{n4} lies in the fact that it establishes both
the connection between the $T$-Weyl calculus and the standard Weyl calculus,
as well as the connection that exists between the symbols used in them. The
maps,%
\begin{equation*}
\Lambda _{{\sigma ,T}}^{w}=\left( \mathrm{Op}^{\widetilde{w}}\right) ^{-1}%
\mathrm{Op}_{\sigma ,T}:\mathcal{S}^{\ast }\left( W\right) \rightarrow 
\mathcal{S}^{\ast }\left( W\right) ,\text{ }\Lambda _{{\sigma ,T}}^{w}\left(
a\right) =\lambda _{{\sigma ,T}}\left( D_{\sigma }\right) \left( a\right)
\circ \left( T+T^{\sigma }\right) ,
\end{equation*}%
\begin{equation*}
\left( \Lambda _{{\sigma ,T}}^{w}\right) ^{-1}:\mathcal{S}^{\ast }\left(
W\right) \rightarrow \mathcal{S}^{\ast }\left( W\right) ,\text{ }\left( \Phi
_{{\sigma ,T}}^{w}\right) ^{-1}(b)=\overline{\lambda }_{{\sigma ,T}}\left(
D_{\sigma }\right) \left( b\circ \left( T+T^{\sigma }\right) ^{-1}\right) ,
\end{equation*}%
are continuous linear isomorphisms. Clearly $\mathcal{S}\left( W\right) $ is
an invariant subspace for both maps. Other invariant subspaces for these
maps are particular cases of of modulation spaces.

Now we shall recall the definition of the classical modulation space $%
M^{p,q}\left( \mathbb{R}^{n}\right) $ with parameters $1\leq p,q\leq \infty $%
.

\begin{definition}
Let $1\leq p,q\leq \infty $. We say that a distribution $u\in \mathcal{D}%
^{\prime }\left( \mathbb{R}^{n}\right) $ belongs to $M^{p,q}\left( \mathbb{R}%
^{n}\right) $ if there is $\chi \in \mathcal{C}_{0}^{\infty }\left( \mathbb{R%
}^{n}\right) \smallsetminus 0$ such that the measurable function%
\begin{gather*}
U_{\chi ,p}:\mathbb{R}^{n}\rightarrow \left[ 0,+\infty \right] , \\
U_{\chi ,p}\left( \xi \right) =\left\{ 
\begin{array}{ccc}
\sup_{y\in \mathbb{R}^{n}}\left\vert \widehat{u\tau _{y}\chi }\left( \xi
\right) \right\vert & \text{\textit{if}} & p=\infty \\ 
\left( \int \left\vert \widehat{u\tau _{y}\chi }\left( \xi \right)
\right\vert ^{p}\mathrm{d}y\right) ^{1/p} & \text{\textit{if}} & 1\leq
p<\infty%
\end{array}%
\right. , \\
\widehat{u\tau _{y}\chi }\left( \xi \right) =\left\langle u,\mathrm{e}^{-%
\mathrm{i}\left\langle \cdot ,\xi \right\rangle }\chi \left( \cdot -y\right)
\right\rangle .
\end{gather*}%
belongs to $L^{q}\left( \mathbb{R}^{n}\right) $.
\end{definition}

These spaces are special cases of modulation spaces which were introduced by
Hans Georg Feichtinger \cite{Feichtinger 1} in 1983 (see also \cite%
{Feichtinger 3}). They were used by many authors (Boulkhemair, Gr\"{o}%
chenig, Heil, Sj\"{o}strand, Toft ...) in the analysis of
pseudo-differential operators defined by symbols more general than usual.

Now we give some properties of these spaces.

\begin{proposition}
$(\mathrm{a})$ Let $u\in M^{p,q}\left( \mathbb{R}^{n}\right) $ and let $\chi
\in \mathcal{C}_{0}^{\infty }\left( \mathbb{R}^{n}\right) $. Then the
measurable function%
\begin{gather*}
U_{\chi ,p}:\mathbb{R}^{n}\rightarrow \left[ 0,+\infty \right] , \\
U_{\chi ,p}\left( \xi \right) =\left\{ 
\begin{array}{ccc}
\sup_{y\in \mathbb{R}^{n}}\left\vert \widehat{u\tau _{y}\chi }\left( \xi
\right) \right\vert & \text{\textit{if}} & p=\infty \\ 
\left( \int \left\vert \widehat{u\tau _{y}\chi }\left( \xi \right)
\right\vert ^{p}\mathrm{d}y\right) ^{1/p} & \text{\textit{if}} & 1\leq
p<\infty%
\end{array}%
\right. , \\
\widehat{u\tau _{y}\chi }\left( \xi \right) =\left\langle u,\mathrm{e}^{-%
\mathrm{i}\left\langle \cdot ,\xi \right\rangle }\chi \left( \cdot -y\right)
\right\rangle .
\end{gather*}%
belongs to $L^{q}\left( \mathbb{R}^{n}\right) $.

$(\mathrm{b})$ If we fix $\chi \in \mathcal{C}_{0}^{\infty }\left( \mathbb{R}%
^{n}\right) \smallsetminus 0$ and if we put%
\begin{equation*}
\left\Vert u\right\Vert _{M^{p,q},\chi }=\left\Vert U_{\chi ,p}\right\Vert
_{L^{q}},\quad u\in M^{p,q}\left( \mathbb{R}^{n}\right) ,
\end{equation*}%
then $\left\Vert \cdot \right\Vert _{M^{p,q},\chi }$ is a norm on $%
M^{p,q}\left( \mathbb{R}^{n}\right) $ and the topology that defines does not
depend on the choice of the function $\chi \in \mathcal{C}_{0}^{\infty
}\left( \mathbb{R}^{n}\right) \smallsetminus 0$.

$(\mathrm{c})$ $M^{p,q}\left( \mathbb{R}^{n}\right) $ is a Banach space.

$\left( \mathrm{d}\right) $ If $1\leq p_{0}\leq p_{1}\leq \infty $ and $%
1\leq q_{0}\leq q_{1}\leq \infty $, then%
\begin{equation*}
\mathcal{S}\left( \mathbb{R}^{n}\right) \subset M^{1,1}\left( \mathbb{R}%
^{n}\right) \subset M^{p_{0},q_{0}}\left( \mathbb{R}^{n}\right) \subset
M^{p_{1},q_{1}}\left( \mathbb{R}^{n}\right) \subset M^{\infty ,\infty
}\left( \mathbb{R}^{n}\right) \subset \mathcal{S}^{\prime }\left( \mathbb{R}%
^{n}\right) .
\end{equation*}
\end{proposition}

To go further, we need a more convenient way to describe $M^{p,q}\left( 
\mathbb{R}^{n}\right) $'s topology.

\begin{lemma}
$(\mathrm{a})$ Let $\chi \in \mathcal{S}\left( \mathbb{R}^{n}\right) $ and $%
u\in \mathcal{S}^{\prime }\left( \mathbb{R}^{n}\right) $. Then $\chi \left(
D\right) u\in \mathcal{S}^{\prime }\left( \mathbb{R}^{n}\right) \cap 
\mathcal{C}_{\mathrm{pol}}^{\infty }\left( \mathbb{R}^{n}\right) $. In fact
we have 
\begin{equation*}
\chi \left( D\right) u\left( x\right) =\left( 2\pi \right) ^{-n}\left\langle 
\widehat{u},\mathrm{e}^{\mathrm{i}\left\langle x,\cdot \right\rangle }\chi
\right\rangle ,\quad x\in \mathbb{R}^{n}.
\end{equation*}

$(\mathrm{b})$ Let $u\in \mathcal{D}^{\prime }\left( \mathbb{R}^{n}\right) $ 
$($or $u\in \mathcal{S}^{\prime }\left( \mathbb{R}^{n}\right) )$ and $\chi
\in \mathcal{C}_{0}^{\infty }\left( \mathbb{R}^{n}\right) $ $($or $\chi \in 
\mathcal{S}\left( \mathbb{R}^{n}\right) )$. Then 
\begin{equation*}
\mathbb{R}^{n}\times \mathbb{R}^{n}\ni \left( x,\xi \right) \rightarrow \chi
\left( D-\xi \right) u\left( x\right) =\left( 2\pi \right) ^{-n}\left\langle 
\widehat{u},\mathrm{e}^{\mathrm{i}\left\langle x,\cdot \right\rangle }\chi
\left( \cdot -\xi \right) \right\rangle \in 
\mathbb{C}%
\end{equation*}%
is a $\mathcal{C}^{\infty }$-function.
\end{lemma}

\begin{proof}
Let $\varphi \in \mathcal{S}\left( \mathbb{R}_{x}^{n}\right) $ and $\check{%
\varphi}\left( x\right) =\varphi \left( -x\right) $. We have 
\begin{eqnarray*}
\left\langle \chi \left( D\right) u,\varphi \right\rangle &=&\left\langle 
\mathcal{F}^{-1}\left( \chi \widehat{u}\right) ,\varphi \right\rangle
=\left\langle \chi \widehat{u},\mathcal{F}^{-1}\varphi \right\rangle =\left(
2\pi \right) ^{-n}\left\langle \chi \widehat{u},\mathcal{F}\check{\varphi}%
\right\rangle \\
&=&\left( 2\pi \right) ^{-n}\left\langle \mathcal{F}\left( \chi \widehat{u}%
\right) ,\check{\varphi}\right\rangle =\left( 2\pi \right) ^{-n}\int
\left\langle \widehat{u},\mathrm{e}^{-\mathrm{i}\left\langle x,\cdot
\right\rangle }\chi \right\rangle \check{\varphi}\left( x\right) \mathrm{d}x
\\
&=&\left( 2\pi \right) ^{-n}\int \left\langle \widehat{u},\mathrm{e}^{%
\mathrm{i}\left\langle x,\cdot \right\rangle }\chi \right\rangle \varphi
\left( x\right) \mathrm{d}x
\end{eqnarray*}%
Hence 
\begin{equation*}
\chi \left( D\right) u\left( x\right) =\left( 2\pi \right) ^{-n}\left\langle 
\widehat{u},\mathrm{e}^{\mathrm{i}\left\langle x,\cdot \right\rangle }\chi
\right\rangle ,\quad x\in \mathbb{R}^{n}.
\end{equation*}
\end{proof}

\begin{lemma}
Let $\chi \in \mathcal{S}\left( \mathbb{R}^{n}\right) $ and $u\in \mathcal{S}%
^{\prime }\left( \mathbb{R}^{n}\right) $. Then%
\begin{equation*}
\chi \left( D-\xi \right) u\left( x\right) =\left( 2\pi \right) ^{-n}\mathrm{%
e}^{\mathrm{i}\left\langle x,\xi \right\rangle }\widehat{u\tau _{x}\widehat{%
\chi }}\left( \xi \right) ,\quad x,\xi \in \mathbb{R}^{n}.
\end{equation*}
\end{lemma}

\begin{proof}
Suppose that $u,\chi \in \mathcal{S}\left( \mathbb{R}^{n}\right) $. Then
using 
\begin{eqnarray*}
\mathcal{F}\left( \mathrm{e}^{\mathrm{i}\left\langle x,\cdot \right\rangle
}\chi \left( \cdot -\xi \right) \right) \left( y\right) &=&\int \mathrm{e}^{-%
\mathrm{i}\left\langle y,\eta \right\rangle }\mathrm{e}^{\mathrm{i}%
\left\langle x,\eta \right\rangle }\chi \left( \eta -\xi \right) \mathrm{d}%
\eta =\int \mathrm{e}^{\mathrm{i}\left\langle x-y,\zeta +\xi \right\rangle
}\chi \left( \zeta \right) \mathrm{d}\zeta \\
&=&\mathrm{e}^{\mathrm{i}\left\langle x-y,\xi \right\rangle }\widehat{\chi }%
\left( y-x\right)
\end{eqnarray*}%
we obtain%
\begin{multline*}
\chi \left( D-\xi \right) u\left( x\right) =\left( 2\pi \right) ^{-n}\int 
\mathrm{e}^{\mathrm{i}\left\langle x,\eta \right\rangle }\chi \left( \eta
-\xi \right) \widehat{u}\left( \eta \right) \mathrm{d}\eta \\
=\left( 2\pi \right) ^{-n}\int \mathcal{F}\left( \mathrm{e}^{\mathrm{i}%
\left\langle x,\cdot \right\rangle }\chi \left( \cdot -\xi \right) \right)
\left( y\right) u\left( y\right) \mathrm{d}y=\left( 2\pi \right) ^{-n}\int 
\mathrm{e}^{\mathrm{i}\left\langle x-y,\xi \right\rangle }\widehat{\chi }%
\left( y-x\right) u\left( y\right) \mathrm{d}y \\
=\left( 2\pi \right) ^{-n}\mathrm{e}^{\mathrm{i}\left\langle x,\xi
\right\rangle }\int \mathrm{e}^{-\mathrm{i}\left\langle y,\xi \right\rangle }%
\widehat{\chi }\left( y-x\right) u\left( y\right) \mathrm{d}y=\left( 2\pi
\right) ^{-n}\mathrm{e}^{\mathrm{i}\left\langle x,\xi \right\rangle }%
\widehat{u\tau _{x}\widehat{\chi }}\left( \xi \right) .
\end{multline*}%
The general case is obtained from the density of $\mathcal{S}\left( \mathbb{R%
}^{n}\right) $ in $\mathcal{S}^{\prime }\left( \mathbb{R}^{n}\right) $
noticing that both, 
\begin{equation*}
\chi \left( D-\xi \right) u\left( x\right) =\left( 2\pi \right)
^{-n}\left\langle \widehat{u},\mathrm{e}^{\mathrm{i}\left\langle x,\cdot
\right\rangle }\chi \left( \cdot -\xi \right) \right\rangle
\end{equation*}%
and%
\begin{equation*}
\widehat{u\tau _{x}\widehat{\chi }}\left( \xi \right) =\left\langle u,%
\mathrm{e}^{-\mathrm{i}\left\langle \cdot ,\xi \right\rangle }\widehat{\chi }%
\left( \cdot -x\right) \right\rangle ,
\end{equation*}%
depend continuously on $u$.
\end{proof}

\begin{corollary}
Let $1\leq p\leq \infty $, $\chi \in \mathcal{S}\left( \mathbb{R}^{n}\right) 
$ and $u\in \mathcal{S}^{\prime }\left( \mathbb{R}^{n}\right) $. Then 
\begin{equation*}
\left\Vert \chi \left( D-\xi \right) u\right\Vert _{L^{p}}=\left( 2\pi
\right) ^{-n}U_{\widehat{\chi },p}\left( \xi \right) ,\quad \xi \in \mathbb{R%
}^{n}.
\end{equation*}
\end{corollary}

\begin{corollary}
Let $1\leq p,q\leq \infty $ and $u\in \mathcal{S}^{\prime }\left( \mathbb{R}%
^{n}\right) $. Then the following statements are equivalent:

$\left( \mathrm{a}\right) $ $u\in M^{p,q}\left( \mathbb{R}^{n}\right) $;

$\left( \mathrm{b}\right) $ There is $\chi \in \mathcal{S}\left( \mathbb{R}%
^{n}\right) \smallsetminus 0$ so that $\xi \rightarrow \left\Vert \chi
\left( D-\xi \right) u\right\Vert _{L^{p}}$ is in $L^{q}\left( \mathbb{R}%
^{n}\right) $.
\end{corollary}

\begin{corollary}
Let $1\leq p,q\leq \infty $ and $\chi \in \mathcal{S}\left( \mathbb{R}%
^{n}\right) \smallsetminus 0$. Then 
\begin{gather*}
M^{p,q}\left( \mathbb{R}^{n}\right) \ni u\rightarrow \left( \int \left\Vert
\chi \left( D-\xi \right) u\right\Vert _{L^{p}}^{q}\mathrm{d}\xi \right)
^{1/q}\equiv \left\vert \left\vert \left\vert u\right\vert \right\vert
\right\vert _{M^{p,q},\chi } \\
u\rightarrow \left\Vert \left\Vert \chi \left( D-\cdot \right) u\right\Vert
_{L^{p}}\right\Vert _{L^{q}}=\left\vert \left\vert \left\vert u\right\vert
\right\vert \right\vert _{M^{p,q},\chi }
\end{gather*}%
is a norm on $M^{p,q}\left( \mathbb{R}^{n}\right) $. The topology defined by
this norm coincides with the topology of $M^{p,q}\left( \mathbb{R}%
^{n}\right) $.
\end{corollary}

Let $A$ be a real, symmetric and non-singular matrix, $\Phi _{A}$ the
quadratic form in $\mathbb{R}^{n}$ defined by $\Phi _{A}\left( x\right)
=-\left\langle Ax,x\right\rangle /2$, $x\in \mathbb{R}^{n}$. Then 
\begin{equation*}
\widehat{\mathrm{e}^{\mathrm{i}\Phi _{A}}}\left( \xi \right) =\left( 2\pi
\right) ^{n/2}\left\vert \det A\right\vert ^{-\frac{1}{2}}\mathrm{e}^{\frac{%
\pi \mathrm{isgn}A}{4}}\mathrm{e}^{-\mathrm{i}\left\langle A^{-1}\xi ,\xi
\right\rangle /2}
\end{equation*}%
This formula suggests the introduction of the operator 
\begin{equation*}
T_{A}:\mathcal{S}\left( \mathbb{R}^{m}\times \mathbb{R}^{n}\right)
\rightarrow \mathcal{S}^{\prime }\left( \mathbb{R}^{m}\times \mathbb{R}%
^{n}\right)
\end{equation*}%
defined by%
\begin{equation*}
T_{A}u=\left( 2\pi \right) ^{-n/2}\left\vert \det A\right\vert ^{\frac{1}{2}}%
\mathrm{e}^{-\frac{\pi \mathrm{isgn}A}{4}}\left( \delta \otimes \mathrm{e}^{%
\mathrm{i}\Phi _{A}}\right) \ast u.
\end{equation*}%
Then%
\begin{equation*}
\widehat{T_{A}u}=\left( 2\pi \right) ^{-n/2}\left\vert \det A\right\vert ^{%
\frac{1}{2}}\mathrm{e}^{-\frac{\pi \mathrm{isgn}A}{4}}\left( 1\otimes 
\widehat{\mathrm{e}^{\mathrm{i}\Phi _{A}}}\right) \cdot \widehat{u}=\left(
1\otimes \mathrm{e}^{\mathrm{i}\Phi _{A^{-1}}}\right) \cdot \widehat{u},
\end{equation*}%
\begin{equation*}
\widehat{T_{-A}u}=\left( 1\otimes \mathrm{e}^{\mathrm{i}\Phi
_{-A^{-1}}}\right) \cdot \widehat{u},
\end{equation*}%
so $T_{A}:\mathcal{S}\left( \mathbb{R}^{m}\times \mathbb{R}^{n}\right)
\rightarrow \mathcal{S}\left( \mathbb{R}^{m}\times \mathbb{R}^{n}\right) $
is invertible and $T_{A}^{-1}=T_{-A}$.

\begin{remark}
$\mathrm{supp}\left( \widehat{T_{A}u}\right) =\mathrm{supp}\left( \widehat{u}%
\right) $.
\end{remark}

\begin{theorem}
\label{n5}Let $1\leq p,q\leq \infty $ and $u\in \mathcal{S}^{\prime }\left( 
\mathbb{R}^{m}\times \mathbb{R}^{n}\right) $. Then

$\left( \mathrm{a}\right) $ $u\in M^{p,q}\left( \mathbb{R}^{m}\times \mathbb{%
R}^{n}\right) $ if and only if $T_{A}u\in M^{p,q}\left( \mathbb{R}^{m}\times 
\mathbb{R}^{n}\right) $.

$\left( \mathrm{b}\right) $ The operator 
\begin{equation*}
T_{A}:M^{p,q}\left( \mathbb{R}^{m}\times \mathbb{R}^{n}\right) \rightarrow
M^{p,q}\left( \mathbb{R}^{m}\times \mathbb{R}^{n}\right)
\end{equation*}
is a bounded isomorphism.

$\left( \mathrm{c}\right) $ If $1\leq q<\infty $, then the map%
\begin{equation*}
\mathrm{GL}\left( n,\mathbb{R}\right) \cap \mathrm{S}\left( n\right) \ni
A\rightarrow T_{A}u\in M^{p,q}\left( \mathbb{R}^{n}\right)
\end{equation*}%
is continuous. Here $\mathrm{S}\left( n\right) $ is the vector space of $%
n\times n$ real symmetric matrices.
\end{theorem}

\begin{proof}
We shall write $\xi =\left( \xi ^{\prime },\xi ^{\prime \prime }\right) $
for an element in $\mathbb{R}^{m}\times \mathbb{R}^{n}$ and accordingly $%
D=\left( D^{\prime },D^{\prime \prime }\right) $. We shall also write $%
\mathcal{F}_{m}$ for $\mathcal{F}_{\mathbb{R}^{m}}$ and $\mathcal{F}_{n}$
for $\mathcal{F}_{\mathbb{R}^{n}}$.

Let $h\in \mathcal{S}\left( \mathbb{R}^{m}\right) \smallsetminus 0$, $\psi
\in \mathcal{S}\left( \mathbb{R}^{n}\right) \smallsetminus 0$, $H=\mathcal{F}%
_{m}^{-1}\left( h\right) $, $\Psi =\mathcal{F}_{n}^{-1}\left( \psi \right) $
and $\chi =h\otimes \psi $. For $u\in M^{p,q}\left( \mathbb{R}^{m}\times 
\mathbb{R}^{n}\right) $ we shall evaluate%
\begin{equation*}
\left\Vert \chi ^{2}\left( D-\cdot \right) T_{A}u\right\Vert _{L^{p}}.
\end{equation*}%
Using the equality $\mathcal{F}^{-1}\circ \tau _{\zeta }=\mathrm{e}^{\mathrm{%
i}\left\langle \cdot ,\zeta \right\rangle }\mathcal{F}^{-1}$, if we set 
\begin{equation*}
C_{A,n}=\left( 2\pi \right) ^{-n/2}\left\vert \det A\right\vert ^{\frac{1}{2}%
}\mathrm{e}^{-\frac{\pi \mathrm{isgn}A}{4}},
\end{equation*}%
then%
\begin{align*}
& \chi ^{2}\left( D-\xi \right) T_{A}u \\
& =C_{A,n}\left[ \mathrm{e}^{\mathrm{i}\left\langle \cdot ,\xi ^{\prime
}\right\rangle }\mathcal{F}_{m}^{-1}\left( h\right) \otimes \left( \left( 
\mathrm{e}^{\mathrm{i}\left\langle \cdot ,\xi ^{\prime \prime }\right\rangle
}\mathcal{F}_{n}^{-1}\left( \psi \right) \right) \ast \mathrm{e}^{\mathrm{i}%
\Phi _{A}}\right) \right] \ast \chi \left( D-\xi \right) u \\
& =C_{A,n}\left[ \left( \mathrm{e}^{\mathrm{i}\left\langle \cdot ,\xi
^{\prime }\right\rangle }H\right) \otimes \left( \left( \mathrm{e}^{\mathrm{i%
}\left\langle \cdot ,\xi ^{\prime \prime }\right\rangle }\Psi \right) \ast
^{\prime \prime }\mathrm{e}^{\mathrm{i}\Phi _{A}}\right) \right] \ast \chi
\left( D-\xi \right) u.
\end{align*}%
Since 
\begin{eqnarray*}
\left( \mathrm{e}^{\mathrm{i}\left\langle \cdot ,\xi ^{\prime \prime
}\right\rangle }\Psi \right) \ast ^{\prime \prime }\mathrm{e}^{\mathrm{i}%
\Phi _{A}}\left( x^{\prime \prime }\right) &=&\int \mathrm{e}^{\mathrm{i}%
\left[ -\left\langle A\left( x^{\prime \prime }-y^{\prime \prime }\right)
,x^{\prime \prime }-y^{\prime \prime }\right\rangle /2+\left\langle
y^{\prime \prime },\xi ^{\prime \prime }\right\rangle \right] }\Psi \left(
y^{\prime \prime }\right) \mathrm{d}y^{\prime \prime } \\
&=&\mathrm{e}^{\mathrm{i}\Phi _{A}\left( x^{\prime \prime }\right) }\int 
\mathrm{e}^{\mathrm{i}\left\langle y^{\prime \prime },Ax^{\prime \prime
}+\xi ^{\prime \prime }\right\rangle }\mathrm{e}^{\mathrm{i}\Phi _{A}\left(
y^{\prime \prime }\right) }\Psi \left( y^{\prime \prime }\right) \mathrm{d}y
\\
&=&\mathrm{e}^{\mathrm{i}\Phi _{A}\left( x^{\prime \prime }\right) }\Psi
_{A}\left( Ax^{\prime \prime }+\xi ^{\prime \prime }\right)
\end{eqnarray*}%
where $\Psi _{A}=\mathcal{F}\left( \mathrm{e}^{\mathrm{i}\Phi _{A}}\check{%
\Psi}\right) $ it follows that%
\begin{align*}
& \left\Vert \chi ^{2}\left( D-\xi \right) T_{A}u\right\Vert _{L^{p}} \\
& \leq \left( 2\pi \right) ^{-n/2}\left\vert \det A\right\vert ^{\frac{1}{2}%
}\left\Vert \left( \mathrm{e}^{\mathrm{i}\left\langle \cdot ,\xi ^{\prime
}\right\rangle }H\right) \otimes \left( \left( \mathrm{e}^{\mathrm{i}%
\left\langle \cdot ,\xi ^{\prime \prime }\right\rangle }\Psi \right) \ast
^{\prime \prime }\mathrm{e}^{\mathrm{i}\Phi _{A}}\right) \right\Vert
_{L^{1}}\left\Vert \chi \left( D-\xi \right) u\right\Vert _{L^{p}} \\
& =\left( 2\pi \right) ^{-n/2}\left\vert \det A\right\vert ^{-\frac{1}{2}%
}\left\Vert H\right\Vert _{L^{1}}\left\Vert \Psi _{A}\right\Vert
_{L^{1}}\left\Vert \chi \left( D-\xi \right) u\right\Vert _{L^{p}},\quad \xi
\in \mathbb{R}^{m}\times \mathbb{R}^{n}.
\end{align*}%
This estimate implies $\left( \mathrm{a}\right) $ and $\left( \mathrm{b}%
\right) $. For part $\left( \mathrm{c}\right) $, we use this estimate and
Lebesgue's dominated convergence theorem.
\end{proof}

For $\lambda \in \mathrm{End}_{\mathbb{R}}\left( \mathbb{R}^{n}\right) $ and 
$u\in \mathcal{S}^{\prime }\left( \mathbb{R}^{n}\right) $ put $u_{\lambda
}=u\circ \lambda $ whenever it makes sense.

\begin{theorem}
\label{n6}If $\lambda \in \mathrm{End}_{\mathbb{R}}\left( \mathbb{R}%
^{n}\right) $ is invertible and $u\in M^{p,q}\left( \mathbb{R}^{n}\right) $,
then $u_{\lambda }\in M^{p,q}\left( \mathbb{R}^{n}\right) $ and there is $%
\mathrm{C}\in \left( 0,+\infty \right) $ independent of $u$ and $\lambda $
such that%
\begin{equation*}
\left\Vert u_{_{\lambda }}\right\Vert _{M^{p,q}}\leq \mathrm{C}\left\vert
\det \lambda \right\vert ^{-1/p-1/q^{\prime }}\left( 1+\left\Vert \lambda
\right\Vert \right) ^{n}\left\Vert u\right\Vert _{M^{p,q}}
\end{equation*}%
where $1/q+1/q^{\prime }=1$.
\end{theorem}

\begin{proof}
Let $\chi \in \mathcal{C}_{0}^{\infty }\left( \mathbb{R}^{n}\right) $ be
such that $\int \chi \left( x\right) \mathrm{d}x=1$. We shall use the
notation $\left\Vert \cdot \right\Vert _{M^{p,q}}$ for $\left\Vert \cdot
\right\Vert _{M^{p,q},\chi }$. Let $r>0$ be such that $\mathrm{supp}\chi
\subset \left\{ x:\left\vert x\right\vert \leq r\right\} $. We denote by $%
\chi _{1}$ the characteristic function of the unit ball in $\mathbb{R}^{n}$.
We evaluate%
\begin{eqnarray*}
\widehat{u_{\lambda }\tau _{y}\chi }\left( \xi \right) &=&\int \mathrm{e}^{-%
\mathrm{i}\left\langle x,\xi \right\rangle }u\left( \lambda x\right) \chi
\left( x-y\right) \mathrm{d}x \\
&=&\int \int \mathrm{e}^{-\mathrm{i}\left\langle x,\xi \right\rangle
}u\left( \lambda x\right) \chi \left( x-y\right) \chi \left( \lambda
x-z\right) \mathrm{d}x\mathrm{d}z.
\end{eqnarray*}%
Since $\chi \left( x-y\right) \chi \left( \lambda x-z\right) \neq 0$ implies 
$\left\vert x-y\right\vert \leq r$ and $\left\vert \lambda x-z\right\vert
\leq r$, we get that $\left\vert z-\lambda y\right\vert \leq r\left(
1+\left\Vert \lambda \right\Vert \right) $ on the support of the integrand.
We can write%
\begin{equation*}
\chi \left( x-y\right) \chi \left( \lambda x-z\right) =\chi \left(
x-y\right) \chi \left( \lambda x-z\right) \chi _{1}\left( \frac{z-\lambda y}{%
r\left( 1+\left\Vert \lambda \right\Vert \right) }\right) .
\end{equation*}%
It follows that 
\begin{multline*}
\widehat{u_{\lambda }\tau _{y}\chi }\left( \xi \right) =\int \int \mathrm{e}%
^{-\mathrm{i}\left\langle x,\xi \right\rangle }\left( u\tau _{z}\chi \right)
\left( \lambda x\right) \chi \left( x-y\right) \chi _{1}\left( \frac{%
z-\lambda y}{r\left( 1+\left\Vert \lambda \right\Vert \right) }\right) 
\mathrm{d}x\mathrm{d}z \\
=\left( 2\pi \right) ^{-n}\int \chi _{1}\left( \frac{z-\lambda y}{r\left(
1+\left\Vert \lambda \right\Vert \right) }\right) \left( \iint \mathrm{e}^{-%
\mathrm{i}\left\langle x,\xi \right\rangle }\mathrm{e}^{\mathrm{i}%
\left\langle x-y,\eta \right\rangle }\left( u\tau _{z}\chi \right) \left(
\lambda x\right) \widehat{\chi }\left( \eta \right) \mathrm{d}x\mathrm{d}%
\eta \right) \mathrm{d}z
\end{multline*}%
and%
\begin{multline*}
\iint \mathrm{e}^{-\mathrm{i}\left\langle x,\xi \right\rangle }\mathrm{e}^{%
\mathrm{i}\left\langle x-y,\eta \right\rangle }\left( u\tau _{z}\chi \right)
\left( \lambda x\right) \widehat{\chi }\left( \eta \right) \mathrm{d}x%
\mathrm{d}\eta \\
=\iint \mathrm{e}^{-\mathrm{i}\left\langle \lambda ^{-1}x,\xi \right\rangle +%
\mathrm{i}\left\langle x-\lambda y,\eta \right\rangle }\left( u\tau _{z}\chi
\right) \left( x\right) \widehat{\chi }\left( ^{t}\lambda \eta \right) 
\mathrm{d}x\mathrm{d}\eta \\
=\iint \mathrm{e}^{-\mathrm{i}\left\langle x,^{t}\lambda ^{-1}\xi -\eta
\right\rangle -\mathrm{i}\left\langle y,^{t}\lambda \eta \right\rangle
}\left( u\tau _{z}\chi \right) \left( x\right) \widehat{\chi }\left(
^{t}\lambda \eta \right) \mathrm{d}x\mathrm{d}\eta \\
=\int \mathrm{e}^{-\mathrm{i}\left\langle y,^{t}\lambda \eta \right\rangle }%
\widehat{u\tau _{z}\chi }\left( ^{t}\lambda ^{-1}\xi -\eta \right) \widehat{%
\chi }\left( ^{t}\lambda \eta \right) \mathrm{d}\eta \\
=\int \mathrm{e}^{-\mathrm{i}\left\langle y,\xi -^{t}\lambda \eta
\right\rangle }\widehat{u\tau _{z}\chi }\left( \eta \right) \widehat{\chi }%
\left( \xi -^{t}\lambda \eta \right) \mathrm{d}\eta .
\end{multline*}%
We obtain that%
\begin{equation*}
\widehat{u_{\lambda }\tau _{y}\chi }\left( \xi \right) =\left( 2\pi \right)
^{-n}\int \int \mathrm{e}^{-\mathrm{i}\left\langle y,\xi -^{t}\lambda \eta
\right\rangle }\chi _{1}\left( \frac{z-\lambda y}{r\left( 1+\left\Vert
\lambda \right\Vert \right) }\right) \widehat{u\tau _{z}\chi }\left( \eta
\right) \widehat{\chi }\left( \xi -^{t}\lambda \eta \right) \mathrm{d}\eta 
\mathrm{d}z
\end{equation*}%
\begin{equation*}
=\left( 2\pi \right) ^{-n}\int \int \mathrm{e}^{-\mathrm{i}\left\langle
y,\xi -^{t}\lambda \eta \right\rangle }\chi _{1}\left( \frac{x}{r\left(
1+\left\Vert \lambda \right\Vert \right) }\right) \widehat{u\tau _{x+\lambda
y}\chi }\left( \eta \right) \widehat{\chi }\left( \xi -^{t}\lambda \eta
\right) \mathrm{d}\eta \mathrm{d}x,
\end{equation*}%
and%
\begin{equation*}
\left\vert \widehat{u_{_{\lambda }}\tau _{y}\chi }\left( \xi \right)
\right\vert \leq \left( 2\pi \right) ^{-n}\int \int \chi _{1}\left( \frac{x}{%
r\left( 1+\left\Vert \lambda \right\Vert \right) }\right) \left\vert 
\widehat{u\tau _{x+\lambda y}\chi }\left( \eta \right) \right\vert
\left\vert \widehat{\chi }\left( \xi -^{t}\lambda \eta \right) \right\vert 
\mathrm{d}\eta \mathrm{d}x
\end{equation*}%
This estimate implies%
\begin{multline*}
U_{\lambda ,\chi ,p}\left( \xi \right) =\left( \int \left\vert \widehat{%
u_{_{\lambda }}\tau _{y}\chi }\left( \xi \right) \right\vert ^{p}\mathrm{d}%
y\right) ^{1/p} \\
\leq \left( 2\pi \right) ^{-n}\int \int \chi _{1}\left( \frac{x}{r\left(
1+\left\Vert \lambda \right\Vert \right) }\right) \left\vert \widehat{\chi }%
\left( \xi -^{t}\lambda \eta \right) \right\vert \left( \int \left\vert 
\widehat{u\tau _{x+\lambda y}\chi }\left( \eta \right) \right\vert ^{p}%
\mathrm{d}y\right) ^{1/p}\mathrm{d}\eta \mathrm{d}x \\
=\left( 2\pi \right) ^{-n}\left\vert \det \lambda \right\vert ^{-1/p}\int
\int \chi _{1}\left( \frac{x}{r\left( 1+\left\Vert \lambda \right\Vert
\right) }\right) \left\vert \widehat{\chi }\left( \xi -^{t}\lambda \eta
\right) \right\vert U_{\chi ,p}\left( \eta \right) \mathrm{d}\eta \mathrm{d}x
\\
=\left( 2\pi \right) ^{-n}r^{n}\left\vert \det \lambda \right\vert
^{-1/p}\left( 1+\left\Vert \lambda \right\Vert \right) ^{n}\mathrm{vol}%
\left( \left\{ \left\vert x\right\vert \leq 1\right\} \right) \int U_{\chi
,p}\left( \eta \right) \left\vert \widehat{\chi }\left( \xi -^{t}\lambda
\eta \right) \right\vert \mathrm{d}\eta
\end{multline*}%
The integral in the last row can be estimated using H\"{o}lder's inequality:%
\begin{align*}
& \int U_{\chi ,p}\left( \eta \right) \left\vert \widehat{\chi }\left( \xi
-^{t}\lambda \eta \right) \right\vert \mathrm{d}\eta \\
& \leq \left( \int U_{\chi ,p}^{q}\left( \eta \right) \left\vert \widehat{%
\chi }\left( \xi -^{t}\lambda \eta \right) \right\vert \mathrm{d}\eta
\right) ^{1/q}\left( \int \left\vert \widehat{\chi }\left( \xi -^{t}\lambda
\eta \right) \right\vert \mathrm{d}\eta \right) ^{1/q^{\prime }} \\
& =\left\vert \det \lambda \right\vert ^{-1/q^{\prime }}\left\Vert \widehat{%
\chi }\right\Vert _{_{L^{1}}}^{1/q^{\prime }}\left( \int U_{\chi
,p}^{q}\left( \eta \right) \left\vert \widehat{\chi }\left( \xi -^{t}\lambda
\eta \right) \right\vert \mathrm{d}\eta \right) ^{1/q}
\end{align*}%
If $\mathrm{c}=\left( 2\pi \right) ^{-n}r^{n}$\textrm{vol}$\left( \left\{
\left\vert x\right\vert \leq 1\right\} \right) $, then 
\begin{eqnarray*}
U_{\lambda ,\chi ,p}^{q}\left( \xi \right) &\leq &\mathrm{c}^{q}\left\vert
\det \lambda \right\vert ^{-q/p}\left( 1+\left\Vert \lambda \right\Vert
\right) ^{qn}\left\vert \det \lambda \right\vert ^{-q/q^{\prime }}\left\Vert 
\widehat{\chi }\right\Vert _{_{L^{1}}}^{q/q^{\prime }} \\
&&\cdot \left( \int U_{\chi ,p}^{q}\left( \eta \right) \left\vert \widehat{%
\chi }\left( \xi -^{t}\lambda \eta \right) \right\vert \mathrm{d}\eta \right)
\end{eqnarray*}%
which by integration with respect to $\xi $ gives us%
\begin{equation*}
\left\Vert U_{\lambda ,\chi ,p}\right\Vert _{_{L^{q}}}^{q}\leq \mathrm{c}%
^{q}\left\vert \det \lambda \right\vert ^{-q/p}\left( 1+\left\Vert \lambda
\right\Vert \right) ^{qn}\left\vert \det \lambda \right\vert ^{-q/q^{\prime
}}\left\Vert \widehat{\chi }\right\Vert _{_{L^{1}}}^{q/q^{\prime
}+1}\left\Vert U_{\chi ,p}\right\Vert _{_{L^{q}}}^{q}
\end{equation*}%
\begin{equation*}
\left\Vert U_{\lambda ,\chi ,p}\right\Vert _{_{L^{q}}}\leq \mathrm{c}%
\left\vert \det \lambda \right\vert ^{-1/p-1/q^{\prime }}\left( 1+\left\Vert
\lambda \right\Vert \right) ^{n}\left\Vert \widehat{\chi }\right\Vert
_{_{L^{1}}}\left\Vert U_{\chi ,p}\right\Vert _{_{L^{q}}}
\end{equation*}%
\begin{equation*}
\left\Vert u_{_{\lambda }}\right\Vert _{M^{p,q}}\leq \mathrm{C}\left\vert
\det \lambda \right\vert ^{-1/p-1/q^{\prime }}\left( 1+\left\Vert \lambda
\right\Vert \right) ^{n}\left\Vert u\right\Vert _{M^{p,q}}
\end{equation*}%
with $\mathrm{C}=\mathrm{c}\left\Vert \widehat{\chi }\right\Vert
_{_{L^{1}}}=\left( 2\pi \right) ^{-n}r^{n}$\textrm{vol}$\left( \left\{
\left\vert x\right\vert \leq 1\right\} \right) \left\Vert \widehat{\chi }%
\right\Vert _{_{L^{1}}}$.
\end{proof}

From Teorems \ref{n5} and \ref{n6} we deduce another important result of the
present paper.

\begin{theorem}
Let $1\leq p,q\leq \infty $ and $u\in \mathcal{S}^{\prime }\left( W\right) $%
. Then

$\left( \mathrm{a}\right) $ $u\in M^{p,q}\left( W\right) $ if and only if $%
\Lambda _{{\sigma ,T}}^{w}u\in M^{p,q}\left( W\right) $.

$\left( \mathrm{b}\right) $ The operator 
\begin{equation*}
\Lambda _{{\sigma ,T}}^{w}:M^{p,q}\left( W\right) \rightarrow M^{p,q}\left(
W\right)
\end{equation*}%
is a bounded isomorphism.
\end{theorem}

This theorem together with Theorem \ref{n4} and the corresponding results
from the standard Weyl calculus implies the first results on Schatten-class
properties for operators in the $T$-Weyl calculus.

\begin{theorem}
\label{n7}Let $\left( \mathcal{H},\mathcal{W}_{\sigma ,T},\omega _{\sigma
,T}\right) $ be an irreducible $\omega _{\sigma ,T}$-representation of $W$
and $1\leq p<\infty $.

$\left( \mathrm{a}\right) $ If $a\in M^{p,1}\left( W\right) $, then 
\begin{equation*}
\mathrm{Op}_{\sigma ,T}\left( a\right) \in \mathcal{B}_{p}\left( \mathcal{H}%
\right) ,
\end{equation*}%
where $\mathcal{B}_{p}\left( \mathcal{H}\right) $ denote the Schatten ideal
of compact operators whose singular values lie in $l^{p}$. We have 
\begin{equation*}
\left\Vert \mathrm{Op}_{\sigma ,T}\left( a\right) \right\Vert _{\mathcal{B}%
_{p}\left( \mathcal{H}\right) }\leq Cst\left\Vert a\right\Vert
_{M^{p,1}\left( W\right) }.
\end{equation*}

$\left( \mathrm{b}\right) $ If $a$ is in the Sj\"{o}strand algebra $%
M^{\infty ,1}\left( W\right) $, then 
\begin{equation*}
\mathrm{Op}_{\sigma ,T}\left( a\right) \in \mathcal{B}\left( \mathcal{H}%
\right) ,
\end{equation*}%
and 
\begin{equation*}
\left\Vert \mathrm{Op}_{\sigma ,T}\left( a\right) \right\Vert _{\mathcal{B}%
\left( \mathcal{H}\right) }\leq Cst\left\Vert a\right\Vert _{M^{\infty
,1}\left( W\right) }.
\end{equation*}
\end{theorem}

\begin{proof}
This theorem is a consequence of the previous theorem and the fact that it
is true for pseudo-differential operators with symbols in $M^{p,1}\left(
W\right) $ (see for instance \cite{Arsu 3} or \cite{Toft1} for $1\leq
p<\infty $ and \cite{Boulkhemair} for $p=\infty $).
\end{proof}

\section{An embedding theorem}

Results such as those on pseudo-differential operators from \cite{Arsu} and 
\cite{Arsu 2}, can be obtained using Theorem \ref{n7} and an embedding
theorem. To formulate the result we define some Sobolev type spaces $\left(
L^{p}\text{-style}\right) $. These spaces are defined by means of weight
functions.

\begin{definition}
$\left( \mathrm{a}\right) $ A positive measurable function $k$ defined in $%
\mathbb{R}^{n}$ will be called a weight function of polynomial growth if
there are positive constants $C$ and $N$ such that 
\begin{equation}
k\left( \xi +\eta \right) \leq C\left\langle \xi \right\rangle ^{N}k\left(
\eta \right) ,\quad \xi ,\eta \in \mathbb{R}^{n}.  \label{n8}
\end{equation}%
The set of all such functions $k$ will be denoted by $\mathcal{K}%
_{pol}\left( \mathbb{R}^{n}\right) $.

$\left( \mathrm{b}\right) $ For a weight function of polynomial growth $k$,
we shall write $M_{k}\left( \xi \right) =C\left\langle \xi \right\rangle
^{N} $, where $C$, $N$ are the positive constants that define $k$.
\end{definition}

\begin{remark}
$\left( \mathrm{a}\right) $ An immediate consequence of Peetre's inequality
is that $M_{k}$ is weakly submultiplicative, i.e.%
\begin{equation*}
M_{k}\left( \xi +\eta \right) \leq C_{k}M_{k}\left( \xi \right) M_{k}\left(
\eta \right) ,\quad \xi ,\eta \in \mathbb{R}^{n},
\end{equation*}%
where $C_{k}=2^{N/2}C^{-1}$ and that $k$ is moderate with respect to the
function $M_{k}$ or simply $M_{k}$-moderate, i.e.%
\begin{equation*}
k\left( \xi +\eta \right) \leq M_{k}\left( \xi \right) k\left( \eta \right)
,\quad \xi ,\eta \in \mathbb{R}^{n}.
\end{equation*}

$\left( \mathrm{b}\right) $ Let $k\in \mathcal{K}_{pol}\left( \mathbb{R}%
^{n}\right) $. From definition we deduce that 
\begin{equation*}
\frac{1}{M_{k}\left( \xi \right) }=C^{-1}\left\langle \xi \right\rangle
^{-N}\leq \frac{k\left( \xi +\eta \right) }{k\left( \eta \right) }\leq
C\left\langle \xi \right\rangle ^{N}=M_{k}\left( \xi \right) ,\quad \xi
,\eta \in \mathbb{R}^{n}.
\end{equation*}%
In fact, the left-hand inequality is obtained if $\xi $ is replaced by $-\xi 
$ and $\eta $ is replaced by $\xi +\eta $ in (\ref{n8}). If we take $\eta =0$
we obtain the useful estimates 
\begin{equation*}
C^{-1}k\left( 0\right) \left\langle \xi \right\rangle ^{-N}\leq k\left( \xi
\right) \leq Ck\left( 0\right) \left\langle \xi \right\rangle ^{N},\quad \xi
\in \mathbb{R}^{n}.
\end{equation*}
\end{remark}

The following lemma is an easy consequence of the definition and the above
estimates.

\begin{lemma}
Let $k$, $k_{1}$, $k_{2}\in \mathcal{K}_{pol}\left( \mathbb{R}^{n}\right) $.
Then:

$\left( \mathrm{a}\right) $ $k_{1}+k_{2}$, $k_{1}\cdot k_{2}$, $\sup \left(
k_{1},k_{2}\right) $, $\inf \left( k_{1},k_{2}\right) \in \mathcal{K}%
_{pol}\left( \mathbb{R}^{n}\right) $.

$\left( \mathrm{b}\right) $ $k^{s}\in \mathcal{K}_{pol}\left( \mathbb{R}%
^{n}\right) $ for every real $s$.

$\left( \mathtt{c}\right) $ If $\check{k}\left( \xi \right) =k\left( -\xi
\right) $, $\xi \in \mathbb{R}^{n}$, then $\check{k}$ is $M_{k}$-moderate
hence $\check{k}\in \mathcal{K}_{pol}\left( \mathbb{R}^{n}\right) $.

$\left( \mathtt{d}\right) $ $0<\inf_{\xi \in K}k\left( \xi \right) \leq
\sup_{\xi \in K}k\left( \xi \right) <\infty $ for every compact subset $%
K\subset \mathbb{R}^{n}$.
\end{lemma}

\begin{definition}
If $k\in \mathcal{K}_{pol}\left( \mathbb{R}^{n}\right) $ and $1\leq p\leq
\infty $, we denote by $H_{p}^{k}\left( \mathbb{R}^{n}\right) $ the set of
all distributions $u\in \mathcal{S}^{\prime }$ such that $k\left( D\right)
u\in L^{p}$. For $u\in H_{p}^{k}\left( \mathbb{R}^{n}\right) $ we define 
\begin{equation*}
\left\Vert u\right\Vert _{k,p}=\left\Vert k\left( D\right) u\right\Vert _{{L}%
^{p}}<\infty .
\end{equation*}
\end{definition}

$H_{p}^{k}\left( \mathbb{R}^{n}\right) $ is a Banach space with the norm $%
\left\Vert \cdot \right\Vert _{k,p}$. We have 
\begin{equation*}
\mathcal{S}\left( \mathbb{R}^{n}\right) \subset H_{p}^{k}\left( \mathbb{R}%
^{n}\right) \subset \mathcal{S}^{\prime }\left( \mathbb{R}^{n}\right)
\end{equation*}%
continuously and densely.

\begin{lemma}
Let $g:\left( 0,+\infty \right) \rightarrow \mathbb{C}$ and $k:\mathbb{R}%
^{n}\rightarrow \left( 0,+\infty \right) $ be two differentiable maps of
class $\geq r$ satisfying the following conditions:

$\left( \mathrm{a}\right) $ For any $j\leq r$, there is $C_{g,j}>0$ such that%
\begin{equation*}
t^{j}\left\vert g^{\left( j\right) }\left( t\right) \right\vert \leq
C_{g,j}\left\vert g\left( t\right) \right\vert ,\quad t>0.
\end{equation*}

$\left( \mathrm{b}\right) $ For any multi-index $\alpha =\left( \alpha
_{1},...,\alpha _{n}\right) $ of length $\left\vert \alpha \right\vert \leq
r $, there is $C_{k,\alpha }>0$ such that%
\begin{equation*}
\left\vert \partial ^{\alpha }k\left( x\right) \right\vert \leq C_{k,\alpha
}k\left( x\right) ,\quad x\in \mathbb{R}^{n}.
\end{equation*}

Then, for any multi-index $\alpha =\left( \alpha _{1},...,\alpha _{n}\right) 
$, $\left\vert \alpha \right\vert \leq r$, there is $C_{g,k,\alpha }>0$ such
that%
\begin{equation*}
\left\vert \partial ^{\alpha }\left( g\circ k\right) \left( x\right)
\right\vert \leq C_{g,k,\alpha }\left\vert g\circ k\left( x\right)
\right\vert ,\quad x\in \mathbb{R}^{n}.
\end{equation*}
\end{lemma}

\begin{proof}
By using the formula 
\begin{equation*}
\left( h^{\left( k\right) }\left( \cdot \right) \left(
v_{1},...,v_{k}\right) \right) ^{\prime }\left( x\right) \left( v_{0}\right)
=h^{\left( k+1\right) }\left( v\right) \left( v_{0},v_{1},...,v_{k}\right)
\end{equation*}%
to recursively construct the coefficients $\left\{ a_{n,j,\alpha ,\sigma
}\right\} $, it is shown by induction that%
\begin{equation*}
\left( g\circ k\right) ^{\left( m\right) }\left( x\right)
=\sum_{j=1}^{m}\sum_{\ell \in 
\mathbb{N}
^{{\footnotesize \ast j}},\left\vert \ell \right\vert =m}\sum_{\sigma \in
S_{n}}a_{n,j,\alpha ,\sigma }g^{\left( j\right) }\left( k\left( x\right)
\right) \sigma \cdot (k^{\left( \ell _{1}\right) }\left( x\right) \otimes
...\otimes k^{\left( \ell _{j}\right) }\left( x\right) ).
\end{equation*}%
This formula together with the assumptions on $g$ and $k$ imply the
conclusion of the lemma.
\end{proof}

\begin{lemma}
\label{n9}Let $1\leq p\leq \infty $, $\chi \in \mathcal{S}\left( \mathbb{R}%
^{n}\right) $ and $v\in L^{p}\left( \mathbb{R}^{n}\right) $. Then $\chi
\left( D\right) v\in L^{p}\left( \mathbb{R}^{n}\right) $ and 
\begin{equation*}
\left\Vert \chi \left( D\right) v\right\Vert _{L^{p}}\leq \left\Vert 
\mathcal{F}^{-1}\chi \right\Vert _{L^{1}}\left\Vert v\right\Vert _{L^{p}}.
\end{equation*}
\end{lemma}

\begin{proof}
We have 
\begin{eqnarray*}
\chi \left( D\right) v &=&\mathcal{F}^{-1}\left( \chi \cdot \widehat{v}%
\right) =\left( 2\pi \right) ^{-n}\mathcal{F}\left( \check{\chi}\cdot 
\widehat{\check{v}}\right) \\
&=&\left( 2\pi \right) ^{-n}\left( 2\pi \right) ^{-n}\left( \mathcal{F}%
\check{\chi}\right) \ast \left( \mathcal{F}\widehat{\check{v}}\right) \\
&=&\mathcal{F}^{-1}\chi \ast v.
\end{eqnarray*}%
Now Schur's lemma implies the result.
\end{proof}

\begin{theorem}
\label{n10}Let $1\leq p,q\leq \infty $, $r=\left[ \frac{n}{2}\right] +1$ and 
$k\in \mathcal{K}_{pol}\left( \mathbb{R}^{n}\right) \cap \mathcal{C}^{\infty
}\left( \mathbb{R}^{n}\right) $. Assume that $k$ satisfies the conditions:

$\left( \mathrm{a}\right) $ For any multi-index $\alpha =\left( \alpha
_{1},...,\alpha _{n}\right) $, $\left\vert \alpha \right\vert \leq 2r$,
there is $C_{\alpha }>0$ such that%
\begin{equation*}
\left\vert \partial ^{\alpha }k\left( \xi \right) \right\vert \leq C_{\alpha
}k\left( \xi \right) ,\quad \xi \in \mathbb{R}^{n}.
\end{equation*}

$\left( \mathrm{b}\right) $ $1/k\in L^{q}\left( \mathbb{R}^{n}\right) $.

\noindent Then 
\begin{equation*}
H_{p}^{k}\left( \mathbb{R}^{n}\right) \hookrightarrow M^{p,q}\left( \mathbb{R%
}^{n}\right) .
\end{equation*}
\end{theorem}

\begin{proof}
Let $\chi \in \mathcal{C}_{0}^{\infty }\left( \mathbb{R}^{n}\right)
\smallsetminus 0$. For $u\in H_{p}^{k}\left( \mathbb{R}^{n}\right) $, we
shall show that 
\begin{equation*}
\xi \rightarrow \left\Vert \chi \left( D-\xi \right) u\right\Vert _{L^{p}}
\end{equation*}%
is in $L^{q}\left( \mathbb{R}^{n}\right) $. Using previous lemma we get%
\begin{eqnarray*}
\left\Vert 1/k\left( D\right) \chi \left( D-\xi \right) k\left( D\right)
u\right\Vert _{L^{p}} &\leq &\left\Vert \mathcal{F}^{-1}\left( \frac{\chi
\left( \cdot -\xi \right) }{k\left( \cdot \right) }\right) \right\Vert
_{L^{1}}\left\Vert k\left( D\right) u\right\Vert _{L^{p}} \\
&=&\left\Vert \mathcal{F}^{-1}\left( \frac{\chi \left( \cdot -\xi \right) }{%
k\left( \cdot \right) }\right) \right\Vert _{L^{1}}\left\Vert u\right\Vert
_{k,p}.
\end{eqnarray*}%
It remains to evaluate $\left\Vert \mathcal{F}^{-1}\left( \frac{\chi \left(
\cdot -\xi \right) }{k\left( \cdot \right) }\right) \right\Vert _{L^{1}}$.
We have%
\begin{eqnarray*}
\mathcal{F}^{-1}\left( \frac{\chi \left( \cdot -\xi \right) }{k\left( \cdot
\right) }\right) \left( x\right) &=&\left( 2\pi \right) ^{-n}\int \mathrm{e}%
^{\mathrm{i}\left\langle x,\eta \right\rangle }\frac{\chi \left( \eta -\xi
\right) }{k\left( \eta \right) }\mathrm{d}\eta \\
&=&\left( 2\pi \right) ^{-n}\left\langle x\right\rangle ^{-2r}\int
\left\langle 1-\triangle \right\rangle ^{r}\left( \mathrm{e}^{\mathrm{i}%
\left\langle x,\cdot \right\rangle }\right) \left( \eta \right) \frac{\chi
\left( \eta -\xi \right) }{k\left( \eta \right) }\mathrm{d}\eta \\
&=&\left( 2\pi \right) ^{-n}\left\langle x\right\rangle ^{-2r}\int \mathrm{e}%
^{\mathrm{i}\left\langle x,\eta \right\rangle }\left\langle 1-\triangle
\right\rangle ^{r}\left( \frac{\chi \left( \cdot -\xi \right) }{k\left(
\cdot \right) }\right) \left( \eta \right) \mathrm{d}\eta .
\end{eqnarray*}%
Consequently,%
\begin{equation*}
\left\Vert \mathcal{F}^{-1}\left( \frac{\chi \left( \cdot -\xi \right) }{%
k\left( \cdot \right) }\right) \right\Vert _{L^{1}}\leq \left( 2\pi \right)
^{-n}\left\Vert \left\langle \cdot \right\rangle ^{-2r}\right\Vert
_{L^{1}}\left\Vert \left\langle 1-\triangle \right\rangle ^{r}\left( \frac{%
\chi \left( \cdot -\xi \right) }{k\left( \cdot \right) }\right) \right\Vert
_{L^{1}}
\end{equation*}%
with%
\begin{equation*}
\left\Vert \left\langle 1-\triangle \right\rangle ^{r}\left( \frac{\chi
\left( \cdot -\xi \right) }{k\left( \cdot \right) }\right) \right\Vert
_{L^{1}}\leq \sum_{\left\vert \alpha \right\vert \leq 2r}C_{\alpha
}\left\Vert \frac{1}{k\left( \cdot \right) }\partial ^{\alpha }\chi \left(
\cdot -\xi \right) \right\Vert _{L^{1}}
\end{equation*}%
by Lemma \ref{n9} above. Further we use the inequality $\left( \ref{n8}%
\right) $ to obtain 
\begin{eqnarray*}
k\left( \xi \right) \left\Vert \frac{1}{k\left( \cdot \right) }\partial
^{\alpha }\chi \left( \cdot -\xi \right) \right\Vert _{L^{1}} &=&\int \frac{%
k\left( \xi \right) }{k\left( \eta \right) }\left\vert \partial ^{\alpha
}\chi \left( \eta -\xi \right) \right\vert \mathrm{d}\eta \\
&\leq &\int C\left\langle \eta -\xi \right\rangle ^{N}\left\vert \partial
^{\alpha }\chi \left( \eta -\xi \right) \right\vert \mathrm{d}\eta \\
&=&\left\Vert M_{k}\left( \cdot \right) \partial ^{\alpha }\chi \left( \cdot
\right) \right\Vert _{L^{1}}
\end{eqnarray*}%
Hence 
\begin{equation*}
\left\Vert \left\langle 1-\triangle \right\rangle ^{r}\left( \frac{\chi
\left( \cdot -\xi \right) }{k\left( \cdot \right) }\right) \right\Vert
_{L^{1}}\leq \left( \sum_{\left\vert \alpha \right\vert \leq 2r}C_{\alpha
}\left\Vert M_{k}\left( \cdot \right) \partial ^{\alpha }\chi \left( \cdot
\right) \right\Vert _{L^{1}}\right) \frac{1}{k\left( \xi \right) }.
\end{equation*}%
Consequently we have%
\begin{multline*}
\left\vert \left\vert \left\vert u\right\vert \right\vert \right\vert
_{M^{p,q},\chi }=\left\Vert \left\Vert \chi \left( D-\cdot \right)
u\right\Vert _{L^{p}}\right\Vert _{L^{q}} \\
\leq \left( 2\pi \right) ^{-n}\left\Vert \left\langle \cdot \right\rangle
^{-2r}\right\Vert _{L^{1}}\left( \sum_{\left\vert \alpha \right\vert \leq
2r}C_{\alpha }\left\Vert M_{k}\left( \cdot \right) \partial ^{\alpha }\chi
\left( \cdot \right) \right\Vert _{L^{1}}\right) \left\Vert \frac{1}{k\left(
\cdot \right) }\right\Vert _{L^{q}}\left\Vert u\right\Vert _{k,p}.
\end{multline*}%
This completes the proof.
\end{proof}

Let $\ell \in \left\{ 1,...,n\right\} $, $1\leq p\leq \infty $ and $\mathbf{t%
}=\left( t_{1},...,t_{\ell }\right) \in \mathbb{R}^{\ell }$. Let $\mathbb{V}%
=\left( V_{1},...,V_{\ell }\right) $ denote an orthogonal decomposition, i.e.%
\begin{equation*}
\mathbb{R}^{n}=V_{1}\oplus ...\oplus V_{\ell }.
\end{equation*}%
We introduce the Banach space $H_{p,\mathbb{V}}^{\mathbf{t}}\left( \mathbb{R}%
^{n}\right) =H_{p,V_{1},...,V_{\ell }}^{t_{1},...,t_{\ell }}\left( \mathbb{R}%
^{n}\right) $ defined by%
\begin{equation*}
H_{p,\mathbb{V}}^{\mathbf{t}}\left( \mathbb{R}^{n}\right) =\left\{ u\in 
\mathcal{S}^{\prime }\left( \mathbb{R}^{n}\right) :\left( 1-\triangle
_{V_{1}}\right) ^{t_{1}/2}\otimes ...\otimes \left( 1-\triangle _{V_{\ell
}}\right) ^{t_{\ell }/2}u\in L^{p}\left( \mathbb{R}^{n}\right) \right\} ,
\end{equation*}%
\begin{equation*}
\left\Vert u\right\Vert _{H_{p,\mathbb{V}}^{\mathbf{t}}}=\left\Vert \left(
1-\triangle _{V_{1}}\right) ^{t_{1}/2}\otimes ...\otimes \left( 1-\triangle
_{V_{\ell }}\right) ^{t_{\ell }/2}u\right\Vert _{L^{p}},\quad u\in H_{p,%
\mathbb{V}}^{\mathbf{t}}.
\end{equation*}%
Then 
\begin{equation*}
H_{p,\mathbb{V}}^{\mathbf{t}}\left( \mathbb{R}^{n}\right) =H_{p}^{k}\left( 
\mathbb{R}^{n}\right)
\end{equation*}%
where%
\begin{equation*}
k\left( \cdot \right) =\left\langle \cdot \right\rangle
_{1}^{t_{1}}...\left\langle \cdot \right\rangle _{\ell }^{t_{\ell }},
\end{equation*}%
\begin{equation*}
\left\langle \xi \right\rangle _{j}=\left\langle \Pr\nolimits_{{V}%
_{j}}\left( \xi \right) \right\rangle =\left( 1+\left\vert \Pr\nolimits_{{V}%
_{j}}\left( \xi \right) \right\vert ^{2}\right) ^{11/2},\quad \xi \in 
\mathbb{R}^{n},1\leq j\leq \ell .
\end{equation*}%
For $1\leq q<\infty $, it is clear that if $qt_{1}>\dim V_{1},...,qt_{\ell
}>\dim V_{\ell }$, then the weight function $k$ satisfies the hypotheses of
the previous theorem.

\begin{corollary}
Suppose that $\mathbb{R}^{n}=V_{1}\oplus ...\oplus V_{\ell }$. If $1\leq
p\leq \infty $, $1\leq q<\infty $, and $qt_{1}>\dim V_{1},...,qt_{\ell
}>\dim V_{\ell }$, then $H_{p,V_{1},...,V_{\ell }}^{t_{1},...,t_{\ell
}}\left( \mathbb{R}^{n}\right) \hookrightarrow M^{p,q}\left( \mathbb{R}%
^{n}\right) $.
\end{corollary}

\begin{theorem}
Let $\left( \mathcal{H},\mathcal{W}_{\sigma ,T},\omega _{\sigma ,T}\right) $
be an irreducible $\omega _{\sigma ,T}$-representation of $W$ and $k\in 
\mathcal{K}_{pol}\left( W\right) \cap \mathcal{C}^{\infty }\left( W\right) $%
. Assume that $k$ satisfies the conditions:

$\left( \mathrm{a}\right) $ For any $j\leq 2\left( n+1\right) $, there is $%
C_{j}>0$ such that%
\begin{equation*}
\left\vert k^{\left( j\right) }\left( \xi \right) \right\vert \leq
C_{j}k\left( \xi \right) ,\quad \xi \in W.
\end{equation*}

$\left( \mathrm{b}\right) $ $1/k\in L^{1}\left( \mathbb{R}^{n}\right) $.

\noindent Then

$\left( \mathrm{i}\right) $ $1\leq p<\infty $ and $a\in H_{p}^{k}\left(
W\right) $ imply $\mathrm{Op}_{\sigma ,T}\left( a\right) \in \mathcal{B}%
_{p}\left( \mathcal{H}\right) $ and 
\begin{equation*}
\left\Vert \mathrm{Op}_{\sigma ,T}\left( a\right) \right\Vert _{\mathcal{B}%
_{p}\left( \mathcal{H}\right) }\leq Cst\left\Vert a\right\Vert
_{H_{p}^{k}\left( W\right) }.
\end{equation*}

$\left( \mathrm{ii}\right) $ $a\in H_{\infty }^{k}\left( W\right) $ implies $%
\mathrm{Op}_{\sigma ,T}\left( a\right) \in \mathcal{B}\left( \mathcal{H}%
\right) $ and 
\begin{equation*}
\left\Vert \mathrm{Op}_{\sigma ,T}\left( a\right) \right\Vert _{\mathcal{B}%
\left( \mathcal{H}\right) }\leq Cst\left\Vert a\right\Vert _{H_{\infty
}^{k}\left( W\right) }.
\end{equation*}%
Here $\mathcal{B}_{p}\left( \mathcal{H}\right) $ denote the Schatten ideal
of compact operators whose singular values lie in $l^{p}$.
\end{theorem}

\section{Cordes' lemma}

Recall the definition of generalized Sobolev spaces. Let $s\in \mathbb{R}$, $%
1\leq p\leq \infty $. Define 
\begin{gather*}
H_{p}^{s}\left( \mathbb{R}^{n}\right) =\left\{ u\in \mathcal{S}^{\prime
}\left( \mathbb{R}^{n}\right) :\left\langle D\right\rangle ^{s}u\in
L^{p}\left( \mathbb{R}^{n}\right) \right\} , \\
\left\Vert u\right\Vert _{H_{p}^{s}}=\left\Vert \left\langle D\right\rangle
^{s}u\right\Vert _{L^{p}},\quad u\in H_{p}^{s}.
\end{gather*}

\begin{lemma}
If $s>n$ and $1\leq p\leq \infty $, then $H_{p}^{s}\left( \mathbb{R}%
^{n}\right) \hookrightarrow M^{p,1}\left( \mathbb{R}^{n}\right) $.
\end{lemma}

\begin{proof}
This is a particular case of theorem \ref{n10}.
\end{proof}

To state and prove Cordes' lemma we shall work with a very restrictive class
of symbols. We shall say that $a:\mathbb{R}^{n}\rightarrow 
\mathbb{C}
$ is a symbol of order $m$ ($m$ any real number) if $a\in \mathcal{C}%
^{\infty }\left( \mathbb{R}^{n}\right) $ and for any $\alpha \in 
\mathbb{N}
^{n}$, there is $C_{\alpha }>0$ such that 
\begin{equation*}
\left\vert \partial ^{\alpha }a\left( x\right) \right\vert \leq C_{\alpha
}\left\langle x\right\rangle ^{m-\left\vert \alpha \right\vert },\quad x\in 
\mathbb{R}^{n}.
\end{equation*}%
We denote by $\mathcal{S}^{m}\left( \mathbb{R}^{n}\right) $ the vector space
of all symbols of order $m$ and observe that 
\begin{equation*}
m_{1}\leq m_{2}\Rightarrow \mathcal{S}^{m_{1}}\left( \mathbb{R}^{n}\right)
\subset \mathcal{S}^{m_{2}}\left( \mathbb{R}^{n}\right) ,\quad \mathcal{S}%
^{m_{1}}\left( \mathbb{R}^{n}\right) \cdot \mathcal{S}^{m_{2}}\left( \mathbb{%
R}^{n}\right) \subset \mathcal{S}^{m_{1}+m_{2}}\left( \mathbb{R}^{n}\right) .
\end{equation*}%
Observe also that $a\in \mathcal{S}^{m}\left( \mathbb{R}^{n}\right)
\Rightarrow \partial ^{\alpha }a\in \mathcal{S}^{m-\left\vert \alpha
\right\vert }\left( \mathbb{R}^{n}\right) $ for each $\alpha \in 
\mathbb{N}
^{n}$. The function $\left\langle x\right\rangle ^{m}$ clearly belongs to $%
\mathcal{S}^{m}\left( \mathbb{R}^{n}\right) $ for any $m\in \mathbb{R}$. We
denote by $\mathcal{S}^{\infty }\left( \mathbb{R}^{n}\right) $ the union of
all the spaces $\mathcal{S}^{m}\left( \mathbb{R}^{n}\right) $ and we note
that $\mathcal{S}\left( \mathbb{R}^{n}\right) =\bigcap_{m\in \mathbb{R}}%
\mathcal{S}^{m}\left( \mathbb{R}^{n}\right) $ the space of tempered test
functions. It is clear that $\mathcal{S}^{m}\left( \mathbb{R}^{n}\right) $
is a Fr\'{e}chet space with the seni-norms given by 
\begin{equation*}
\left\vert a\right\vert _{m,\alpha }=\sup_{x\in \mathbb{R}^{n}}\left\langle
x\right\rangle ^{-m+\left\vert \alpha \right\vert }\left\vert \partial
^{\alpha }a\left( x\right) \right\vert ,\quad a\in \mathcal{S}^{m}\left( 
\mathbb{R}^{n}\right) ,\alpha \in 
\mathbb{N}
^{n}.
\end{equation*}%
In addition to the previous lemma, we use part $\left( \mathrm{iii}\right) $
of Proposition 2.4 from \cite{Arsu 2} which state that%
\begin{equation*}
\mathcal{F}^{-1}\left( \bigcup_{m<0}\mathcal{S}^{m}\left( \mathbb{R}%
^{n}\right) \right) \subset L^{1}\left( \mathbb{R}^{n}\right) .
\end{equation*}

\begin{corollary}
$\mathcal{F}^{-1}\left( \bigcup_{m>n}\mathcal{S}^{-m}\left( \mathbb{R}%
^{n}\right) \right) \cup \mathcal{F}\left( \bigcup_{m>n}\mathcal{S}%
^{-m}\left( \mathbb{R}^{n}\right) \right) \subset M^{1,1}\left( \mathbb{R}%
^{n}\right) $.
\end{corollary}

\begin{proof}
It is sufficient to prove the inclusion 
\begin{equation*}
\mathcal{F}^{-1}\left( \bigcup_{m>n}\mathcal{S}^{-m}\left( \mathbb{R}%
^{n}\right) \right) \subset M^{1,1}\left( \mathbb{R}^{n}\right) .
\end{equation*}

Let $m>n$ and $s\in \left( n,m\right) $. We shall show that if $a\in 
\mathcal{S}^{-m}\left( \mathbb{R}^{n}\right) $, then $\mathcal{F}^{-1}a\in
H_{p}^{s}$. Indeed, if $u=\mathcal{F}^{-1}a$, then%
\begin{equation*}
\left\langle D\right\rangle ^{s}u=\mathcal{F}^{-1}\left( \left\langle \cdot
\right\rangle ^{s}a\right) \in L^{1}\left( \mathbb{R}^{n}\right) ,
\end{equation*}%
since $\left\langle \cdot \right\rangle ^{s}a\in \mathcal{S}^{s-m}\left( 
\mathbb{R}^{n}\right) $ and $s-m<0$. Using previous lemma we obtain that $%
\mathcal{F}^{-1}a$ is in the Feichtinger algebra $M^{1,1}\left( \mathbb{R}%
^{n}\right) $.
\end{proof}

Using the above lemma, the fact that the map 
\begin{equation*}
M^{p,1}\left( \mathbb{R}^{n}\right) \times M^{p,1}\left( \mathbb{R}%
^{m}\right) \ni \left( u,v\right) \rightarrow u\otimes v\in M^{p,1}\left( 
\mathbb{R}^{n}\times \mathbb{R}^{m}\right)
\end{equation*}%
is well defined and continuous and Theorem \ref{n7}, we get the following
extension of Cordes' lemma.

\begin{corollary}[Cordes' lemma]
\label{n13}Suppose that $W=V_{1}\oplus ...\oplus V_{\ell }$ is an orthogonal
decomposition with respect to a $\sigma $-compatible inner product on $%
\left( W,\sigma \right) $. Let $t_{1}>n_{1}=\dim V_{1},...,$ $t_{\ell
}>n_{\ell }=\dim V_{\ell }$. Let $a_{1}\in \mathcal{S}^{-t_{1}}\left(
V_{1}\right) ,...,a_{k}\in \mathcal{S}^{-t_{\ell }}\left( V_{\ell }\right) $
and $g:W\rightarrow 
\mathbb{C}
$,%
\begin{equation*}
g=\mathcal{F}_{{V}_{{1}}}^{\dag }\left( a_{1}\right) \otimes ...\otimes 
\mathcal{F}_{{V}_{{\ell }}}^{\dag }\left( a_{\ell }\right) ,
\end{equation*}%
where $\mathcal{F}_{{V}}^{\dag }$ is either $\mathcal{F}_{V}$ or $\mathcal{F}%
_{V}^{-1}$. If the representation $\left( \mathcal{H},\mathcal{W}_{\sigma
,T},\omega _{\sigma ,T}\right) $ is irreducible, then $\mathrm{Op}_{\sigma
,T}\left( g\right) \in \mathcal{B}_{1}\left( \mathcal{H}\right) \cdot $
\end{corollary}

\section{Kato's identity}

In this section we shall describe an extension of a formula due to T. Kato 
\cite{Kato}.

On the symplectic vector space $\left( W,\sigma \right) $ both, duality $%
\left\langle \cdot ,\cdot \right\rangle _{{S,S}^{\prime }}^{\sigma }$ and
antiduality $\left\langle \cdot ,\cdot \right\rangle _{{S,S}^{\ast
}}^{\sigma }$ are defined taking into account the symplectic structure of W,
i.e. if $\varphi ,u\in \mathcal{S}\left( W\right) $, then 
\begin{equation*}
\left\langle \varphi ,u\right\rangle _{\mathcal{S},\mathcal{S}^{\prime
}}^{\sigma }=\int_{W}\varphi \left( \eta \right) u\left( \eta \right) 
\mathrm{d}^{\sigma }\eta ,\quad \left\langle \varphi ,u\right\rangle _{%
\mathcal{S},\mathcal{S}^{\ast }}^{\sigma }=\int_{W}\overline{\varphi \left(
\eta \right) }u\left( \eta \right) \mathrm{d}^{\sigma }\eta .
\end{equation*}%
Likewise, the convolution, $\ast _{\sigma }$, is defined by means of the
Fourier measure, $\mathrm{d}^{\sigma }\xi $, by 
\begin{equation*}
\left( \varphi \ast _{\sigma }u\right) \left( \xi \right) =\int_{W}\varphi
\left( \xi -\eta \right) u\left( \eta \right) \mathrm{d}^{\sigma }\eta
,\quad \varphi ,u\in \mathcal{S}\left( W\right) .
\end{equation*}%
Let $\varphi \in \mathcal{S}\left( W\right) $ and $u\in \mathcal{S}^{\prime
}\left( W\right) $ (or $u\in \mathcal{S}^{\ast }\left( W\right) $). Then the
formal integral%
\begin{equation*}
\left( \varphi \ast _{\sigma }u\right) \left( \xi \right) =\int_{W}\varphi
\left( \xi -\eta \right) u\left( \eta \right) \mathrm{d}^{\sigma }\eta
\end{equation*}%
has a rigorous meaning when interpreted as the action of the distribution $u$
on the test function $\eta \rightarrow \varphi \left( \xi -\eta \right) $
(or $\eta \rightarrow \overline{\varphi }\left( \xi -\eta \right) $), i.e.%
\begin{equation*}
\left( \varphi \ast _{\sigma }u\right) \left( \xi \right) =\left\langle
\varphi \left( \xi -\cdot \right) ,u\right\rangle _{\mathcal{S},\mathcal{S}%
^{\prime }}^{\sigma },\quad \left( \varphi \ast _{\sigma }u\right) \left(
\xi \right) =\left\langle \overline{\varphi }\left( \xi -\cdot \right)
,u\right\rangle _{\mathcal{S},\mathcal{S}^{\ast }}^{\sigma }.
\end{equation*}%
If $S:W\rightarrow W$ is a linear isomorphism such that $S^{\sigma }=S$ then 
$\mathrm{d}^{\sigma _{S}}\eta =\left( \det S\right) ^{\frac{1}{2}}\mathrm{d}%
^{\sigma }\eta $,%
\begin{equation*}
\left\langle \varphi ,u\right\rangle _{\mathcal{S}{,}\mathcal{S}^{\prime
}}^{\sigma _{S}}=\left( \det S\right) ^{\frac{1}{2}}\left\langle \varphi
,u\right\rangle _{\mathcal{S},\mathcal{S}^{\prime }}^{\sigma },\quad
\left\langle \varphi ,u\right\rangle _{\mathcal{S},\mathcal{S}^{\ast
}}^{\sigma _{S}}=\left( \det S\right) ^{\frac{1}{2}}\left\langle \varphi
,u\right\rangle ,
\end{equation*}%
and 
\begin{equation*}
\varphi \ast _{\sigma _{S}}u=\left( \det S\right) ^{\frac{1}{2}}\left(
\varphi \ast _{\sigma }u\right) .
\end{equation*}%
For $b$, $c\in \mathcal{S}\left( W\right) $,%
\begin{eqnarray*}
\left( \left( b\ast _{\sigma }c\right) \circ S\right) \left( \xi \right)
&=&\int_{W}b\left( S\left( \xi \right) -\eta \right) c\left( \eta \right) 
\mathrm{d}^{\sigma }\eta \\
&=&\int_{W}\left( b\circ S\right) \left( \xi -S^{-1}\eta \right) \left(
c\circ S\right) \left( S^{-1}\eta \right) \mathrm{d}^{\sigma }\eta \\
&=&\det S\int_{W}\left( b\circ S\right) \left( \xi -\zeta \right) \left(
c\circ S\right) \left( \zeta \right) \mathrm{d}^{\sigma }\eta \\
&=&\left( \det S\right) ^{\frac{1}{2}}\int_{W}\left( b\circ S\right) \left(
\xi -\zeta \right) \left( c\circ S\right) \left( \zeta \right) \mathrm{d}%
^{\sigma _{S}}\eta \\
&=&\left( \det S\right) ^{\frac{1}{2}}\left( \left( b\circ S\right) \ast
_{\sigma _{S}}\left( c\circ S\right) \right) \left( \xi \right) ,
\end{eqnarray*}%
hence%
\begin{equation*}
\left( b\ast _{\sigma }c\right) \circ S=\left( \det S\right) ^{\frac{1}{2}%
}\left( \left( b\circ S\right) \ast _{\sigma _{S}}\left( c\circ S\right)
\right) ,
\end{equation*}%
and this formula remains true in all cases where the operations of
convolution and composition with a linear bijection make sense.

If $b\in \mathcal{S}\left( W\right) $, $c\in \mathcal{S}^{\ast }\left(
W\right) $, then $b\ast _{\sigma }c\in \mathcal{S}^{\ast }\left( W\right) $
and from Theorem \ref{n4} one obtains that%
\begin{equation*}
\mathrm{Op}_{\sigma ,T}\left( b\ast _{\sigma }c\right) =\mathrm{Op}^{%
\widetilde{w}}\left( \left( b\ast _{\sigma }c\right) _{\sigma ,T}^{w}\right)
,
\end{equation*}%
where%
\begin{equation*}
\left( b\ast _{\sigma }c\right) _{\sigma ,T}^{w}=\lambda _{{\sigma ,T}%
}\left( D_{\sigma _{T+T^{{\sigma }}}}\right) \left( \left( b\ast _{\sigma
}c\right) \circ \left( T+T^{\sigma }\right) \right) \qquad \quad \qquad
\quad \qquad \quad \qquad \quad \qquad \quad \qquad \quad
\end{equation*}%
\begin{align*}
\qquad & =\left( \det \left( T+T^{\sigma }\right) \right) ^{\frac{1}{2}%
}\lambda _{{\sigma ,T}}\left( D_{\sigma _{T+T^{{\sigma }}}}\right) \left(
\left( b\circ \left( T+T^{\sigma }\right) \right) \ast _{\sigma _{T+T^{{%
\sigma }}}}\left( c\circ \left( T+T^{\sigma }\right) \right) \right) \\
& =\left( \det \left( T+T^{\sigma }\right) \right) ^{\frac{1}{2}}\left(
\lambda _{{\sigma ,T}}\left( D_{\sigma _{T+T^{{\sigma }}}}\right) \left(
b\circ \left( T+T^{\sigma }\right) \right) \ast _{\sigma _{T+T^{{\sigma }%
}}}\left( c\circ \left( T+T^{\sigma }\right) \right) \right) \\
& =\left( \det \left( T+T^{\sigma }\right) \right) ^{\frac{1}{2}}\left(
\left( b\circ \left( T+T^{\sigma }\right) \right) \ast _{\sigma _{T+T^{{%
\sigma }}}}\lambda _{{\sigma ,T}}\left( D_{\sigma _{T+T^{{\sigma }}}}\right)
\left( c\circ \left( T+T^{\sigma }\right) \right) \right) ,
\end{align*}%
hence%
\begin{equation*}
\left( b\ast _{\sigma }c\right) _{\sigma ,T}^{w}\qquad \quad \qquad \quad
\qquad \quad \qquad \quad \qquad \quad \qquad \quad \qquad \quad \qquad
\quad \qquad \quad \qquad \quad \qquad \quad \qquad \quad
\end{equation*}%
\begin{equation*}
=\left( \det \left( T+T^{\sigma }\right) \right) ^{\frac{1}{2}}\left( \left(
b\circ \left( T+T^{\sigma }\right) \right) \ast _{\sigma _{T+T^{{\sigma }%
}}}\lambda _{{\sigma ,T}}\left( D_{\sigma _{T+T^{{\sigma }}}}\right) \left(
c\circ \left( T+T^{\sigma }\right) \right) \right)
\end{equation*}%
\begin{equation*}
=\left( \det \left( T+T^{\sigma }\right) \right) ^{\frac{1}{2}}\left(
\lambda _{{\sigma ,T}}\left( D_{\sigma _{T+T^{{\sigma }}}}\right) \left(
b\circ \left( T+T^{\sigma }\right) \right) \ast _{\sigma _{T+T^{{\sigma }%
}}}\left( c\circ \left( T+T^{\sigma }\right) \right) \right) ,
\end{equation*}%
i.e.%
\begin{align*}
\left( b\ast _{\sigma }c\right) _{\sigma ,T}^{w}& =\left( \det \left(
T+T^{\sigma }\right) \right) ^{\frac{1}{2}}\left( \left( b\circ \left(
T+T^{\sigma }\right) \right) \ast _{\sigma _{T+T^{{\sigma }}}}c_{\sigma
,T}^{w}\right) \\
& =\left( \det \left( T+T^{\sigma }\right) \right) ^{\frac{1}{2}}\left(
b_{\sigma ,T}^{w}\ast _{\sigma _{T+T^{{\sigma }}}}\left( c\circ \left(
T+T^{\sigma }\right) \right) \right) .
\end{align*}%
Recalling that $\mathrm{Op}^{\widetilde{w}}$ is the standard Weyl calculus
corresponding to the Weyl system $\left( \mathcal{H},\widetilde{\mathcal{W}}%
_{\sigma ,T},\widetilde{\omega }_{\sigma ,T}\right) $, for the symplectic
space $\left( W,\sigma _{T+T^{{\sigma }}}\right) $, this can be used
together with Lemma 2.1 and Lemma 2.2 from \cite{Arsu} to represent $\mathrm{%
Op}_{\sigma ,T}\left( b\ast _{\sigma }c\right) $.%
\begin{equation*}
\mathrm{Op}_{\sigma ,T}\left( b\ast _{\sigma }c\right) \qquad \quad \qquad
\quad \qquad \quad \qquad \quad \qquad \quad \qquad \quad \qquad \quad
\qquad \quad \qquad \quad \qquad \quad \qquad \quad \qquad \quad
\end{equation*}%
\begin{align*}
& =\left( \det \left( T+T^{\sigma }\right) \right) ^{\frac{1}{2}%
}\int_{W}\left( b\circ \left( T+T^{\sigma }\right) \right) \left( \xi
\right) \widetilde{\mathcal{W}}_{\sigma ,T}\left( \xi \right) \mathrm{Op}^{%
\widetilde{w}}\left( c_{\sigma ,T}^{w}\right) \widetilde{\mathcal{W}}%
_{\sigma ,T}\left( -\xi \right) \mathrm{d}^{\sigma _{T+T^{{\sigma }}}}\xi \\
& =\det \left( T+T^{\sigma }\right) \int_{W}\left( b\circ \left( T+T^{\sigma
}\right) \right) \left( \xi \right) \widetilde{\mathcal{W}}_{\sigma
,T}\left( \xi \right) \mathrm{Op}_{\sigma ,T}\left( c\right) \widetilde{%
\mathcal{W}}_{\sigma ,T}\left( -\xi \right) \mathrm{d}^{\sigma }\xi \\
& =\det \left( T+T^{\sigma }\right) \int_{W}\left( c\circ \left( T+T^{\sigma
}\right) \right) \left( \xi \right) \widetilde{\mathcal{W}}_{\sigma
,T}\left( \xi \right) \mathrm{Op}_{\sigma ,T}\left( b\right) \widetilde{%
\mathcal{W}}_{\sigma ,T}\left( -\xi \right) \mathrm{d}^{\sigma }\xi
\end{align*}%
i.e.%
\begin{equation*}
\mathrm{Op}_{\sigma ,T}\left( b\ast _{\sigma }c\right) \qquad \quad \qquad
\quad \qquad \quad \qquad \quad \qquad \quad \qquad \quad \qquad \quad
\qquad \quad \qquad \quad \qquad \quad \qquad \quad \qquad \quad
\end{equation*}%
\begin{align*}
\quad \quad \qquad & =\det \left( T+T^{\sigma }\right) \int_{W}\left( b\circ
\left( T+T^{\sigma }\right) \right) \left( \xi \right) \widetilde{\mathcal{W}%
}_{\sigma ,T}\left( \xi \right) \mathrm{Op}_{\sigma ,T}\left( c\right) 
\widetilde{\mathcal{W}}_{\sigma ,T}\left( -\xi \right) \mathrm{d}^{\sigma
}\xi \\
& =\det \left( T+T^{\sigma }\right) \int_{W}\left( c\circ \left( T+T^{\sigma
}\right) \right) \left( \xi \right) \widetilde{\mathcal{W}}_{\sigma
,T}\left( \xi \right) \mathrm{Op}_{\sigma ,T}\left( b\right) \widetilde{%
\mathcal{W}}_{\sigma ,T}\left( -\xi \right) \mathrm{d}^{\sigma }\xi ,
\end{align*}%
where the first integral is weakly absolutely convergent while the second
one must be interpreted in the sese of distributions and represents the
operator defined by%
\begin{align*}
& \left\langle \varphi ,\left( \int_{W}c\circ \left( T+T^{\sigma }\right)
\left( \xi \right) \widetilde{\mathcal{W}}_{\sigma ,T}\left( \xi \right) 
\mathrm{Op}_{\sigma ,T}\left( b\right) \widetilde{\mathcal{W}}_{\sigma
,T}\left( -\xi \right) \mathrm{d}^{\sigma }\xi \right) \psi \right\rangle _{%
\mathcal{S},\mathcal{S}^{\ast }} \\
& =\left\langle \overline{\left\langle \varphi ,\widetilde{\mathcal{W}}%
_{\sigma ,T}\left( \cdot \right) \mathrm{Op}_{\sigma ,T}\left( b\right) 
\widetilde{\mathcal{W}}_{\sigma ,T}\left( -\cdot \right) \psi \right\rangle }%
_{\mathcal{S},\mathcal{S}^{\ast }},c\circ \left( T+T^{\sigma }\right)
\right\rangle _{\mathcal{S}\left( W\right) ,\mathcal{S}^{\ast }\left(
W\right) }^{\sigma },
\end{align*}%
for all $\varphi $, $\psi \in \mathcal{S}$.

It may be appropriate here to introduce $\left\{ \mathcal{U}_{\sigma
,T}\left( \xi \right) \right\} _{\xi \in W}$, defined by 
\begin{eqnarray*}
\mathcal{U}_{\sigma ,T}\left( \xi \right) &=&\widetilde{\mathcal{W}}_{\sigma
,T}\left( \left( T+T^{\sigma }\right) ^{-1}\xi \right) \\
&=&\mathrm{e}^{\frac{\mathrm{i}}{2}\sigma \left( \left( T+T^{\sigma }\right)
^{-1}\xi ,T\left( T+T^{\sigma }\right) ^{-1}\xi \right) }\mathcal{W}_{\sigma
,T}\left( \left( T+T^{\sigma }\right) ^{-1}\xi \right) \\
&=&\mathrm{e}^{\frac{\mathrm{i}}{2}\sigma \left( \xi ,\left( T+T^{\sigma
}\right) ^{-1}T\left( T+T^{\sigma }\right) ^{-1}\xi \right) }\mathcal{W}%
_{\sigma ,T}\left( \left( T+T^{\sigma }\right) ^{-1}\xi \right) ,\quad \xi
\in W.
\end{eqnarray*}%
Then a change of variables gives 
\begin{equation*}
\mathrm{Op}_{\sigma ,T}\left( b\ast _{\sigma }c\right) \qquad \quad \qquad
\quad \qquad \quad \qquad \quad \qquad \quad \qquad \quad \qquad \quad
\qquad \quad \qquad \quad \qquad \quad \qquad \quad \qquad \quad
\end{equation*}%
\begin{align*}
& =\int_{W}b\left( \xi \right) \widetilde{\mathcal{W}}_{\sigma ,T}\left(
\left( T+T^{\sigma }\right) ^{-1}\xi \right) \mathrm{Op}_{\sigma ,T}\left(
c\right) \widetilde{\mathcal{W}}_{\sigma ,T}\left( -\left( T+T^{\sigma
}\right) ^{-1}\xi \right) \mathrm{d}^{\sigma }\xi \\
& =\int_{W}b\left( \xi \right) \mathcal{U}_{\sigma ,T}\left( \xi \right) 
\mathrm{Op}_{\sigma ,T}\left( c\right) \mathcal{U}_{\sigma ,T}\left( -\xi
\right) \mathrm{d}^{\sigma }\xi \\
& =\int_{W}c\left( \xi \right) \mathcal{U}_{\sigma ,T}\left( \xi \right) 
\mathrm{Op}_{\sigma ,T}\left( b\right) \mathcal{U}_{\sigma ,T}\left( -\xi
\right) \mathrm{d}^{\sigma }\xi ,
\end{align*}%
again with the first and second integrals weakly absolutely convergent and
the third interpreted in the sese of distributions.

The family $\left\{ \mathcal{U}_{\sigma ,T}\left( \xi \right) \right\} _{\xi
\in W}$ has very nice properties that can be deduced from Lemma 2.1 in \cite%
{Arsu} by means of Theorem \ref{n4}.

\begin{lemma}
Let $\left( \mathcal{H},\mathcal{W}_{\sigma ,T},\omega _{\sigma ,T}\right) $
be a $\omega _{\sigma ,T}$-representation of the symplectic space $\left(
W,\sigma \right) $, $\mathcal{S}=\mathcal{S}(\mathcal{H},\mathcal{W}_{\sigma
,T})=\mathcal{S}(\mathcal{H},\widetilde{\mathcal{W}}_{\sigma ,T})$ the space
of $\mathcal{C}^{\infty }$ vectors and $\mathrm{Op}_{\sigma ,T}$ the $T$%
-Weyl calculus. Let $\varphi ,\psi \in \mathcal{S}$.

$\left( \mathrm{a}\right) $ If $a\in \mathcal{S}^{\ast }\left( W\right) $
and $\xi \in W$, then 
\begin{equation*}
\mathcal{U}_{\sigma ,T}\left( \xi \right) \mathrm{Op}_{\sigma ,T}\left(
a\right) \mathcal{U}_{\sigma ,T}\left( -\xi \right) =\mathrm{Op}_{\sigma
,T}\left( \tau _{\xi }a\right) ,
\end{equation*}%
where $\tau _{\xi }a$ denote the translate by $\xi $ of the distribution $a$%
, i.e. $\left( \tau _{\xi }a\right) \left( \cdot \right) =a\left( \cdot -\xi
\right) $.

$\left( \mathrm{b}\right) $ If $a\in \mathcal{S}^{\ast }\left( W\right) $,
then the map 
\begin{equation*}
W\ni \xi \rightarrow \left\langle \varphi ,\mathcal{U}_{\sigma ,T}\left( \xi
\right) \mathrm{Op}_{\sigma ,T}\left( a\right) \mathcal{U}_{\sigma ,T}\left(
-\xi \right) \psi \right\rangle _{\mathcal{S},\mathcal{S}^{\ast }}\in 
\mathbb{C}
\end{equation*}%
belongs to $\mathcal{C}_{pol}^{\infty }\left( W\right) $.

$\left( \mathrm{c}\right) $ If $a\in \mathcal{S}\left( W\right) $, then the
map 
\begin{equation*}
W\ni \xi \rightarrow \left\langle \varphi ,\mathcal{U}_{\sigma ,T}\left( \xi
\right) \mathrm{Op}_{\sigma ,T}\left( a\right) \mathcal{U}_{\sigma ,T}\left(
-\xi \right) \psi \right\rangle _{\mathcal{S},\mathcal{S}^{\ast }}\in 
\mathbb{C}
\end{equation*}%
belongs to $\mathcal{S}(W)$.
\end{lemma}

\begin{proof}
Part $\left( \mathrm{a}\right) $ of the lemma follows immediately from
Theorem \ref{n4}, Lemma 2.1 in \cite{Arsu} and Corollary \ref{n12}.%
\begin{equation*}
\mathcal{U}_{\sigma ,T}\left( \xi \right) \mathrm{Op}_{\sigma ,T}\left(
a\right) \mathcal{U}_{\sigma ,T}\left( -\xi \right) \qquad \quad \qquad
\quad \qquad \quad \qquad \quad \qquad \quad \qquad \quad \qquad \quad
\qquad \quad \qquad \quad \qquad \quad \qquad \quad \qquad \quad
\end{equation*}%
\begin{eqnarray*}
&=&\widetilde{\mathcal{W}}_{\sigma ,T}\left( \left( T+T^{\sigma }\right)
^{-1}\xi \right) \mathrm{Op}^{\widetilde{w}}\left( a_{\sigma ,T}^{w}\right) 
\widetilde{\mathcal{W}}_{\sigma ,T}\left( -\left( T+T^{\sigma }\right)
^{-1}\xi \right) \quad \text{by Theorem \ref{n4}} \\
&=&\mathrm{Op}^{\widetilde{w}}\left( \tau _{\left( T+T^{\sigma }\right)
^{-1}\xi }a_{\sigma ,T}^{w}\right) \quad \text{by Lemma 2.1 in \cite{Arsu}}
\\
&=&\mathrm{Op}^{\widetilde{w}}\left( \left( \tau _{\xi }a\right) _{\sigma
,T}^{w}\right) \quad \text{by Corollary \ref{n12}} \\
&=&\mathrm{Op}_{\sigma ,T}\left( \tau _{\xi }a\right) \quad \text{by Theorem %
\ref{n4}}.
\end{eqnarray*}%
Here we applied Corollary \ref{n12} for $S=T+T^{\sigma }$ and $\lambda
=\lambda _{{\sigma ,T}}\left( \cdot \right) =\mathrm{e}^{-\frac{\mathrm{i}}{2%
}\theta _{{\sigma ,T}}\left( \cdot \right) }$ as follows:%
\begin{eqnarray*}
a_{{\sigma ,T}}^{w} &=&\tau _{\left( T+T^{\sigma }\right) ^{-1}\xi }\left(
T+T^{\sigma }\right) ^{\ast }\left( \lambda _{{\sigma ,T}}\left( D_{\sigma
}\right) \left( a\right) \right) \\
&=&\left( T+T^{\sigma }\right) ^{\ast }\left( \lambda _{{\sigma ,T}}\left(
D_{\sigma }\right) \left( \tau _{\xi }a\right) \right) \\
&=&\left( \tau _{\xi }a\right) _{\sigma ,T}^{w}.
\end{eqnarray*}

For $\left( \mathrm{b}\right) $ and $\left( \mathrm{c}\right) $ we use $%
\left( \mathrm{a}\right) $ and Lemma \ref{n2}. By this lemma we know that $%
w=w_{\varphi ,\psi }=\left\langle \varphi ,\mathcal{W}_{\sigma ,T}(\cdot
)\psi \right\rangle _{\mathcal{S},\mathcal{S}^{\ast }}\in \mathcal{S}\left(
W\right) $. Assume that $a\in \mathcal{S}^{\ast }\left( W\right) $. Then,
writing $\widehat{a}$ for $\mathcal{F}_{\sigma }a$, we have 
\begin{equation*}
\left\langle \varphi ,\mathcal{U}_{\sigma ,T}\left( \xi \right) \mathrm{Op}%
_{\sigma ,T}\left( a\right) \mathcal{U}_{\sigma ,T}\left( -\xi \right) \psi
\right\rangle _{\mathcal{S},\mathcal{S}^{\ast }}=\left\langle \varphi ,%
\mathrm{Op}_{\sigma ,T}\left( \tau _{\xi }a\right) \psi \right\rangle _{%
\mathcal{S},\mathcal{S}^{\ast }}
\end{equation*}%
and%
\begin{eqnarray*}
\left\langle \varphi ,\mathrm{Op}_{\sigma ,T}\left( \tau _{\xi }a\right)
\psi \right\rangle _{\mathcal{S},\mathcal{S}^{\ast }} &=&\left\langle 
\overline{\left\langle \varphi ,\mathcal{W}_{\sigma ,T}\left( \cdot \right)
\psi \right\rangle }_{\mathcal{S},\mathcal{S}^{\ast }},\widehat{\tau _{\xi }a%
}\right\rangle _{\mathcal{S}\left( W\right) ,\mathcal{S}^{\ast }\left(
W\right) }^{\sigma } \\
&=&\left\langle \overline{w},\widehat{\tau _{\xi }a}\right\rangle _{\mathcal{%
S}\left( W\right) ,\mathcal{S}^{\ast }\left( W\right) }^{\sigma } \\
&=&\left\langle \tau _{-\xi }\widehat{\overline{w}},a\right\rangle _{%
\mathcal{S}\left( W\right) ,\mathcal{S}^{\ast }\left( W\right) }^{\sigma
}=\left\langle \tau _{-\xi }\overline{\widehat{w}}\left( -\cdot -\xi \right)
,a\right\rangle _{\mathcal{S}\left( W\right) ,\mathcal{S}^{\ast }\left(
W\right) }^{\sigma } \\
&=&.
\end{eqnarray*}%
Hence 
\begin{equation}
\left\langle \varphi ,\mathcal{U}_{\sigma ,T}\left( \xi \right) \mathrm{Op}%
_{\sigma ,T}\left( a\right) \mathcal{U}_{\sigma ,T}\left( -\xi \right) \psi
\right\rangle _{\mathcal{S},\mathcal{S}^{\ast }}=a\ast _{\sigma }\widehat{w}%
\left( -\xi \right) ,\quad \xi \in W,  \label{K1}
\end{equation}%
and $\left( \mathrm{b}\right) $ and $\left( \mathrm{c}\right) $ follows at
once from this equality.
\end{proof}

\begin{theorem}[Kato's identity]
\label{n14}Let $\left( \mathcal{H},\mathcal{W}_{\sigma ,T},\omega _{\sigma
,T}\right) $ be a $\omega _{\sigma ,T}$-representation of the symplectic
space $\left( W,\sigma \right) $, $\mathcal{S}=\mathcal{S}(\mathcal{H},%
\mathcal{W}_{\sigma ,T})=\mathcal{S}(\mathcal{H},\widetilde{\mathcal{W}}%
_{\sigma ,T})$ the space of $\mathcal{C}^{\infty }$ vectors and $\mathrm{Op}%
_{\sigma ,T}$ the $T$-Weyl calculus.

$\left( \mathrm{a}\right) $ If $b\in \mathcal{S}\left( W\right) $, $c\in 
\mathcal{S}^{\ast }\left( W\right) $, then $b\ast _{\sigma }c\in \mathcal{S}%
^{\ast }\left( W\right) $ and 
\begin{align*}
\mathrm{Op}_{\sigma ,T}\left( b\ast _{\sigma }c\right) & =\int_{W}b\left(
\xi \right) \mathcal{U}_{\sigma ,T}\left( \xi \right) \mathrm{Op}_{\sigma
,T}\left( c\right) \mathcal{U}_{\sigma ,T}\left( -\xi \right) \mathrm{d}%
^{\sigma }\xi \\
& =\int_{W}c\left( \xi \right) \mathcal{U}_{\sigma ,T}\left( \xi \right) 
\mathrm{Op}_{\sigma ,T}\left( b\right) \mathcal{U}_{\sigma ,T}\left( -\xi
\right) \mathrm{d}^{\sigma }\xi ,
\end{align*}%
where the first integral is weakly absolutely convergent while the second
one must be interpreted in the sese of distributions and represents the
operator defined by%
\begin{align*}
& \left\langle \varphi ,\left( \int_{W}c\left( \xi \right) \mathcal{U}%
_{\sigma ,T}\left( \xi \right) \mathrm{Op}_{\sigma ,T}\left( b\right) 
\mathcal{U}_{\sigma ,T}\left( -\xi \right) \mathrm{d}^{\sigma }\xi \right)
\psi \right\rangle _{\mathcal{S},\mathcal{S}^{\ast }} \\
& =\left\langle \overline{\left\langle \varphi ,\mathcal{U}_{\sigma
,T}\left( \cdot \right) \mathrm{Op}_{\sigma ,T}\left( b\right) \mathcal{U}%
_{\sigma ,T}\left( -\cdot \right) \psi \right\rangle }_{\mathcal{S},\mathcal{%
S}^{\ast }},c\right\rangle _{\mathcal{S}\left( W\right) ,\mathcal{S}^{\ast
}\left( W\right) }^{\sigma },
\end{align*}%
for all $\varphi $, $\psi \in \mathcal{S}$.

$\left( \mathrm{b}\right) $ Let $h\in \mathcal{C}_{\mathrm{pol}}^{\infty
}\left( W\right) $. If $b\in L^{p}\left( W\right) $ and $c\in L^{q}\left(
W\right) $, where $1\leq p,q\leq \infty $ and $p^{-1}+q^{-1}\geq 1$, then $%
b\ast _{\sigma }c\in L^{r}\left( W\right) $, $r^{-1}=p^{-1}+q^{-1}-1$ and 
\begin{equation}
\mathrm{Op}_{\sigma ,T}\left( h\left( D_{\sigma }\right) \left( b\ast
_{\sigma }c\right) \right) =\int_{W}b\left( \xi \right) \mathcal{U}_{\sigma
,T}\left( \xi \right) \mathrm{Op}_{\sigma ,T}\left( h\left( D_{\sigma
}\right) c\right) \mathcal{U}_{\sigma ,T}\left( -\xi \right) \mathrm{d}%
^{\sigma }\xi ,  \label{K2}
\end{equation}%
where the integral is weakly absolutely convergent.
\end{theorem}

\begin{proof}
$\left( \mathrm{a}\right) $ Let $\varphi $, $\psi \in \mathcal{S}$. Then $%
w=w_{\varphi ,\psi }=\left\langle \varphi ,\mathcal{W}_{\sigma ,T}(\cdot
)\psi \right\rangle _{\mathcal{S},\mathcal{S}^{\ast }}\in \mathcal{S}\left(
W\right) $.

First we consider the case when $b$, $c\in \mathcal{S}\left( W\right) $.
Then, writing $\widehat{a}$ for $\mathcal{F}_{\sigma }a$, we have 
\begin{eqnarray*}
\left\langle \varphi ,\mathrm{Op}_{\sigma ,T}\left( b\ast _{\sigma }c\right)
\psi \right\rangle _{\mathcal{S},\mathcal{S}^{\ast }} &=&\int_{W}\widehat{%
b\ast _{\sigma }c}\left( \eta \right) \left\langle \varphi ,\mathcal{W}%
_{\sigma ,T}\left( \eta \right) \psi \right\rangle _{\mathcal{S},\mathcal{S}%
^{\ast }}\mathrm{d}^{\sigma }\eta \\
&=&\int_{W}\widehat{b\ast _{\sigma }c}\left( \eta \right) w\left( \eta
\right) \mathrm{d}^{\sigma }\eta =\int_{W}\left( b\ast _{\sigma }c\right)
\left( \eta \right) \widehat{w}\left( -\eta \right) \mathrm{d}^{\sigma }\eta
\\
&=&\int_{W}\left( \int_{W}b\left( \xi \right) c\left( \eta -\xi \right) 
\mathrm{d}^{\sigma }\xi \right) \widehat{w}\left( -\eta \right) \mathrm{d}%
^{\sigma }\eta \\
&=&\int_{W}b\left( \xi \right) \left( \int_{W}c\left( \eta -\xi \right) 
\widehat{w}\left( -\eta \right) \mathrm{d}^{\sigma }\eta \right) \mathrm{d}%
^{\sigma }\xi \\
&=&\int_{W}b\left( \xi \right) \left( \int_{W}\left( \tau _{\xi }c\right)
\left( \eta \right) \widehat{w}\left( -\eta \right) \mathrm{d}^{\sigma }\eta
\right) \mathrm{d}^{\sigma }\xi \\
&=&\int_{W}b\left( \xi \right) \left( \int_{W}\widehat{\tau _{\xi }c}\left(
\eta \right) w\left( \eta \right) \mathrm{d}^{\sigma }\eta \right) \mathrm{d}%
^{\sigma }\xi \\
&=&\int_{W}b\left( \xi \right) \left( \int_{W}\widehat{\tau _{\xi }c}\left(
\eta \right) \left\langle \varphi ,\mathcal{W}_{\sigma ,T}\left( \eta
\right) \psi \right\rangle _{\mathcal{S},\mathcal{S}^{\ast }}\mathrm{d}%
^{\sigma }\eta \right) \mathrm{d}^{\sigma }\xi \\
&=&\int_{W}b\left( \xi \right) \left\langle \varphi ,\mathrm{Op}_{\sigma
,T}\left( \tau _{\xi }c\right) \psi \right\rangle _{\mathcal{S},\mathcal{S}%
^{\ast }}\mathrm{d}^{\sigma }\xi \\
&=&\int_{W}b\left( \xi \right) \left\langle \varphi ,\mathcal{U}_{\sigma
,T}\left( \xi \right) \mathrm{Op}_{\sigma ,T}\left( c\right) \mathcal{U}%
_{\sigma ,T}\left( -\xi \right) \psi \right\rangle _{\mathcal{S},\mathcal{S}%
^{\ast }}\mathrm{d}^{\sigma }\xi ,
\end{eqnarray*}%
where in the last equality we used the formula 
\begin{equation*}
\mathrm{Op}_{\sigma ,T}\left( \tau _{\xi }c\right) =\mathcal{U}_{\sigma
,T}\left( \xi \right) \mathrm{Op}_{\sigma ,T}\left( c\right) \mathcal{U}%
_{\sigma ,T}\left( -\xi \right) .
\end{equation*}

Let $c\in \mathcal{S}^{\ast }\left( W\right) $. Then 
\begin{eqnarray*}
\left\langle \varphi ,\mathcal{U}_{\sigma ,T}\left( \xi \right) \mathrm{Op}%
_{\sigma ,T}\left( c\right) \mathcal{U}_{\sigma ,T}\left( -\xi \right) \psi
\right\rangle _{\mathcal{S},\mathcal{S}^{\ast }} &=&\left\langle \varphi ,%
\mathrm{Op}_{\sigma ,T}\left( \tau _{\xi }c\right) \psi \right\rangle _{%
\mathcal{S},\mathcal{S}^{\ast }} \\
&=&\left\langle \overline{w},\widehat{\tau _{\xi }c}\right\rangle _{\mathcal{%
S}\left( W\right) ,\mathcal{S}^{\ast }\left( W\right) }^{\sigma } \\
&=&\left\langle \overline{w},\mathrm{e}^{-\mathrm{i}\sigma \left( \cdot ,\xi
\right) }\widehat{c}\right\rangle _{\mathcal{S}\left( W\right) ,\mathcal{S}%
^{\ast }\left( W\right) }^{\sigma } \\
&=&\left\langle \overline{\mathrm{e}^{\mathrm{i}\sigma \left( \cdot ,\xi
\right) }w},\widehat{c}\right\rangle _{\mathcal{S}\left( W\right) ,\mathcal{S%
}^{\ast }\left( W\right) }^{\sigma }.
\end{eqnarray*}%
If $\left\{ c_{j}\right\} \subset \mathcal{S}\left( W\right) $ is such that $%
c_{j}\rightarrow c$ weakly in $\mathcal{S}^{\ast }\left( W\right) $, then%
\begin{equation*}
\left\langle \varphi ,\mathcal{U}_{\sigma ,T}\left( \xi \right) \mathrm{Op}%
_{\sigma ,T}\left( c_{j}\right) \mathcal{U}_{\sigma ,T}\left( -\xi \right)
\psi \right\rangle _{\mathcal{S},\mathcal{S}^{\ast }}\rightarrow
\left\langle \varphi ,\mathcal{U}_{\sigma ,T}\left( \xi \right) \mathrm{Op}%
_{\sigma ,T}\left( c\right) \mathcal{U}_{\sigma ,T}\left( -\xi \right) \psi
\right\rangle _{\mathcal{S},\mathcal{S}^{\ast }},
\end{equation*}%
for every $\xi \in W$, and the uniform boundedness principle implies that
there are $M\in 
\mathbb{N}
$\textit{, }$C=C\left( M,w\right) >0$\textit{\ }such that\textit{\ }%
\begin{equation*}
\left\vert \left\langle \varphi ,\mathcal{U}_{\sigma ,T}\left( \xi \right) 
\mathrm{Op}_{\sigma ,T}\left( c_{j}\right) \mathcal{U}_{\sigma ,T}\left(
-\xi \right) \psi \right\rangle _{\mathcal{S},\mathcal{S}^{\ast
}}\right\vert \leq C\left\langle \xi \right\rangle ^{M},\quad \xi \in W,
\end{equation*}%
where $\left\langle \xi \right\rangle =\left( 1+\left\vert \xi \right\vert
^{2}\right) ^{\frac{1}{2}}$ and $\left\vert \cdot \right\vert $ is an
euclidean norm on $W$.

The general case can be deduced from the above case if we observe that%
\begin{equation*}
\left\langle \varphi ,\mathrm{Op}_{\sigma ,T}\left( b\ast _{\sigma
}c_{j}\right) \psi \right\rangle _{\mathcal{S},\mathcal{S}^{\ast
}}\rightarrow \left\langle \varphi ,\mathrm{Op}_{\sigma ,T}\left( b\ast
_{\sigma }c\right) \psi \right\rangle _{\mathcal{S},\mathcal{S}^{\ast }},
\end{equation*}%
\begin{multline*}
\left\langle \overline{\left\langle \varphi ,\mathcal{U}_{\sigma ,T}\left(
\cdot \right) \mathrm{Op}_{\sigma ,T}\left( b\right) \mathcal{U}_{\sigma
,T}\left( -\cdot \right) \psi \right\rangle }_{\mathcal{S},\mathcal{S}^{\ast
}},c_{j}\right\rangle _{\mathcal{S}\left( W\right) ,\mathcal{S}^{\ast
}\left( W\right) }^{\sigma } \\
\rightarrow \left\langle \overline{\left\langle \varphi ,\mathcal{U}_{\sigma
,T}\left( \cdot \right) \mathrm{Op}_{\sigma ,T}\left( b\right) \mathcal{U}%
_{\sigma ,T}\left( -\cdot \right) \psi \right\rangle }_{\mathcal{S},\mathcal{%
S}^{\ast }},c\right\rangle _{\mathcal{S}\left( W\right) ,\mathcal{S}^{\ast
}\left( W\right) }^{\sigma }
\end{multline*}%
and that the sequence $\left\{ b\left( \cdot \right) \left\langle \varphi ,%
\mathcal{U}_{\sigma ,T}\left( \cdot \right) \mathrm{Op}_{\sigma ,T}\left(
c_{j}\right) \mathcal{U}_{\sigma ,T}\left( -\cdot \right) \psi \right\rangle
_{\mathcal{S},\mathcal{S}^{\ast }}\right\} $ converge dominated to $b\left(
\cdot \right) \left\langle \varphi ,\mathcal{U}_{\sigma ,T}\left( \cdot
\right) \mathrm{Op}_{\sigma ,T}\left( c\right) \mathcal{U}_{\sigma ,T}\left(
-\cdot \right) \psi \right\rangle _{\mathcal{S},\mathcal{S}^{\ast }}$.

$\left( \mathrm{b}\right) $ We recall the Young inequality. If $1\leq
p,q\leq \infty $, $p^{-1}+q^{-1}\geq 1$, $r^{-1}=p^{-1}+q^{-1}-1$, $b\in
L^{p}\left( W\right) $ and $c\in L^{q}\left( W\right) $, then $b\ast
_{\sigma }c\in L^{r}\left( W\right) $ and 
\begin{equation*}
\left\Vert b\ast _{\sigma }c\right\Vert _{L^{r}\left( W\right) }\leq
\left\Vert b\right\Vert _{L^{p}\left( W\right) }\left\Vert c\right\Vert
_{L^{q}\left( W\right) }.
\end{equation*}%
Let $p^{\prime },r^{\prime }\geq 1$ such that $p^{-1}+p^{\prime
-1}=r^{-1}+r^{\prime -1}=1$. Then $r^{\prime -1}+q^{-1}\geq 1$ and $%
p^{\prime -1}=r^{\prime -1}+q^{-1}-1$.

If $b\in \mathcal{S}\left( W\right) $ and $c\in L^{q}\left( W\right) $, then 
$g=h\left( D_{\sigma }\right) c\in \mathcal{S}^{\ast }\left( W\right) $ and $%
h\left( D_{\sigma }\right) \left( b\ast _{\sigma }c\right) =b\ast _{\sigma
}h\left( D_{\sigma }\right) c=b\ast _{\sigma }g$. Using $\left( \mathrm{a}%
\right) $ it follows that%
\begin{equation*}
\left\langle \varphi ,\mathrm{Op}_{\sigma ,T}\left( b\ast _{\sigma }g\right)
\psi \right\rangle _{\mathcal{S},\mathcal{S}^{\ast }}=\int_{W}b\left( \xi
\right) \left\langle \varphi ,\mathcal{U}_{\sigma ,T}\left( \xi \right) 
\mathrm{Op}_{\sigma ,T}\left( g\right) \mathcal{U}_{\sigma ,T}\left( -\xi
\right) \psi \right\rangle _{\mathcal{S},\mathcal{S}^{\ast }}\mathrm{d}%
^{\sigma }\xi .
\end{equation*}%
Similarly, if $b\in L^{p}\left( W\right) \subset \mathcal{S}^{\ast }\left(
W\right) $ and $c\in \mathcal{S}\left( W\right) $, then $g=h\left( D_{\sigma
}\right) c\in \mathcal{S}\left( W\right) $ and $h\left( D_{\sigma }\right)
\left( b\ast _{\sigma }c\right) =b\ast _{\sigma }h\left( D_{\sigma }\right)
c=b\ast _{\sigma }g$. Using $\left( \mathrm{a}\right) $ again, it follows
that 
\begin{align*}
\left\langle \varphi ,\mathrm{Op}_{\sigma ,T}\left( b\ast _{\sigma }g\right)
\psi \right\rangle _{\mathcal{S},\mathcal{S}^{\ast }}& =\left\langle 
\overline{\left\langle \varphi ,\mathcal{U}_{\sigma ,T}\left( \cdot \right) 
\mathrm{Op}_{\sigma ,T}\left( g\right) \mathcal{U}_{\sigma ,T}\left( -\cdot
\right) \psi \right\rangle }_{\mathcal{S},\mathcal{S}^{\ast
}},b\right\rangle _{\mathcal{S}\left( W\right) ,\mathcal{S}^{\ast }\left(
W\right) }^{\sigma } \\
& =\int_{W}b\left( \xi \right) \left\langle \varphi ,\mathcal{U}_{\sigma
,T}\left( \xi \right) \mathrm{Op}_{\sigma ,T}\left( g\right) \mathcal{U}%
_{\sigma ,T}\left( -\xi \right) \psi \right\rangle _{\mathcal{S},\mathcal{S}%
^{\ast }}\mathrm{d}^{\sigma }\xi .
\end{align*}%
Hence we proved $\left( \mathrm{b}\right) $ in the case when either $b\in 
\mathcal{S}\left( W\right) $ or $c\in \mathcal{S}\left( W\right) $.

The general case can be obtained from these particular cases by an
approximation argument. Observe that condition $p^{-1}+q^{-1}\geq 1$ implies
that $p<\infty $ or $q<\infty $. For the convergence of the left-hand side
of (\ref{K2}) we use the Young inequality and the continuity of the map $%
\mathrm{Op}_{\sigma ,T}$. To estimate right-hand side of (\ref{K2}) we use (%
\ref{K1}) and the H\"{o}lder and Young inequalities. We have%
\begin{eqnarray*}
\left\langle \varphi ,\mathcal{U}_{\sigma ,T}\left( \xi \right) \mathrm{Op}%
_{\sigma ,T}\left( h\left( D_{\sigma }\right) c\right) \mathcal{U}_{\sigma
,T}\left( -\xi \right) \psi \right\rangle _{\mathcal{S},\mathcal{S}^{\ast }}
&=&\left( h\left( D_{\sigma }\right) c\ast \widehat{w}\right) \left( -\xi
\right) \\
&=&\left( c\ast \widehat{hw}\right) \left( -\xi \right)
\end{eqnarray*}%
and%
\begin{multline*}
\left\vert \left\langle \varphi ,\left( \int_{W}b\left( \xi \right) \mathcal{%
U}_{\sigma ,T}\left( \xi \right) \mathrm{Op}_{\sigma ,T}\left( h\left(
D_{\sigma }\right) c\right) \mathcal{U}_{\sigma ,T}\left( -\xi \right) 
\mathrm{d}^{\sigma }\xi \right) \psi \right\rangle _{\mathcal{S},\mathcal{S}%
^{\ast }}\right\vert \\
=\left\vert \left( \int_{W}b\left( \xi \right) \left( c\ast _{\sigma }%
\widehat{hw}\right) \left( -\xi \right) \mathrm{d}^{\sigma }\xi \right)
\right\vert \leq \left\Vert b\right\Vert _{L^{p}\left( W\right) }\left\Vert
c\ast _{\sigma }\widehat{hw}\right\Vert _{L^{p^{\prime }}\left( W\right) } \\
\leq \left\Vert b\right\Vert _{L^{p}\left( W\right) }\left\Vert c\right\Vert
_{L^{q}\left( W\right) }\left\Vert \widehat{hw}\right\Vert _{L^{r^{\prime
}}\left( W\right) },
\end{multline*}%
where $w=w_{\varphi ,\psi }=\left\langle \varphi ,\mathcal{W}_{\sigma
,T}(\cdot )\psi \right\rangle _{\mathcal{S},\mathcal{S}^{\ast }}\in \mathcal{%
S}\left( W\right) $.
\end{proof}

\begin{remark}
The last two results are true for $\omega _{\sigma ,T}$-representations that
are not necessarily irreducible.
\end{remark}

\section{Kato's operator calculus}

In \cite{Cordes}, H.O. Cordes noticed that the $L^{2}$-boundedness of an
operator $a\left( x,D\right) $ in $OPS_{0,0}^{0}$ could be deduced by a
synthesis of $a\left( x,D\right) $ from trace-class operators. In \cite{Kato}%
, T. Kato extended this argument to the general case $OPS_{\rho ,\rho }^{0}$%
, $0<\rho <1$, and abstracted the functional analysis involved in Cordes'
argument. This operator calculus can be extended further to investigate the
Schatten-class properties of operators in the $T$-Weyl calculus for an
irreducible $\omega _{\sigma ,T}$-representation $\left( \mathcal{H},%
\mathcal{W}_{\sigma ,T},\omega _{\sigma ,T}\right) $ of $W$.

Let $\mathcal{H}$ be a separable Hilbert space. For $1\leq p<\infty $, we
denote by $\mathcal{B}_{p}\left( \mathcal{H}\right) $ the Scatten ideal of
compact operators on $\mathcal{H}$ whose singular values lie in $l^{p}$ with
the associated norm $\left\Vert \cdot \right\Vert _{p}$. For $p=\infty $, $%
\mathcal{B}_{\infty }\left( \mathcal{H}\right) $ is the ideal of compact
operators on $\mathcal{H}$ with $\left\Vert \cdot \right\Vert _{\infty
}=\left\Vert \cdot \right\Vert $.

\begin{definition}
Let $T,A,B\in \mathcal{B}\left( \mathcal{H}\right) $, $A\geq 0$, $B\geq 0$.
We write 
\begin{equation*}
T\ll \left( A;B\right) \overset{def}{\Longleftrightarrow }\left\vert \left(
u,Tv\right) \right\vert ^{2}\leq \left( u,Au\right) \left( v,Bv\right)
,\quad \text{for }u,v\in \mathcal{H}.
\end{equation*}
\end{definition}

\begin{lemma}
Let $S,T,A,B\in \mathcal{B}\left( \mathcal{H}\right) $, $A\geq 0$, $B\geq 0$%
. Then

$\left( \mathrm{i}\right) $ $T\ll \left( \left\vert T^{\ast }\right\vert
;\left\vert T\right\vert \right) $.

$\left( \mathrm{ii}\right) $ $T\ll \left( A;B\right) \Rightarrow T^{\ast
}\ll \left( B;A\right) $.

$\left( \mathrm{iii}\right) $ $T\ll \left( A;B\right) \Rightarrow S^{\ast
}TS\ll \left( S^{\ast }AS;S^{\ast }BS\right) .$
\end{lemma}

\begin{proof}
For a proof see Lemma 3.2 in \cite{Arsu}.
\end{proof}

\begin{lemma}
Let $Y$ be a measure space and $Y\ni y\rightarrow U\left( y\right) \in 
\mathcal{B}\left( \mathcal{H}\right) $ a weakly measurable map.

$\left( \mathrm{a}\right) $ Assume that there is $C>0$ such that 
\begin{equation*}
\int_{Y}\left\vert \left( \varphi ,U\left( y\right) \psi \right) \right\vert
^{2}\mathrm{d}y\leq C\left\Vert \varphi \right\Vert ^{2}\left\Vert \psi
\right\Vert ^{2},\quad \varphi ,\psi \in \mathcal{H}.
\end{equation*}%
If $b\in L^{\infty }\left( Y\right) $ and $G\in \mathcal{B}_{1}\left( 
\mathcal{H}\right) $, then the integral%
\begin{equation*}
b\left\{ G\right\} =\int_{Y}b\left( y\right) U\left( y\right) ^{\ast
}GU\left( y\right) \mathrm{d}y
\end{equation*}%
is weakly absolutely convergent and defines a bounded operator such that%
\begin{equation*}
\left\Vert b\left\{ G\right\} \right\Vert \leq C\left\Vert b\right\Vert
_{L^{\infty }}\left\Vert G\right\Vert _{1}.
\end{equation*}

$\left( \mathrm{b}\right) $ Assume that there is $C>0$ such that 
\begin{equation*}
\left\Vert U\left( y\right) \right\Vert \leq C^{\frac{1}{2}}\quad a.e.\ y\in
Y.
\end{equation*}%
If $b\in L^{1}\left( Y\right) $ and $G\in \mathcal{B}_{1}\left( \mathcal{H}%
\right) $, then the integral%
\begin{equation*}
b\left\{ G\right\} =\int_{Y}b\left( y\right) U\left( y\right) ^{\ast
}GU\left( y\right) \mathrm{d}y
\end{equation*}%
is absolutely convergent and defines a trace class operator such that%
\begin{equation*}
\left\Vert b\left\{ G\right\} \right\Vert _{1}\leq C\left\Vert b\right\Vert
_{L^{1}}\left\Vert G\right\Vert _{1}.
\end{equation*}

$\left( \mathrm{c}\right) $ Assume that there is $C>0$ such that 
\begin{equation*}
\left\Vert U\left( y\right) \right\Vert \leq C_{1}^{\frac{1}{2}}\quad a.e.\
y\in Y,\quad \text{and}
\end{equation*}%
\begin{equation*}
\int_{Y}\left\vert \left( \varphi ,U\left( y\right) \psi \right) \right\vert
^{2}\mathrm{d}y\leq C_{\infty }\left\Vert \varphi \right\Vert ^{2}\left\Vert
\psi \right\Vert ^{2},\quad \varphi ,\psi \in \mathcal{H}.
\end{equation*}%
If $b\in L^{p}\left( Y\right) $ with $1\leq p<\infty $ and $G\in \mathcal{B}%
_{1}\left( \mathcal{H}\right) $, then the integral%
\begin{equation*}
b\left\{ G\right\} =\int_{Y}b\left( y\right) U\left( y\right) ^{\ast
}GU\left( y\right) \mathrm{d}y
\end{equation*}%
is weakly absolutely convergent and defines an operator $b\left\{ G\right\} $
in $\mathcal{B}_{p}\left( \mathcal{H}\right) $ and%
\begin{equation*}
\left\Vert b\left\{ G\right\} \right\Vert _{p}\leq C_{1}^{\frac{1}{p}%
}C_{\infty }^{1-\frac{1}{p}}\left\Vert b\right\Vert _{L^{p}}\left\Vert
G\right\Vert _{1}.
\end{equation*}
\end{lemma}

\begin{proof}
For a proof see Lemma 3.3 in \cite{Arsu}.
\end{proof}

\begin{lemma}
Let $\left( \mathcal{H},\mathcal{W}_{\sigma ,T},\omega _{\sigma ,T}\right) $
be an irreducible $\omega _{\sigma ,T}$-representation of $W$ and let $%
\left\{ \mathcal{U}_{\sigma ,T}\left( \xi \right) \right\} _{\xi \in W}$ be
the family of unitary operators 
\begin{equation*}
\mathcal{U}_{\sigma ,T}\left( \xi \right) =\widetilde{\mathcal{W}}_{\sigma
,T}\left( \left( T+T^{\sigma }\right) ^{-1}\xi \right) ,\quad \xi \in W.
\end{equation*}%
If $\varphi ,\psi \in \mathcal{H}$, then the map $\xi \rightarrow \left(
\varphi ,\mathcal{U}_{\sigma ,T}\left( \xi \right) \psi \right) _{\mathcal{H}%
}$ belongs to $L^{2}\left( W\right) \cap \mathcal{C}_{\mathcal{\infty }}(W)$
and 
\begin{equation*}
\int_{W}\left\vert \left( \varphi ,\mathcal{U}_{\sigma ,T}\left( \xi \right)
\psi \right) \right\vert ^{2}\mathrm{d}^{\sigma }\xi =\left( \det \left(
T+T^{\sigma }\right) \right) ^{\frac{1}{2}}\left\Vert \varphi \right\Vert
^{2}\left\Vert \psi \right\Vert ^{2},
\end{equation*}%
and%
\begin{equation*}
\left\Vert \left( \varphi ,\mathcal{U}_{\sigma ,T}\left( \cdot \right) \psi
\right) \right\Vert _{\infty }\leq \left\Vert \varphi \right\Vert \left\Vert
\psi \right\Vert .
\end{equation*}%
If $\varphi ,\psi ,\varphi ^{\prime },\psi ^{\prime }\in \mathcal{H}$, then%
\begin{equation*}
\int_{W}\overline{\left( \varphi ^{\prime },\mathcal{U}_{\sigma ,T}\left(
\xi \right) \psi ^{\prime }\right) }\left( \varphi ,\mathcal{U}_{\sigma
,T}\left( \xi \right) \psi \right) \mathrm{d}^{\sigma }\xi =\left( \det
\left( T+T^{\sigma }\right) \right) ^{\frac{1}{2}}\left( \varphi ,\varphi
^{\prime }\right) \left( \psi ^{\prime },\psi \right) .
\end{equation*}
\end{lemma}

\begin{proof}
Since $\left( \mathcal{H},\mathcal{W}_{\sigma ,T},\omega _{\sigma ,T}\right) 
$ is an irreducible $\omega _{\sigma ,T}$-representation if and only if $%
\left( \mathcal{H},\widetilde{\mathcal{W}}_{\sigma ,T},\widetilde{\omega }%
_{\sigma ,T},\left( W,\sigma _{T+T^{{\sigma }}}\right) \right) $ is an
irreducible Weyl system, the lemma can be deduced from Lemma 3.4 in \cite%
{Arsu} by noticing that 
\begin{eqnarray*}
\int_{W}\left\vert \left( \varphi ,\mathcal{U}_{\sigma ,T}\left( \xi \right)
\psi \right) \right\vert ^{2}\mathrm{d}^{\sigma }\xi &=&\int_{W}\left\vert
\left( \varphi ,\widetilde{\mathcal{W}}_{\sigma ,T}\left( \left( T+T^{\sigma
}\right) ^{-1}\xi \right) \psi \right) \right\vert ^{2}\mathrm{d}^{\sigma
}\xi \\
&=&\det \left( T+T^{\sigma }\right) \int_{W}\left\vert \left( \varphi ,%
\widetilde{\mathcal{W}}_{\sigma ,T}\left( \eta \right) \psi \right)
\right\vert ^{2}\mathrm{d}^{\sigma }\eta \\
&=&\left( \det \left( T+T^{\sigma }\right) \right) ^{\frac{1}{2}%
}\int_{W}\left\vert \left( \varphi ,\widetilde{\mathcal{W}}_{\sigma
,T}\left( \eta \right) \psi \right) \right\vert ^{2}\mathrm{d}^{\sigma
_{T+T^{\sigma }}}\eta \\
&=&\left( \det \left( T+T^{\sigma }\right) \right) ^{\frac{1}{2}}\left\Vert
\varphi \right\Vert ^{2}\left\Vert \psi \right\Vert ^{2}\quad \text{by Lemma
3.4 in \cite{Arsu}.}
\end{eqnarray*}%
The last formula is a consequence of polarization identity.
\end{proof}

\begin{theorem}[Kato's operator calculus]
\label{n15}Let $\left( \mathcal{H},\mathcal{W}_{\sigma ,T},\omega _{\sigma
,T}\right) $ be an irreducible $\omega _{\sigma ,T}$-representation of $W$.

$\left( \mathrm{a}\right) $ If $b\in L^{\infty }\left( W\right) $ and $G\in 
\mathcal{B}_{1}\left( \mathcal{H}\right) $, then the integral%
\begin{equation*}
b\left\{ G\right\} =\int_{W}b\left( \xi \right) \mathcal{U}_{\sigma
,T}\left( \xi \right) G\mathcal{U}_{\sigma ,T}\left( -\xi \right) \mathrm{d}%
^{\sigma }\xi
\end{equation*}%
is weakly absolutely convergent and defines a bounded operator such that%
\begin{equation*}
\left\Vert b\left\{ G\right\} \right\Vert \leq \left( \det \left(
T+T^{\sigma }\right) \right) ^{\frac{1}{2}}\left\Vert b\right\Vert
_{L^{\infty }}\left\Vert G\right\Vert _{1}.
\end{equation*}%
Moreover, if $b$ vanishes at $\infty $ in the sense that for any $%
\varepsilon >0$ there is a compact subset $K$ of $W$ such that 
\begin{equation*}
\left\Vert b\right\Vert _{L^{\infty }\left( W\smallsetminus K\right) }\leq
\varepsilon ,
\end{equation*}%
then $b\left\{ G\right\} $ is a compact operator.

The mapping $\left( b,G\right) \rightarrow b\left\{ G\right\} $ has the
following properties.

$\left( \mathrm{i}\right) $ $b\geq 0,G\geq 0\Rightarrow b\left\{ G\right\}
\geq 0.$

$\left( \mathrm{ii}\right) $ $1\left\{ G\right\} =\left( \det \left(
T+T^{\sigma }\right) \right) ^{\frac{1}{2}}\mathrm{Tr}\left( G\right) 
\mathrm{id}_{\mathcal{H}}$.

$\left( \mathrm{iii}\right) $ $\left( b_{1}b_{2}\right) \left\{ G\right\}
\ll \left( \left\vert b_{1}\right\vert ^{2}\left\{ \left\vert G^{\ast
}\right\vert \right\} ;\left\vert b_{2}\right\vert ^{2}\left\{ \left\vert
G\right\vert \right\} \right) $.

$\left( \mathrm{b}\right) $ If $b\in L^{p}\left( W\right) $ with $1\leq
p<\infty $ and $G\in \mathcal{B}_{1}\left( \mathcal{H}\right) $, then the
integral%
\begin{equation*}
b\left\{ G\right\} =\int_{W}b\left( \xi \right) \mathcal{U}_{\sigma
,T}\left( \xi \right) G\mathcal{U}_{\sigma ,T}\left( -\xi \right) \mathrm{d}%
^{\sigma }\xi
\end{equation*}%
is weakly absolutely convergent and defines an operator $b\left\{ G\right\} $
in $\mathcal{B}_{p}\left( \mathcal{H}\right) $ and%
\begin{equation*}
\left\Vert b_{{\sigma ,T}}\left\{ G\right\} \right\Vert _{p}\leq \left( \det
\left( T+T^{\sigma }\right) \right) ^{\frac{1}{2}\left( 1-\frac{1}{p}\right)
}\left\Vert b\right\Vert _{L^{p}}\left\Vert G\right\Vert _{1}.
\end{equation*}
\end{theorem}

\begin{proof}
$\left( \mathrm{ii}\right) $ Let $G=|\varphi )(\psi |=\left( \psi ,\cdot
\right) \varphi ,$ $\varphi ,\psi \in \mathcal{H}$. Then 
\begin{equation*}
\mathcal{U}_{\sigma ,T}\left( \xi \right) G\mathcal{U}_{\sigma ,T}\left(
-\xi \right) =|\mathcal{U}_{\sigma ,T}\left( \xi \right) \varphi )(\mathcal{U%
}_{\sigma ,T}\left( \xi \right) \psi |=\left( \mathcal{U}_{\sigma ,T}\left(
\xi \right) \psi ,\cdot \right) \mathcal{U}_{\sigma ,T}\left( \xi \right)
\varphi
\end{equation*}%
and 
\begin{eqnarray*}
\left( u,1\left\{ \varphi )(\psi |\right\} v\right) &=&\int_{W}\left( u,%
\mathcal{U}_{\sigma ,T}\left( \xi \right) \varphi \right) \left( \mathcal{U}%
_{\sigma ,T}\left( \xi \right) \psi ,v\right) \mathrm{d}^{\sigma }\xi \\
&=&\left( \det \left( T+T^{\sigma }\right) \right) ^{\frac{1}{2}}\left( \psi
,\varphi \right) \left( u,v\right) \\
&=&\left( u,\left( \det \left( T+T^{\sigma }\right) \right) ^{\frac{1}{2}}%
\mathrm{Tr}\left( |\varphi )(\psi |\right) v\right) .
\end{eqnarray*}%
So the equality holds for operators of rank $1$. Next we extend this
equality by linearty and continuity.

$\left( \mathrm{iii}\right) $ We have 
\begin{equation*}
\mathcal{U}_{\sigma ,T}\left( \xi \right) G\mathcal{U}_{\sigma ,T}\left(
-\xi \right) \ll \left( \mathcal{U}_{\sigma ,T}\left( \xi \right) \left\vert
G^{\ast }\right\vert \mathcal{U}_{\sigma ,T}\left( -\xi \right) ;\mathcal{U}%
_{\sigma ,T}\left( \xi \right) \left\vert G\right\vert \mathcal{U}_{\sigma
,T}\left( -\xi \right) \right)
\end{equation*}%
which gives 
\begin{align*}
\left\vert b_{1}\left( \xi \right) b_{2}\left( \xi \right) \left( \varphi ,%
\mathcal{U}_{\sigma ,T}\left( \xi \right) G\mathcal{U}_{\sigma ,T}\left(
-\xi \right) \psi \right) \right\vert & \leq \left( \left\vert b_{1}\left(
\xi \right) \right\vert ^{2}\left( \varphi ,\mathcal{U}_{\sigma ,T}\left(
\xi \right) \left\vert G^{\ast }\right\vert \mathcal{U}_{\sigma ,T}\left(
-\xi \right) \varphi \right) \right) ^{\frac{1}{2}} \\
& \cdot \left( \left\vert b_{2}\left( \xi \right) \right\vert ^{2}\left(
\psi ,\mathcal{U}_{\sigma ,T}\left( \xi \right) \left\vert G\right\vert 
\mathcal{U}_{\sigma ,T}\left( -\xi \right) \psi \right) \right) ^{\frac{1}{2}%
}.
\end{align*}%
Next, all that remains is to use Schwarz inequality to conclude that $\left( 
\mathrm{iii}\right) $ is true.
\end{proof}

\section{Schatten-class properties of operators in the $T$-Weyl calculus.
The Cordes-Kato method.}

Now, if $\left( \mathcal{H},\mathcal{W}_{\sigma ,T},\omega _{\sigma
,T}\right) $ is an irreducible $\omega _{\sigma ,T}$-representation of $W$,
then we are able to study boundedness and Schatten-class properties of
certain operators in the $T$-Weyl calculus using the Cordes-Kato method.

\begin{theorem}
Let $\left( \mathcal{H},\mathcal{W}_{\sigma ,T},\omega _{\sigma ,T}\right) $
be an irreducible $\omega _{\sigma ,T}$-representation of $W$, and assume
that $W=V_{1}\oplus ...\oplus V_{\ell }$ is an orthogonal decomposition with
respect to a $\sigma $-compatible inner product on $\left( W,\sigma \right) $%
. Let $a\in \mathcal{S}^{\prime }(W)$.

$\left( \mathrm{a}\right) $ If there are $t_{1}>n_{1}=\dim V_{1},...,$ $%
t_{\ell }>n_{\ell }=\dim V_{\ell }$ such that 
\begin{equation*}
b=\left( 1-\triangle _{V_{1}}\right) ^{t_{1}/2}\otimes ...\otimes \left(
1-\triangle _{V_{\ell }}\right) ^{t_{\ell }/2}a\in L^{\infty }\left(
W\right) ,
\end{equation*}%
then $\mathrm{Op}_{\sigma ,T}\left( a\right) \in \mathcal{B}\left( \mathcal{H%
}\right) $ and 
\begin{equation*}
\left\Vert \mathrm{Op}_{\sigma ,T}\left( a\right) \right\Vert _{\mathcal{B}%
\left( \mathcal{H}\right) }\leq Cst\left\Vert \left( 1-\triangle
_{V_{1}}\right) ^{t_{1}/2}\otimes ...\otimes \left( 1-\triangle _{V_{\ell
}}\right) ^{t_{\ell }/2}a\right\Vert _{L^{\infty }\left( W\right) }.
\end{equation*}

$\left( \mathrm{b}\right) $ Let $1\leq p<\infty $. If there are $%
t_{1}>n_{1}=\dim V_{1},...,$ $t_{\ell }>n_{\ell }=\dim V_{\ell }$ such that 
\begin{equation*}
b=\left( 1-\triangle _{V_{1}}\right) ^{t_{1}/2}\otimes ...\otimes \left(
1-\triangle _{V_{\ell }}\right) ^{t_{\ell }/2}a\in L^{p}\left( W\right) ,
\end{equation*}%
then $\mathrm{Op}_{\sigma ,T}\left( a\right) \in \mathcal{B}_{p}\left( 
\mathcal{H}\right) $ and 
\begin{equation*}
\left\Vert \mathrm{Op}_{\sigma ,T}\left( a\right) \right\Vert _{\mathcal{B}%
_{p}\left( \mathcal{H}\right) }\leq Cst\left\Vert \left( 1-\triangle
_{V_{1}}\right) ^{t_{1}/2}\otimes ...\otimes \left( 1-\triangle _{V_{\ell
}}\right) ^{t_{\ell }/2}a\right\Vert _{L^{p}\left( W\right) }.
\end{equation*}
\end{theorem}

\begin{proof}
For $j\in \left\{ 1,...,\ell \right\} $, let $\psi _{t_{j}}=\psi
_{t_{j}}^{V_{j}}$ be the unique solution within $\mathcal{S}^{\prime }\left(
V_{j}\right) $ for 
\begin{equation*}
\left( 1-\triangle _{V_{1}}\right) ^{t_{1}/2}\psi _{t_{j}}=\delta .
\end{equation*}%
where $\triangle _{V_{j}}$ is the Laplacian in $V_{j}$, and $\delta \in 
\mathcal{S}^{\prime }\left( V_{j}\right) $ is the delta function in $V_{j}$.
We consider 
\begin{equation*}
g=\psi _{t_{1}}\otimes ...\otimes \mathcal{\psi }_{{t}_{{\ell }}}\in 
\mathcal{S}^{\prime }(W).
\end{equation*}%
Then $g\in L^{1}\left( W\right) $ and $a=b\ast _{\sigma }g$ because%
\begin{equation*}
\left( 1-\triangle _{V_{1}}\right) ^{t_{1}/2}\otimes ...\otimes \left(
1-\triangle _{V_{\ell }}\right) ^{t_{\ell }/2}g=\delta \in \mathcal{S}%
^{\prime }(W).
\end{equation*}%
By Cordes' lemma, Corollary \ref{n13}, 
\begin{equation*}
G=\mathrm{Op}_{\sigma ,T}\left( g\right) \in \mathcal{B}_{1}\left( \mathcal{H%
}\right) ,
\end{equation*}%
and by Kato's identity, Theorem \ref{n14}, 
\begin{equation*}
\mathrm{Op}_{\sigma ,T}\left( a\right) =\mathrm{Op}_{\sigma ,T}\left( b\ast
_{\sigma }g\right) =\int_{W}b\left( \xi \right) \mathcal{U}_{\sigma
,T}\left( \xi \right) \mathrm{Op}_{\sigma ,T}\left( g\right) \mathcal{U}%
_{\sigma ,T}\left( -\xi \right) \mathrm{d}^{\sigma }\xi =b\left\{ G\right\} .
\end{equation*}%
The proof is established upon employing Theorem \ref{n15}.
\end{proof}

Recall that the Sobolev space $H_{p}^{s}\left( W\right) $, $s\in \mathbb{R}$%
, $1\leq p\leq \infty $, consists of all $a\in \mathcal{S}^{\ast }(W)$ such
that $\left( 1-\triangle _{W}\right) ^{s/2}a\in L^{p}\left( W\right) $, and
we set 
\begin{equation*}
\left\Vert a\right\Vert _{H_{p}^{s}\left( W\right) }\equiv \left\Vert \left(
1-\triangle _{W}\right) ^{s/2}a\right\Vert _{L^{p}\left( W\right) }.
\end{equation*}

\begin{theorem}
Let $\left( \mathcal{H},\mathcal{W}_{\sigma ,T},\omega _{\sigma ,T}\right) $
be an irreducible $\omega _{\sigma ,T}$-representation of $W$.

$\left( \mathrm{a}\right) $ If $s>2n$ and $a\in H_{\infty }^{s}\left(
W\right) $, then $\mathrm{Op}_{\sigma ,T}\left( a\right) \in \mathcal{B}%
\left( \mathcal{H}\right) $ and%
\begin{equation*}
\left\Vert \mathrm{Op}_{\sigma ,T}\left( a\right) \right\Vert _{\mathcal{B}%
\left( \mathcal{H}\right) }\leq Cst\left\Vert a\right\Vert _{H_{\infty
}^{s}\left( W\right) }.
\end{equation*}

$\left( \mathrm{b}\right) $ If $1\leq p<\infty $, $s>2n$ and $a\in
H_{p}^{s}\left( W\right) $, then $\mathrm{Op}_{\sigma ,T}\left( a\right) \in 
\mathcal{B}_{p}\left( \mathcal{H}\right) $ and%
\begin{equation*}
\left\Vert \mathrm{Op}_{\sigma ,T}\left( a\right) \right\Vert _{\mathcal{B}%
_{p}\left( \mathcal{H}\right) }\leq Cst\left\Vert a\right\Vert
_{H_{p}^{s}\left( W\right) }.
\end{equation*}
\end{theorem}

If we note that $\mathrm{Op}_{\sigma ,T}\left( a\right) \in \mathcal{B}%
_{2}\left( \mathcal{H}\right) $ whenever $a\in L^{2}\left( W\right)
=H_{2}^{0}\left( W\right) $ for any irreducible $\omega _{\sigma ,T}$%
-representation $\left( \mathcal{H},\mathcal{W}_{\sigma ,T},\omega _{\sigma
,T}\right) $ of $W$, then standard interpolation results in Sobolev spaces
give us the following result.

\begin{theorem}
Let $\left( \mathcal{H},\mathcal{W}_{\sigma ,T},\omega _{\sigma ,T}\right) $
be an irreducible $\omega _{\sigma ,T}$-representation of $W$. Let $\mu >1$
and $1\leq p<\infty $. If $a\in H_{p}^{2\mu n\left\vert 1-2/p\right\vert
}\left( W\right) $, then $\mathrm{Op}_{\sigma ,T}\left( a\right) \in 
\mathcal{B}_{p}\left( \mathcal{H}\right) $ and 
\begin{equation*}
\left\Vert \mathrm{Op}_{\sigma ,T}\left( a\right) \right\Vert _{\mathcal{B}%
_{p}\left( \mathcal{H}\right) }\leq Cst\left\Vert a\right\Vert _{H_{p}^{2\mu
n\left\vert 1-2/p\right\vert }\left( W\right) }.
\end{equation*}
\end{theorem}

\end{document}